\documentclass[final]{siamart1116}
\usepackage{amssymb}
\usepackage{stmaryrd}
\usepackage{mathrsfs}
\usepackage{psfrag}

\newcommand{\esf}{\mathsf{e}}
\newcommand{\ysf}{\boldsymbol{\mathsf{y}}}
\newcommand{\psf}{\mathsf{p}}
\newcommand{\wsf}{\boldsymbol{\mathsf{w}}}
\newcommand{\qsf}{\mathsf{q}}
\newcommand{\usf}{\boldsymbol{\mathsf{u}}}
\newcommand{\fsf}{\boldsymbol{\mathsf{f}}}
\newcommand{\csf}{\boldsymbol{\mathsf{c}}}
\newcommand{\vsf}{\boldsymbol{\mathsf{v}}}
\newcommand{\aasf}{\boldsymbol{\mathsf{a}}}
\newcommand{\bbsf}{\boldsymbol{\mathsf{b}}}
\newcommand{\xisf}{\boldsymbol{\mathsf{\xi}}}
\newcommand{\zetasf}{\boldsymbol{\mathsf{\zeta}}}
\newcommand{\T}{\mathscr{T}}
\newcommand{\norm}[1]{{{|\!|\!|} #1 {|\!|\!|}}}
\newcommand{\normv}[1]{{{|\!|\!|} #1 {|\!|\!|}}}
\newcommand{\normp}[1]{{{|\!|\!|} #1 {|\!|\!|}}}
\newcommand{\bia}{\mathcal{A}}
\newcommand{\bib}{\mathcal{B}}
\newcommand{\bic}{\mathcal{C}}
\newcommand{\bis}{\mathcal{S}}

\newsiamremark{expl}{Example}
\newtheorem{example}{Example}
\newsiamremark{remark}{Remark}
\numberwithin{equation}{section}
\numberwithin{theorem}{section}

\headers{Error estimation for PDE--constrained optimization}{A. Allendes, E. Ot\'arola, and R. Rankin}

\title{A posteriori error estimation for finite element approximations of a PDE--constrained optimization problem in fluid dynamics\thanks{The first author's work was supported by USM grant 116.12.1. The second author's work was supported by CONICYT through FONDECYT project 3160201. 
}}

\author{Alejandro Allendes\thanks{Departamento de Matem\'atica, Universidad T\'ecnica Federico Santa Mar\'ia, Valpara\'iso, Chile.
\texttt{alejandro.allendes@usm.cl}}
\and
Enrique Ot\'arola\thanks{Departamento de Matem\'atica, Universidad T\'ecnica Federico Santa Mar\'ia, Valpara\'iso, Chile.
\texttt{enrique.otarola@usm.cl}}
\and
Richard Rankin\thanks{Departamento de Matem\'atica, Universidad T\'ecnica Federico Santa Mar\'ia, Valpara\'iso, Chile.
\texttt{richard.rankin@usm.cl}}
}

\date{Draft version of \today.}

\begin{document}

\maketitle

\begin{abstract}
We derive globally reliable a posteriori error estimators for a PDE--constrained optimization problem involving linear models in fluid dynamics as state equation; control constraints are also considered. The corresponding local error indicators are locally efficient. The assumptions under which we perform the analysis are such that they can be satisfied for a wide variety of stabilized finite element methods as well as for standard finite element methods. When stabilized methods are considered, no a priori relation between the stabilization terms for the state and adjoint equations is required. If a lower bound for the inf--sup constant is available, a posteriori error estimators that are fully computable and provide guaranteed upper bounds on the norm of the error can be obtained. We illustrate the theory with numerical examples.
\end{abstract}




\section{Introduction}

In this work we shall be interested in the design and analysis of computable a posteriori error estimators for a linear--quadratic optimal control problem involving linear models in fluid dynamics as state equation; control constraints are considered. To make matters precise, let $\Omega\subset\mathbb{R}^{d}$, with $d \in \{2,3\}$, be an open and bounded polytopal domain with Lipschitz boundary $\partial \Omega$ and $\fsf\in L^{2}(\Omega)^{d}$. Given a regularization parameter $\vartheta > 0$ and a desired state $\ysf_{\Omega}\in L^{2}(\Omega)^{d}$, we define
\begin{equation}\label{cost_functional}
J(\ysf,\usf) = \frac{1}{2} \|\ysf-\ysf_{\Omega}\|_{L^{2}(\Omega)^{d}}^{2} + \frac{\vartheta}{2} \|\usf\|_{L^{2}(\Omega)^{d}}^{2}.
\end{equation}
We will be interested in the following PDE--constrained optimization problem: Find
\begin{equation}
 \min J(\ysf,\usf)
 \label{eq:min}
\end{equation}
subject to the \emph{generalized Oseen equations}
\begin{equation}\label{cont_optimal_control}
\left\{
\begin{array}{rcll}
-\varepsilon\Delta\ysf + \left(\csf \cdot\nabla\right)\ysf + \kappa\ysf +\nabla \psf & = & \fsf + \usf & \textrm{in}~\Omega,\\
\nabla\cdot\ysf & = & 0 & \textrm{in}~\Omega,\\
\ysf & = & \mathbf{0} & \textrm{on}~\partial\Omega,\\
\end{array}
\right.
\end{equation}
and the control constraints
\begin{equation}\label{constraint}
\aasf\leq \usf \leq \bbsf \quad\textrm{a.e. in } \Omega,
\end{equation}
with $\aasf,\bbsf\in\mathbb{R}^{d}$ satisfying $\aasf<\bbsf$; the previous vector inequalities being understood componentwise. In \eqref{cont_optimal_control}, $\varepsilon,\kappa\in\mathbb{R}$ and are such that $\varepsilon > 0$ and $\kappa \geq 0$ and $\csf \in \mathbf{W}^{1,\infty}(\Omega)$ is a solenoidal field. The generalized Oseen equations describe the low--Reynolds--number flow in porous media in situations where velocity gradients are non--negligible; they provide a unified approach to model flows of viscous fluids in a cavity and a porous media. Our analysis allows for choices of the terms $\csf$ and $\kappa$ that yield different flow models:
\begin{equation}\label{differential_D}
\left\{
\begin{array}{cl}
\csf = \mathbf{0}, \kappa = 0&: -\varepsilon\Delta\ysf +\nabla \psf \quad (\textrm{Stokes}), \\
\csf = \mathbf{0}&: -\varepsilon\Delta\ysf + \kappa\ysf +\nabla \psf  \quad (\textrm{Brinkman}), \\
\kappa = 0&: -\varepsilon\Delta\ysf + \left(\csf\cdot\nabla\right)\ysf+\nabla \psf \quad (\textrm{Oseen}). \\
\end{array}
\right.
\end{equation}

The design of numerical techniques for approximating the solution to \eqref{cont_optimal_control} has two major difficulties: first, in view of the so--called inf--sup condition \cite{MR2050138,MR851383}, arbitrary finite element methods are not allowed, and second, considering standard finite element methods produces poor approximation results when convection--dominated regimes are considered \cite{MR2454024}. In order to overcome such difficulties, a variety of finite element techniques have been proposed and analyzed in the literature: the family of \textit{stabilized finite element methods.} We refer the reader to \cite{MR2454024} for an extensive overview.

In the PDE--constrained optimization context, a usual alternative for approximating the solution to the optimal control problem \eqref{eq:min}--\eqref{constraint} is based on the so-called \emph{optimize--then--discretize} approach. This technique discretizes the associated optimality system: the state equations \eqref{cont_optimal_control}, the adjoint equations and a variational inequality that characterizes the optimal control $\bar \usf$. Consequently, the difficulties presented in the discretization of \eqref{cont_optimal_control} are also present in the numerical approximation of the solution to \eqref{eq:min}--\eqref{constraint}. In addition, \eqref{eq:min}--\eqref{constraint} is intrinsically nonlinear and, if $\csf \neq \mathbf{0}$, presents a crosswind phenomena; the convection field of the adjoint equations is the negative of the one appearing in \eqref{cont_optimal_control}. The latter further motives the development of an efficient solution technique that, in convection--dominated regimes, properly treats the oscillatory behaviors that occur when approximating $\bar \ysf$ and its adjoint variable $\bar \wsf$ and resolves interior or boundary layers exhibited by both variables. Failure to resolve boundary layers can pollute the numerical solution in the entire domain; see \cite{MR2595051} for results involving the scalar version of  \eqref{eq:min}--\eqref{constraint}. However, numerical schemes based only on stabilized techniques are not sufficient to approximate the solution to \eqref{eq:min}--\eqref{constraint}: in addition to the efficient resolution of either interior or boundary layers, some possible geometric singularities must be resolved. This motivates the methods that we will use in this work: \emph{stabilized adaptive finite element methods.}

Adaptive finite element methods (AFEMs) are iterative methods that improve the quality of the finite element approximation to a partial differential equation (PDE) on the basis of an essential ingredient: an a posteriori error estimator. The a posteriori error analysis for the standard finite element approximation of elliptic problems has a solid foundation \cite{AObook,MR2648380,MR3059294}. When stabilized approximations are considered, several estimators have been introduced and analyzed in the literature; see, for instance, \cite{MR3057330,MR3556402,MR2398346,MR3407239,MR2871298}. However, in the PDE--constrained optimization context, the theory has not been fully developed. The main source of difficulty is its inherent nonlinear feature, which appears due to the control constraints. An attempt to unify the available results has been carried out recently in \cite{MR3212590} where the authors derive an important relationship between the error in optimal control problems and estimators, that satisfy a set of suitable assumptions, for problems associated with the state and adjoint equations  \cite[Theorem 3.2]{MR3212590}.

In the current work, the assumptions under which we perform the analysis are such that they can be satisfied for a wide variety of stabilized finite element methods as well as for standard finite element methods. This includes using a different stabilization method to approximate the state equation from that used to approximate the adjoint equation. We derive a posteriori error estimators that are globally reliable. Moreover, if a lower bound for the inf--sup constant is available, we can obtain a posteriori error estimators that are fully computable and provide guaranteed upper bounds on the norm of the error. Consequently, the estimators can be used as a stopping criterion in adaptive algorithms. The local error indicators that can be used to adaptively refine the mesh are locally efficient. Furthermore, we observe that they can be used to efficiently resolve boundary layers.

The outline of this paper is as follows. In section \ref{primeraaa} we introduce some terminology used throughout this work. In section \ref{control_optimo} we study the optimal control problem \eqref{eq:min}--\eqref{constraint} and obtain the associated optimality system. In section \ref{FEM_optimal_control} we give the general form of the finite element methods that we consider for approximating the solution to \eqref{eq:min}--\eqref{constraint}. The core of our work is section \ref{A_posteriori}, where we devise a family of a posteriori error estimators. Under suitable assumptions, we obtain abstract reliability results in section \ref{reliability} and local efficiency of the corresponding error indicators in section \ref{efficiency}. In section \ref{particular} we consider the estimators that we can obtain for a particular approximation method in more detail. Finally, in section \ref{sec:numex} we present a series of numerical examples to illustrate the theory.


\section{Preliminaries}\label{primeraaa}
\subsection{Notation}
For a bounded domain $A \subset \mathbb{R}^t$, $t \in \{1,2,3\}$, $L^{2}(A)$ and $H^{1}(A)$ denote the standard Lebesgue and Sobolev spaces, respectively; $L_{0}^{2}(A)$ is the subspace of $L^{2}(A)$ containing functions with zero mean value on $A$, and $H_{0}^{1}(A)$ is the subspace of $H^{1}(A)$ containing functions whose trace is zero on $\partial A$. We use bold letters to denote the vector--valued counterparts of the aforementioned spaces and an extra under accent for their matrix--valued counterparts. For instance, for $d \in \{ 2,3\}$, we denote $\mathbf{L}^{2}(A)=L^{2}(A)^d$ and $\underset{\approx}{\boldsymbol{L}}^{2}(A)=L^{2}(A)^{d\times d}$.

We now proceed to define notation associated with the discretization of the domain. Let $\T = \{ K \}$ be a conforming partition of $\bar \Omega$ into simplical elements $K$ \cite{MR1930132,MR2050138}. We assume that $\T$ is a member of a shape regular family of partitions. Let $\mathcal{F}$ denote the set of all element edges(2D)/faces(3D) and $\mathcal{F}_{I}\subset \mathcal{F}$ denote the set of interior edges(2D)/faces(3D).

\noindent For an element $K\in\T$, let:
\begin{itemize}
\item $\mathbb{P}_{n}(K)$ denote the space of polynomials on $K$ of total degree at most $n$;
\item $\mathcal{F}_{K}\subset \mathcal{F}$ denote the set containing the individual edges(2D)/faces(3D) of $K$;
\item $h_{K}$ denote the diameter of $K$;
\item $\boldsymbol{n}_{\gamma}^{K}$ denote the unit exterior normal vector to the edge(2D)/face(3D) $\gamma\in\mathcal{F}_{K}$.
\end{itemize}
For an edge(2D)/face(3D) $\gamma\in\mathcal{F}$, let:
\begin{itemize}
\item $\mathbb{P}_{n}(\gamma)$ denote the space of polynomials on $\gamma$ of total degree at most $n$;
\item $\Omega_{\gamma}=\{K\in\T:~\gamma\in\mathcal{F}_{K}\}$;
\item $h_{\gamma}$ denote the diameter of the edge(2D)/face(3D) $\gamma$.
\end{itemize}

To simplify the exposition of the material, we define $\mathbf{V} = \boldsymbol{H}_{0}^{1}(\Omega)$ and $Q = L_0^2(\Omega)$ with norms $\norm{\cdot}_{\mathbf{V},\Omega}$ and $\norm{\cdot}_{Q,\Omega}$ defined, for all $\xisf \in \mathbf{V}$ and $\phi \in Q$, by
\begin{equation}
 \label{eq:energy_norm_intro}
 \norm{\xisf}_{\mathbf{V},\Omega}^2:=\sum_{K\in\T}\norm{\xisf}_{\mathbf{V},K}^2
\mbox{ and }
\normp{\phi}_{Q,\Omega}^{2}:=\sum_{K\in\T}\normp{\phi}_{Q,K}^{2}
\end{equation}
where
\begin{equation}
\label{eq:local_energy_norm}
 \norm{\xisf}_{\mathbf{V},K}^2:= \varepsilon \| \nabla \xisf \|^2_{\underset{\approx}{\boldsymbol{L}}^{2}(K)} + \kappa \|\xisf \|^2_{\boldsymbol{L}^{2}(K)}\mbox{ and }\normp{\phi}_{Q,K}^{2}:=\|\phi\|_{L^{2}(K)}^{2}.
\end{equation}

The relation $a \lesssim b$ indicates that there exists a constant $C$ such that $a \leq C b$. The constant $C$ may be different at each occurence but is independent of $a$, $b$ and the size of the elements in the mesh.

\subsection{Inequalities}
For $K\in \T$ and nonnegative integers $l$, we denote by $\Pi_{K,l}$ the $\boldsymbol{L}^{2}(K)$--orthogonal projection operator onto $\mathbb{P}_{l}(K)^{d}$. This operator is defined as
\begin{equation}\label{projectionL2K}
\Pi_{K,l}:\boldsymbol{L}^{2}(K)\rightarrow\mathbb{P}_{l}(K)^{d}, \qquad
\left(\mathbf{t}-\Pi_{K,l}(\mathbf{t}),\mathbf{v}\right)_{\boldsymbol{L}^{2}(K)}=0\quad\forall \mathbf{v}\in\mathbb{P}_{l}(K)^{d}.
\end{equation}

Throughout the manuscript we will frequently make use of the following inequalities. First, if  $K \in \T$ and $\xisf\in \mathbf{V}$, we have the Poincar\'e inequalities \cite{MR2036927,MR853915,MR0117419}
\begin{equation}\label{Poincare}
\|\xisf\|_{\boldsymbol{L}^{2}(\Omega)} 
 \leq 
\mathtt{C}_{P,\Omega}\|\nabla\xisf\|_{\underset{\approx}{\boldsymbol{L}}^{2}(\Omega)}
\mbox{ and }
\|\xisf-\Pi_{K,0}(\xisf)\|_{\boldsymbol{L}^{2}(K)} 
\leq  
\dfrac{h_K}{\pi}
\|\nabla\xisf\|_{\underset{\approx}{\boldsymbol{L}}^{2}(K)},
\end{equation}
where
\begin{equation}
\mathtt{C}_{P,\Omega}=\frac{1}{\pi}\left(\sum_{i=1}^d\frac{1}{\left|l_i\right|^2}\right)^{-1/2}
\end{equation}
with $\left|l_1\right|,\ldots,\left|l_d\right|$ being the sides of a $d$-dimensional box containing $\Omega$. We immediately comment that these inequalities imply that, for $\xisf \in \mathbf{V}$ and $K \in \T$,
\begin{equation}\label{C_Omega}
\|\xisf\|_{\boldsymbol{L}^{2}(\Omega)} \leq \texttt{C}_{\Omega}
\norm{\xisf}_{\mathbf{V},\Omega}
\mbox{ and }
\|\xisf-\Pi_{K,0}(\xisf)\|_{\boldsymbol{L}^{2}(K)} \leq \texttt{C}_{K}
\norm{\xisf}_{\mathbf{V},K},
\end{equation}
where
\begin{equation}\label{CO}
\texttt{C}_{\Omega}=\left\{
\begin{array}{ll}
\frac{\texttt{C}_{P,\Omega}}{\sqrt{\varepsilon}}, & \textrm{if}~\kappa=0, \\
\min\left\{\frac{\texttt{C}_{P,\Omega}}{\sqrt{\varepsilon}},\frac{1}{\sqrt{\kappa}}\right\}, & \textrm{if}~\kappa\neq 0,
\end{array}
\right.
\end{equation}
and
\begin{equation}\label{CK}
\texttt{C}_{K}=\left\{
\begin{array}{ll}
\frac{h_K}{\pi\sqrt{\varepsilon}}, & \textrm{if}~\kappa=0, \\
\min\left\{\frac{h_K}{\pi\sqrt{\varepsilon}},\frac{1}{\sqrt{\kappa}}\right\}, & \textrm{if}~\kappa\neq 0.
\end{array}
\right.
\end{equation}

We define $\bia: \mathbf{V} \times \mathbf{V} \rightarrow \mathbb{R}$, $\bib: \mathbf{V} \times Q \rightarrow \mathbb{R}$ and $\bic: \mathbf{V} \times \mathbf{V} \rightarrow \mathbb{R}$ by
\begin{equation}\label{bilinear_form_B}
\left\{
\begin{array}{l}
\bia(\xisf,\zetasf):=
\varepsilon(\nabla\xisf,\nabla\zetasf)_{\underset{\approx}{\boldsymbol{L}}^{2}(\Omega)} + (\kappa\xisf + \left(\csf \cdot\nabla\right)\xisf ,\zetasf)_{\boldsymbol{L}^{2}(\Omega)},
\\
\bib(\zetasf,\phi):=(\phi,\nabla\cdot\zetasf)_{L^{2}(\Omega)},
\vspace{0.1cm}
\\
\bic(\xisf,\zetasf):= \varepsilon(\nabla \xisf,\nabla\zetasf)_{\underset{\approx}{\boldsymbol{L}}^{2}(\Omega)} +
(\kappa \xisf-\left(\csf \cdot\nabla\right)\xisf,\zetasf)_{\boldsymbol{L}^{2}(\Omega)}.
\end{array}
\right.
\end{equation}
The fact that $\csf$ is a solenoidal vector field and integration by parts implies that
\begin{equation}\label{ac_link}
\bia(\xisf,\zetasf)=\bic(\zetasf,\xisf) \quad \forall \xisf,\zetasf \in \mathbf{V}.
\end{equation}
Moreover, for all $\xisf \in \mathbf{V}$,
\begin{equation}
 \label{eq:a_coercive}
\bia(\xisf,\xisf)=\bic(\xisf,\xisf) = \norm{\xisf}_{\mathbf{V},\Omega}^2
\end{equation}
and, for all $\xisf,\zetasf \in \mathbf{V}$,
\begin{equation}
 \label{eq:a_continuous}
\bia(\xisf,\zetasf) \leq \mathtt{C}_{\textsf{ct}} \norm{\xisf}_{\mathbf{V},\Omega}\norm{\zetasf}_{\mathbf{V},\Omega},
\qquad
\bic(\xisf,\zetasf) \leq \mathtt{C}_{\textsf{ct}} \norm{\xisf}_{\mathbf{V},\Omega}\norm{\zetasf}_{\mathbf{V},\Omega},
\end{equation}
where
\begin{equation}
 \label{eq:Cct}
\mathtt{C}_{\textsf{ct}}=
1 + \frac{\mathtt{C}_{\Omega}}{\sqrt{\varepsilon}} \| |\csf| \|_{L^{\infty}(\Omega)},
\end{equation}
with $\| |\csf| \|_{L^{\infty}(\Omega)}$ being the $L^{\infty}(\Omega)$ norm of $|\csf|$ and $\mathtt{C}_{\Omega}$ being given by \eqref{CO}.

We now recall the standard inf--sup condition \cite{MR2050138,MR851383}: there exists a positive constant $\beta$ such that
\begin{equation}\label{inf-sup}
\beta\|\phi\|_{L^{2}(\Omega)} \leq \sup_{\xisf\in\mathbf{V}\setminus\{\boldsymbol{0}\}}
\frac{\bib(\xisf,\phi)}{\|\nabla\xisf\|_{\underset{\approx}{\boldsymbol{L}}^{2}(\Omega)}}
\quad \forall \phi \in Q.
\end{equation}
Notice that, in view of $\norm{\xisf}_{\mathbf{V},\Omega}^2 \leq (\varepsilon + \kappa \texttt{C}_{P,\Omega}^2) \| \nabla \xisf\|^2_{\underset{\approx}{\boldsymbol{L}}^{2}(\Omega)}$, we have that
\begin{equation}\label{inf-sup_2}
\norm{\phi}_{Q,\Omega}\leq \texttt{C}_{\textsf{is}}\sup_{\xisf\in \mathbf{V} \setminus\{\boldsymbol{0}\}}
\frac{\bib(\xisf,\phi)}{\norm{\xisf}_{\mathbf{V},\Omega}} \quad \forall \phi \in Q,
\end{equation}
where
\begin{equation}\label{Cis}
\texttt{C}_{\textsf{is}}=
\frac{\sqrt{\varepsilon+\kappa\mathtt{C}_{P,\Omega}^{2}}}{\beta}.
\end{equation}


\section{Optimal control problem: optimize}\label{control_optimo}
In this section we briefly analyze the optimal control problem \eqref{eq:min}--\eqref{constraint}. To accomplish this task, we begin by introducing the following weak version of the state equations \eqref{cont_optimal_control}: Find $(\ysf,\psf) \in \mathbf{V} \times Q$ such that
\begin{equation}\label{state_constraint}
\left\{
\begin{array}{rcll}
\bia(\ysf,\xisf) - \bib(\xisf,\psf) & = & (\fsf+\usf,\xisf)_{\boldsymbol{L}^{2}(\Omega)}
& \forall~\xisf\in \mathbf{V},\\
\bib(\ysf,\phi) & = & 0
& \forall~\phi\in Q,
\end{array}
\right.
\end{equation}
where the bilinear forms $\bia$ and $\bib$ are defined by \eqref{bilinear_form_B} and we recall that $\varepsilon > 0$, $\kappa \geq 0$, $\csf \in \mathbf{W}^{1,\infty}(\Omega)$ is a solenoidal field, $\fsf \in \mathbf{L}^2(\Omega)$, $\mathbf{V} = \boldsymbol{H}_{0}^{1}(\Omega)$ and $Q=\boldsymbol{L}_0^2(\Omega)$. In view of the fact that $\bia$ satisfies \eqref{eq:a_coercive} and \eqref{eq:a_continuous} and $\bib$ satisfies the inf--sup conditions \eqref{inf-sup} and \eqref{inf-sup_2}, we conclude the well--posedness of problem \eqref{state_constraint} \cite{MR2050138,MR851383}.
We also mention that, due to de Rham's Theorem (see Section 4.1.3 and Theorem B73 in \cite{MR2050138}), we can consider the following equivalent formulation of problem \eqref{state_constraint}: Find $\ysf \in \mathbf{V}_0$ such that
\begin{equation}\label{optimal_control_constrained}
\bia(\ysf,\xisf)  =  (\fsf+\usf,\xisf)_{\boldsymbol{L}^{2}(\Omega)}
\quad \forall\xisf\in \mathbf{V}_0,
\end{equation}
where $\mathbf{V}_0 :=\{\vsf\in\boldsymbol{H}_{0}^{1}(\Omega):~\nabla\cdot\vsf=0\}$. 

To analyze our optimal control problem, we follow \cite{MR0271512,MR2583281} and introduce the so--called control to state map $S:\boldsymbol{L}^{2}(\Omega)\rightarrow \mathbf{V}_0$ which,  given a control $\usf$, associates to it the state $\ysf$ that solves \eqref{optimal_control_constrained}. In addition, we define, for $\mathbf{a}, \mathbf{b} \in \mathbb{R}^d$ with $\mathbf{a} < \mathbf{b}$, the set
\begin{equation}\label{eq:Uad}
\mathbf{U}_{ad}:=\{\vsf\in \mathbf{L}^{2}(\Omega):~
\aasf\leq \vsf \leq \bbsf \quad\textrm{a.e. in } \Omega \};
\end{equation}
the vector inequalities being understood componentwise. The set $\mathbf{U}_{ad}$ is a bounded, convex, closed and nonempty subset of $\mathbf{L}^{2}(\Omega)$ and consequently weakly sequentially compact. Thus, in view of the fact that the reduced cost functional
\begin{equation*}
f(\usf):=\frac{1}{2}\|S(\usf)-\ysf_{\Omega}\|_{\boldsymbol{L}^{2}(\Omega)}^{2}+
\frac{\vartheta}{2}\|\usf\|_{\boldsymbol{L}^{2}(\Omega)}^{2}
\end{equation*}
is weakly lower semicontinuous and strictly convex $(\vartheta > 0)$, we conclude the existence and uniqueness of an optimal control $\bar \usf $ and an optimal state $\bar{\ysf}$ that satisfy \eqref{optimal_control_constrained}, or equivalently \eqref{state_constraint}; see Theorem 2.14 in \cite{MR2583281}. The existence of $\bar \psf$ such that $(\bar \ysf,\bar \psf)$ solves \eqref{state_constraint} follows from de Rham's Theorem. In addition, we have that $\bar \usf$ satisfies the first--order optimality condition
\begin{equation}
\label{variational_inequality}
f'(\bar{\usf})(\usf-\bar{\usf})\geq 0 \quad\forall~\usf\in\mathbf{U}_{ad};
\end{equation}
see \cite[Lemma 2.21]{MR2583281}. To explore this variational inequality, and to obtain optimality conditions, we define, on the basis of the formal Lagrange method (see \cite[Section 3.3]{MR3308473} and \cite[Section 2.10]{MR2583281}), the adjoint state $(\wsf,\qsf)$ as the unique solution to the following weak problem: Find $(\wsf,\qsf) \in \mathbf{V} \times Q$ such that
\begin{equation}
\label{adjoint_state}
\left\{
\begin{array}{rcll}
\bic(\wsf,\zetasf) + \bib(\zetasf,\qsf) & = & (\ysf - \ysf_{\Omega},\zetasf)_{\boldsymbol{L}^{2}(\Omega)}
& \forall~\zetasf\in \mathbf{V},\\
\bib(\wsf,\psi) & = & 0
& \forall~\psi\in Q.
\end{array}
\right.
\end{equation} 
With this adjoint state at hand, the variational inequality \eqref{variational_inequality} can be rewritten as
\begin{equation}
\label{variational_inequality2} 
(\bar \wsf+\vartheta \bar \usf,\usf- \bar \usf)_{\boldsymbol{L}^{2}(\Omega)} \geq  0 \quad \forall~\usf\in\mathbf{U}_{ad}.
\end{equation}
We have thus arrived at the following optimality system: $(\bar \ysf,\bar \psf,\bar \usf) \in \mathbf{V} \times Q \times \mathbf{U}_{ad}$ is optimal for the PDE--constrained optimization problem \eqref{eq:min}--\eqref{constraint} if and only if $(\bar \ysf,\bar \psf,\bar \wsf,\bar \qsf,\bar \usf) \in \mathbf{V} \times Q \times \mathbf{V} \times Q \times \mathbf{U}_{ad}$ solves
\begin{equation}\label{optimal_weak_system}
\left\{
\begin{array}{rcll}
\bia(\bar{\ysf},\xisf) -\bib(\xisf,\bar{\psf})  & = & (\fsf+\bar{\usf},\xisf)_{\boldsymbol{L}^{2}(\Omega)},
& \forall~\xisf\in \mathbf{V},\\
\bib(\bar{\ysf},\phi) & = & 0,
& \forall~\phi\in Q,\\
\bic(\bar{\wsf},\zetasf) +\bib(\zetasf,\bar{\qsf}) & = & (\bar{\ysf}-\ysf_{\Omega},\zetasf)_{\boldsymbol{L}^{2}(\Omega)},
& \forall~\zetasf\in \mathbf{V},\\
\bib(\bar{\wsf},\psi) & = & 0,
& \forall~\psi\in Q,\\
(\bar{\wsf}+\vartheta\bar{\usf},\usf-\bar{\usf})_{\boldsymbol{L}^{2}(\Omega)} & \geq & 0,
& \forall~\usf\in\mathbf{U}_{ad};
\end{array}
\right.
\end{equation}
see also \cite[Section 2]{MR2263034} and \cite[Section 2]{MR2803865} for similar results when the state equations \eqref{cont_optimal_control} are the Stokes equations.

We finally recall the projection formula for the optimal control variable: the variational inequality in \eqref{variational_inequality2} can be equivalently written as \cite[Chapter 2]{MR2583281}
\begin{equation}\label{projection_formula}
\bar{\usf}=\Pi_{[\aasf,\bbsf]}\left(-\frac{1}{\vartheta}\bar{\wsf}\right)
\quad\textrm{a.e. in }\Omega,
\end{equation}
where $\Pi_{[\aasf,\bbsf]}\left(\zetasf \right)(\boldsymbol{x}):=\min\left\{\bbsf,\max\left\{\aasf,\zetasf(\boldsymbol{x})\right\}\right\}$ and it is understood componentwise. We note that
\begin{equation}\label{eq:projproperty}
\left\| \Pi_{[\aasf,\bbsf]} (\xisf)-\Pi_{[\aasf,\bbsf]} (\zetasf)  \right\|_{\boldsymbol{L}^2(K)}
\leq
\left\|\xisf-\zetasf\right\|_{\boldsymbol{L}^2(K)} \quad \forall \xisf,\zetasf \in \mathbf{V}.
\end{equation}

\section{Finite element discretization}\label{FEM_optimal_control}
We follow the \emph{optimize--then--discretize} approach and introduce a numerical scheme to approximate the solution to \eqref{optimal_weak_system}. The scheme allows for the incorporation of stabilization terms into the standard Galerkin discretizations of the state and adjoint equations; no a priori relation between the stabilized terms is required. We refer the reader to Remark \ref{rk:otd} below for a discussion regarding the advantages of the proposed approach when solving \eqref{eq:min}--\eqref{constraint}.

The stabilized scheme reads as follows: Find $(\bar{\ysf}_{\T},\bar{\psf}_{\T},\bar{\wsf}_{\T},\bar{\qsf}_{\T},\bar{\usf}_{\T}) \in \mathbf{V}(\T) \times Q(\T) \times \mathbf{V}(\T) \times Q(\T) \times \mathbf{U}_{ad}(\T)$ such that
\begin{equation}\label{fem_control}
\left\{
\begin{array}{rcl}
\bia(\bar{\ysf}_{\T},\xisf) - \bib(\xisf,\bar{\psf}_{\T}) + \bis(\bar{\ysf}_{\T},\bar{\psf}_{\T},\fsf+\bar{\usf}_{\T};\xisf)
& = & (\fsf+\bar{\usf}_{\T},\xisf)_{\boldsymbol{L}^{2}(\Omega)},
\\
\bib(\bar{\ysf}_{\T},\phi) + \mathcal{H}(\bar{\ysf}_{\T},\bar{\psf}_{\T},\fsf+\bar{\usf}_{\T};\phi) 
& = & 0,
\\
\bic(\bar{\wsf}_{\T},\zetasf) + \bib(\zetasf,\bar{\qsf}_{\T})
+ \mathcal{Q}(\bar{\wsf}_{\T},\bar{\qsf}_{\T},\bar{\ysf}_{\T}-\ysf_{\Omega};\zetasf)
& = & (\bar{\ysf}_{\T}-\ysf_{\Omega},\zetasf)_{\boldsymbol{L}^{2}(\Omega)},
\\
\bib(\bar{\wsf}_{\T},\psi) + \mathcal{K}(\bar{\wsf}_{\T},\bar{\qsf}_{\T},\bar{\ysf}_{\T}-\ysf_{\Omega};\psi)  & = & 0,
\\
(\bar{\wsf}_{\T}+\vartheta\bar{\usf}_{\T},\usf-\bar{\usf}_{\T})_{\boldsymbol{L}^{2}(\Omega)} & \geq & 0,
\end{array}
\right.
\end{equation}
for all $(\xisf,\phi,\zetasf,\psi,\usf)\in \mathbf{V}(\T) \times Q(\T) \times \mathbf{V}(\T) \times Q(\T) \times \mathbf{U}_{ad}(\T)$; the bilinear forms $\bia, \bib$ and $\bic$ being defined as in \eqref{bilinear_form_B}. We consider the setting where the discrete spaces $\mathbf{V}(\T)$ and $Q(\T)$ are subspaces of $\mathbf{V}$ and $Q$, respectively, and the discrete set $\mathbf{U}_{ad}(\T)$ is a subset of $\mathbf{U}_{ad}$. Hence, $\mathbf{V}(\T)\subset \mathbf{V}$, $Q(\T)\subset Q$ and $\mathbf{U}_{ad}(\T)\subset \mathbf{U}_{ad}$. The terms $\bis$ and $\mathcal{H}$, and $\mathcal{Q}$ and $\mathcal{K}$ in \eqref{fem_control}, correspond to stabilization terms for the state and adjoint equations, respectively. Finally, we assume that $\mathbf{V}(\T)$, $Q(\T)$, $\mathbf{U}_{ad}(\T)$, $\bis$, $\mathcal{H}$, $\mathcal{Q}$ and $\mathcal{K}$ are such that at least one solution to \eqref{fem_control} exists.

\begin{remark}[optimize--then--discretize approach]\rm
In this work, we consider the \emph{optimize--then--discretize} approach because it allows for the incorporation of different stabilization terms into the discrete state and adjoint equations. The purpose of the latter is twofold: first, the use of low--order methods, and second, the efficient resolution of \eqref{fem_control}, by appropriately tuning some associated stabilization parameters, in convection--dominated regimes. The latter is especially important since, as is observed in \cite{MR2595051} for the scalar case, the failure to resolve boundary layers exhibited by the solution of \eqref{optimal_weak_system} can pollute the numerical solution in the entire domain. In contrast, the use of the \emph{discretize--then--optimize} approach imposes a relationship between the stabilization terms. To be precise, for a given stabilization terms $\bis$ in the state equations, the aforementioned approach imposes that the stabilization term $\mathcal{Q}$ is its adjoint counterpart \cite{collis2002analysis}. This could lead to an unnatural stabilization term in the adjoint equations delivering oscillatory solutions and therefore poor approximation results in convection--dominated regimes \cite{collis2002analysis}.
\label{rk:otd}
\end{remark}

Before proceeding with the analysis of our method, it is instructive to comment on those advocated in the literature. Regarding the a priori theory, in the absence of control constraints, the design and analysis of numerical techniques for solving \eqref{eq:min}--\eqref{cont_optimal_control}, with $\mathbf{c} = \mathbf{0}$ and $\kappa =0$, have been investigated in several papers; see \cite{MR2107381,MR2497339,MR3347461} and references therein. To the best of our knowledge, and again, for $\mathbf{c} = \mathbf{0}$ and $\kappa = 0$, the first work that incorporates control--constraints and analyzes stabilized schemes for \eqref{eq:min}--\eqref{constraint} is \cite{MR2263034}; the optimal control is discretized by using piecewise constant functions. The authors, on the basis of postprocessing techniques, provide a quadratic error estimate for the approximation of the optimal control variable \cite[Theorem 2.8]{MR2263034}. Subsequently, the authors of \cite{MR2803865} extend the results of \cite{MR2263034} and analyze nonconforming schemes for the discretization of the state and adjoint equations; in contrast to \cite{MR2263034}, the vector field is not assumed to be in $\mathbf{H}^2(\Omega) \cap \mathbf{W}^{1,\infty}(\Omega)$. In addition, \cite{MR2803865} analyzes an anisotropic scheme for approximating the solution to \eqref{eq:min}--\eqref{constraint} when $\Omega$ is not convex; a domain with a reentrant edge ($d=3$) is considered. We conclude this paragraph by mentioning the reference \cite{MR2472877}, where the authors investigate numerical techniques for solving a modification of problem \eqref{eq:min}--\eqref{constraint} that, in addition, includes constraints on the state variable.

Regarding the a posteriori error analysis, to the best of our knowledge, the first work to propose an error estimator for \eqref{eq:min}--\eqref{constraint}, with $\mathbf{c} = \mathbf{0}$ and $\kappa =0$, is \cite{MR1950625}. In this work, the authors follow the discretize--then--optimize approach and obtain a discrete optimality system with no stabilization terms \cite[equation (2.9)]{MR1950625}. They propose an error estimator in a two--dimensional setting and analyze its reliability properties \cite[Theorem 3.1]{MR1950625}. However, there is no efficiency analysis. Later, an asymptotically exact ZZ--type a posteriori error estimator was proposed in \cite{MR2387126}. The authors derive upper and lower bounds for the error in terms of the proposed estimator \cite[Theorem 5.1]{MR2387126} that relies on an error non--degeneracy condition \cite[inequality (2.24)]{MR2387126} and strong regularity assumptions on $(\bar \ysf,\bar \psf)$: it is assumed to belong to $\mathbf{H^3}(\Omega) \cap \mathbf{V} \times \mathbf{H}^1(\Omega) \cap Q $ \cite[Lemma 4.2]{MR2387126}. In \cite{MR2968744}, the authors propose an a posteriori error estimator for \eqref{eq:min}--\eqref{constraint} but with the state equations \eqref{cont_optimal_control} replaced by a Stokes-Darcy system: they study the reliability and efficiency properties of the proposed estimator. We also mention \cite{MR2912614}, where a similar PDE--constrained optimization problem has been analyzed but with the control--constraint \eqref{constraint} replaced by the state--constraint $\| \ysf \|_{\mathbf{L}^2(\Omega)} \leq \gamma$, where $\gamma >0$: an error estimator is proposed and its reliability and efficiency properties are investigated. All the aforementioned references consider plain Galerkin discretizations for the state and adjoint equations, \emph{i.e.}, no stabilization terms are considered. We conclude this paragraph by mentioning the so--called dual weighted residual method  (DWR)  \cite{MR1430239} and its applications to the optimal control of flow problems \cite{MR1867885,MR2275762}.

Recently, the authors of \cite{MR3212590} propose and analyze an a posteriori error estimator for problem \eqref{eq:min}--\eqref{constraint} when $\kappa =0$ \cite[Section 5]{MR3212590}. The associated discrete optimal system incorporates stabilized terms, into the state and adjoint equations, that are based on the streamline--diffusion finite element method (SDFEM). On the basis of proposed and analyzed a posteriori error estimators for the state and adjoint equations, the authors derive an estimator for \eqref{eq:min}--\eqref{constraint}. We comment that the obtained upper bound for the error, in terms of the a posteriori error estimator, is not computable.

In this work we analyze a family of a posteriori error estimators in a unifying framework that incorporates a wide variety of standard and stabilized finite element methods.


\section{A posteriori error analysis}\label{A_posteriori}
In this section we derive and analyze a posteriori error estimators for the solution to the discretization \eqref{fem_control} of the optimal control problem \eqref{optimal_weak_system}.

\subsection{Reliability analysis}
\label{reliability}

We begin this section by introducing the following notation. Let $\boldsymbol{\esf}_{\ysf}:=\bar{\ysf}-\bar{\ysf}_{\T}$, $\esf_{\psf}:=\bar{\psf}-\bar{\psf}_{\T}$, $\boldsymbol{\esf}_{\wsf}:=\bar{\wsf}-\bar{\wsf}_{\T}$, $\esf_{\qsf}:=\bar{\qsf}-\bar{\qsf}_{\T}$ and $\boldsymbol{\esf}_{\usf}:=\bar{\usf}-\bar{\usf}_{\T}$, where $(\bar{\ysf},\bar{\psf},\bar{\wsf},\bar{\qsf},\bar{\usf}) \in \mathbf{V} \times Q \times \mathbf{V} \times Q \times \mathbf{U}_{ad}$ is the solution to the optimality system \eqref{optimal_weak_system} and $(\bar{\ysf}_{\T},\bar{\psf}_{\T},\bar{\wsf}_{\T},\bar \qsf_{\T},\bar{\usf}_{\T}) \in \mathbf{V}(\T) \times Q(\T) \times \mathbf{V}(\T) \times Q(\T) \times \mathbf{U}_{\textrm{ad}}(\T)$ is its numerical approximation given as the solution to \eqref{fem_control}. The goal of this section is to obtain an upper bound for
\begin{equation}
\label{eq:global_error}
\norm{(\boldsymbol{\esf}_{\ysf},\esf_{\psf},\boldsymbol{\esf}_{\wsf},\esf_{\qsf},\boldsymbol{\esf}_{\usf})}_{\Omega}^2:=\sum_{K\in\T}\norm{(\boldsymbol{\esf}_{\ysf},\esf_{\psf},\boldsymbol{\esf}_{\wsf},\esf_{\qsf},\boldsymbol{\esf}_{\usf})}_{K}^2
\end{equation}
where
\[
\norm{(\boldsymbol{\esf}_{\ysf},\esf_{\psf},\boldsymbol{\esf}_{\wsf},\esf_{\qsf},\boldsymbol{\esf}_{\usf})}_{K}^2:=
\normv{\boldsymbol{\esf}_{\ysf}}_{\mathbf{V},K}^{2}
+\varrho\normp{\esf_{\psf}}_{Q,K}^{2}
+\normv{\boldsymbol{\esf}_{\wsf}}_{\mathbf{V},K}^{2}
+\varrho\normp{\esf_{\qsf}}_{Q,K}^{2}
+\|\boldsymbol{\esf}_{\usf}\|_{\boldsymbol{L}^2(K)}^{2}.
\]
The norms $\norm{\cdot}_{\mathbf{V},K}$ and  $\normp{\cdot}_{Q,K}$ are defined as in \eqref{eq:local_energy_norm} and the parameter $\varrho$ is a nonnegative constant that will be arbitrary in the analysis but fixed in the numerical experiments of Section \ref{sec:numex}. 

The upper bound for the error \eqref{eq:global_error} that we obtain is constructed using upper bounds on the error between the solution to the discretization \eqref{fem_control} and auxilliary variables that we define in what follows. Let $(\hat{\ysf},\hat{\psf}) \in \mathbf{V} \times Q$ be the solution to
\begin{equation}
\label{eq:y_hat-p_hat}
\left\{
\begin{array}{rcll}
\bia(\hat{\ysf},\xisf)-\bib(\xisf,\hat{\psf})
& = &
(\fsf+\bar{\usf}_{\T},\xisf)_{\boldsymbol{L}^{2}(\Omega)} &
\forall~\xisf\in \mathbf{V},\\
\bib(\hat{\ysf},\phi) & = & 0 &
\forall~\phi\in Q.
\end{array}
\right.
\end{equation}
We notice that, in view of \eqref{fem_control}, we have that $(\bar \ysf_{\T},\bar \psf_{\T})\in \mathbf{V}(\T)\times Q(\T)$ satisfies
\begin{equation}\label{fem_control_state}
\left\{
\begin{array}{rcll}
\bia(\bar{\ysf}_{\T},\xisf) -\bib(\xisf,\bar{\psf}_{\T}) + \mathcal{S}(\bar{\ysf}_{\T},\bar{\psf}_{\T},\fsf+\bar{\usf}_{\T};\xisf)
& = & (\fsf+\bar{\usf}_{\T},\xisf)_{\boldsymbol{L}^{2}(\Omega)} &
\\
\bib(\bar{\ysf}_{\T},\phi) + \mathcal{H}(\bar{\ysf}_{\T},\bar{\psf}_{\T},\fsf+\bar{\usf}_{\T};\phi) 
& = & 0
\end{array}
\right.
\end{equation}
for all $\xisf\in \mathbf{V}(\T)$ and $\phi\in Q(\T)$. Consequently, $(\bar \ysf_{\T},\bar \psf_{\T})$ can be seen as a finite element approximation of the solution to \eqref{eq:y_hat-p_hat}. We thus make the following assumption:

\textbf{Assumption 1.} There exist quantities $\eta_{\ysf}$ and $\eta_{\psf}$ which are such that
\begin{equation}
\normv{\hat{\ysf}-\bar{\ysf}_{\T}}_{\mathbf{V},\Omega} \leq \eta_{\ysf}
\mbox{ and }
\normp{\hat{\psf}-\bar{\psf}_{\T}}_{Q,\Omega} \leq \eta_{\psf}.
\end{equation}

Let $(\hat{\wsf},\hat{\qsf}) \in \mathbf{V} \times Q$ be the solution to 
\begin{equation}\label{eq:w_hat-q_hat}
\left\{
\begin{array}{rcll}
\bic(\hat{\wsf}, \zetasf)+\bib(\zetasf,\hat{\qsf}) & = & (\bar{\ysf}_{\T}-\ysf_{\Omega},\zetasf)_{\boldsymbol{L}^{2}(\Omega)} &
\forall~\zetasf\in \mathbf{V},\\
\bib(\hat{\wsf},\psi) & = & 0 &
\forall~\psi\in Q.
\end{array}
\right.
\end{equation}
We notice that, again in view of \eqref{fem_control}, $(\bar{\wsf}_{\T},\bar{\qsf}_{\T})\in\mathbf{V}(\T)\times Q(\T)$ satisfies
\begin{equation}\label{fem_control_adjoint}
\left\{
\begin{array}{rcll}
\bic(\bar{\wsf}_{\T},\zetasf)+\bib(\zetasf,\bar{\qsf}_{\T})
+ \mathcal{Q}(\bar{\wsf}_{\T},\bar{\qsf}_{\T},\bar{\ysf}_{\T}-\ysf_{\Omega};\zetasf)
& = & (\bar{\ysf}_{\T}-\ysf_{\Omega},\zetasf)_{\boldsymbol{L}^{2}(\Omega)}, 
\\
\bib(\bar{\wsf}_{\T},\psi) + \mathcal{K}(\bar{\wsf}_{\T},\bar{\qsf}_{\T},\bar{\ysf}_{\T}-\ysf_{\Omega};\psi)  & = & 0, 
\end{array}
\right.
\end{equation}
for all $\zetasf\in\mathbf{V}(\T)$ and $\psi\in Q(\T)$, and hence $(\bar{\wsf}_{\T},\bar{\qsf}_{\T})$ corresponds to a finite element approximation of the solution to \eqref{eq:w_hat-q_hat}. We thus make the following assumption:

\textbf{Assumption 2.} There exist quantities $\eta_{\wsf}$ and $\eta_{\qsf}$ which are such that
\begin{equation}
\normv{\hat{\wsf}-\bar{\wsf}_{\T}}_{\mathbf{V},\Omega} \leq \eta_{\wsf}
\mbox{ and }
\normp{\hat{\qsf}-\bar{\qsf}_{\T}}_{Q,\Omega} \leq \eta_{\qsf}.
\end{equation}

We introduce the auxiliary control variable 
\begin{equation}\label{eq:u_tilde}
\tilde{\usf}=\Pi_{[\aasf,\bbsf]}\left(-\tfrac{1}{\vartheta}\bar{\wsf}_{\T}\right).
\end{equation}
We define the error between this auxilliary control variable and $\bar \usf_{\T}$ as follows:
\begin{equation}
 \label{eq:indicator_control}
 \eta_{\usf}:= \left( \sum_{K \in \T} \eta_{\textsf{ct},K}^2 \right)^{1/2},
\mbox{ with } 
  \eta_{\usf,K} := \| \tilde{\usf} - \bar \usf_{\T} \|_{\boldsymbol{L}^2(K)}.
\end{equation}

We also define
\begin{equation}
\label{eq:Cy}
\mathfrak{C}_{\ysf}=2
+2\mu\mathtt{C}_{\Omega}^{6}
+4(1+\varrho\omega)(\mathtt{C}_{\Omega}^{4}+\mu\mathtt{C}_{\Omega}^{8}+2\mu\mathtt{C}_{\Omega}^{12}),
\end{equation}
\begin{equation}
\label{eq:Cw}
\mathfrak{C}_{\wsf}=2
+\mu\mathtt{C}_{\Omega}^{2}
+2\mu(1+\varrho\omega)(\mathtt{C}_{\Omega}^{4}+2\mathtt{C}_{\Omega}^{8}),
\end{equation}
and
\begin{equation}
\label{eq:Cu}
\mathfrak{C}_{\usf}=2
+2\mu\mathtt{C}_{\Omega}^{8}
+4(1+\varrho\omega)(\mathtt{C}_{\Omega}^{2}+2\mathtt{C}_{\Omega}^{6}+\mu\mathtt{C}_{\Omega}^{10}+2\mu\mathtt{C}_{\Omega}^{14}),
\end{equation}
with $\mu = 4\vartheta^{-2}$ and $\omega = \mathtt{C}_{\textsf{is}}^2 ( 1 + \mathtt{C}_{\textsf{ct}})^2$.

We now present the analysis through which we obtain an upper bound for the total error.
\begin{theorem}[global reliability]
\label{th:global_reliability}
If \textbf{Assumptions 1} and \textbf{2} hold, then 
\begin{equation}
\label{eq:reliability}
\norm{(\boldsymbol{\esf}_{\ysf},\esf_{\psf},\boldsymbol{\esf}_{\wsf},\esf_{\qsf},\boldsymbol{\esf}_{\usf})}_{\Omega}^2
\leq 
\Upsilon^{2}
\end{equation}
where
\begin{equation}
\Upsilon^{2}:=
\mathfrak{C}_{\ysf} 
\eta_{\ysf}^{2} +
2\varrho
\eta_{\psf}^{2} +
\mathfrak{C}_{\wsf} 
\eta_{\wsf}^{2} +
2\varrho
\eta_{\qsf}^{2} +
\mathfrak{C}_{\usf} 
\eta_{\usf}^{2},
\end{equation}
and $\mathfrak{C}_{\ysf}$, $\mathfrak{C}_{\wsf}$ and $\mathfrak{C}_{\usf}$ are defined by \eqref{eq:Cy}, \eqref{eq:Cw} and \eqref{eq:Cu}, respectively.
\end{theorem}

\begin{proof}
We proceed in 6 steps.
~\\
\noindent Step 1. The goal of this step is to control the term $\|\boldsymbol{\esf}_{\usf}\|_{{\boldsymbol{L}^2(\Omega)}}$. We begin with a simple application of the triangle inequality to write
\begin{equation}\label{u-uh-utilde}
 \|\boldsymbol{\esf}_{\usf}\|_{\boldsymbol{L}^2(\Omega)}^{2} \leq 2 \|\bar{\usf}-\tilde{\usf}\|_{\boldsymbol{L}^2(\Omega)}^{2}+2\|\tilde{\usf}-\bar{\usf}_{\T}\|_{\boldsymbol{L}^2(\Omega)}^{2}=2\|  \bar{\usf} -\tilde{\usf} \|_{\boldsymbol{L}^2(\Omega)}^{2} + 2\eta_{\usf}^{2},
\end{equation}
where $\tilde{\usf} = \Pi_{[\aasf,\bbsf]} \left(- \tfrac{1}{\vartheta} \bar{\wsf}_{\T}\right)$ and $\eta_{\usf}$ is defined as in \eqref{eq:indicator_control}. 

Let us now bound the first term on the right hand side of \eqref{u-uh-utilde}. To accomplish this task we first observe a key property that the auxiliary control variable $\tilde{\usf}$ satisfies:
\begin{equation}\label{tildeu}
(\bar{\wsf}_{\T}+\vartheta \tilde{\usf}, \usf - \tilde{\usf})_{\boldsymbol{L}^2(\Omega)} \geq 0 \quad \forall \usf \in \mathbf{U}_{\textrm{ad}};
\end{equation}
see Lemma 2.26 and Theorem 2.28 in \cite{MR2583281}. Set $\usf = \tilde \usf$ in the variational inequality of \eqref{optimal_weak_system} and $\usf = \bar \usf$ in \eqref{tildeu}. We thus obtain that
\begin{equation*}
(\bar{\wsf}+\vartheta\bar{\usf},\tilde{\usf}-\bar{\usf})_{\boldsymbol{L}^{2}(\Omega)} \geq 0, \quad
(\bar{\wsf}_{\T}+\vartheta \tilde{\usf}, \bar{\usf} - \tilde{\usf})_{\boldsymbol{L}^2(\Omega)} \geq 0,
\end{equation*}
and, consequently, that
\begin{equation}\label{eq:inter}
\vartheta\| \bar{\usf} - \tilde{\usf} \|^2_{\boldsymbol{L}^2(\Omega)} \leq (\bar{\wsf} - \bar{\wsf}_{\T}, \tilde{\usf} - \bar{\usf})_{\boldsymbol{L}^2(\Omega)}.
\end{equation}

In order to bound the right hand side of \eqref{eq:inter}, we first define $(\tilde{\ysf},\tilde{\psf})\in \mathbf{V} \times Q$ as the solution to
\begin{equation}\label{eq:y_tilde,p_tilde}
\left\{
\begin{array}{rcll}
\bia(\tilde{\ysf},\xisf)-\bib(\xisf,\tilde{\psf}) 
& = &
\left(\fsf+\tilde{\usf},\xisf\right)_{\boldsymbol{L}^{2}(\Omega)} &
\forall~\xisf\in \mathbf{V},\\
\bib(\tilde{\ysf},\phi) & = & 0 &
\forall~\phi\in Q.
\end{array}
\right.
\end{equation}
In addition, we define $(\tilde{\wsf},\tilde{\qsf}) \in \mathbf{V} \times Q$ as the solution to
\begin{equation}\label{eq:w_tilde,q_tilde}
\left\{
\begin{array}{rcll}
\bic(\tilde{\wsf},\zetasf)+\bib(\zetasf,\tilde{\qsf}) & = & (\tilde{\ysf}-\ysf_{\Omega},\zetasf)_{\boldsymbol{L}^{2}(\Omega)} &
\forall~\zetasf\in\mathbf{V},\\
\bib(\tilde{\wsf},\psi) & = & 0 &
\forall ~\psi\in Q.
\end{array}
\right.
\end{equation}
Utilizing the states $\hat \wsf$ and $\tilde \wsf$ defined as the solutions to \eqref{eq:w_hat-q_hat} and \eqref{eq:w_tilde,q_tilde}, respectively, we arrive at
\begin{multline*}
\vartheta\| \bar{\usf} - \tilde{\usf} \|^2_{\boldsymbol{L}^2(\Omega)}
\leq 
(\bar{\wsf}  - \tilde{\wsf}, \tilde{\usf} - \bar{\usf})_{\boldsymbol{L}^2(\Omega)} + 
(\tilde{\wsf}  - \hat{\wsf}, \tilde{\usf} - \bar{\usf})_{\boldsymbol{L}^2(\Omega)} +
(\hat{\wsf} - \bar{\wsf}_{\T}, \tilde{\usf} - \bar{\usf})_{\boldsymbol{L}^2(\Omega)} \\
\leq 
(\bar{\wsf}  - \tilde{\wsf}, \tilde{\usf} - \bar{\usf})_{\boldsymbol{L}^2(\Omega)} + 
\tfrac{1}{\vartheta}\|\tilde{\wsf}  - \hat{\wsf}\|_{\boldsymbol{L}^{2}(\Omega)}^{2} +
\tfrac{1}{\vartheta}\|\hat{\wsf} - \bar{\wsf}_{\T}\|_{\boldsymbol{L}^{2}(\Omega)}^{2} +
\tfrac{\vartheta}{2}
\|\bar{\usf} - \tilde{\usf}\|_{\boldsymbol{L}^2(\Omega)}^{2}
\end{multline*}
upon using Cauchy--Schwarz and Young's inequalities. Hence,
\begin{equation}
\| \bar{\usf} - \tilde{\usf} \|^2_{\boldsymbol{L}^2(\Omega)} \leq 
\tfrac{2}{\vartheta}(\bar{\wsf}  - \tilde{\wsf}, \tilde{\usf} - \bar{\usf})_{\boldsymbol{L}^2(\Omega)} + 
\tfrac{2}{\vartheta^2}\left( \|\tilde{\wsf}  - \hat{\wsf}\|_{\boldsymbol{L}^{2}(\Omega)}^{2} +
\|\hat{\wsf} - \bar{\wsf}_{\T}\|_{\boldsymbol{L}^{2}(\Omega)}^{2} \right).
\label{eq:first_estimate_u}
\end{equation}

We proceed to bound $(\bar{\wsf}  - \tilde{\wsf}, \tilde{\usf} - \bar{\usf})_{\boldsymbol{L}^2(\Omega)}$. To accomplish this task, we first notice that, since $(\bar{\wsf},\bar \qsf)$ solves the adjoint problem of the optimality system \eqref{optimal_weak_system} and $(\tilde \wsf, \tilde \qsf)$ solves \eqref{eq:w_tilde,q_tilde}, the fact that $\bar \psf - \tilde \psf \in Q$ implies that $\bib(\bar{\wsf}-\tilde{\wsf},\tilde{\psf}-\bar{\psf})=0$. Thus, since $(\bar{\ysf},\bar \psf)$ and $(\tilde{\ysf},\tilde \psf)$ solve \eqref{optimal_weak_system} and \eqref{eq:y_tilde,p_tilde}, respectively, we arrive at
\[
 ( \bar{\usf} - \tilde{\usf},\bar{\wsf}  - \tilde{\wsf} )_{\boldsymbol{L}^2(\Omega)} = \bia(\bar{\ysf} - \tilde{\ysf},\bar{\wsf}  - \tilde{\wsf}). 
\]
We now invoke \eqref{ac_link} and, again, the fact that $(\bar \wsf, \bar \qsf)$ and $(\tilde \wsf, \tilde \qsf)$ solve \eqref{optimal_weak_system} and \eqref{eq:w_tilde,q_tilde}, respectively, to obtain that
\begin{equation}
 (\tilde{\usf} - \bar{\usf},\bar{\wsf}  - \tilde{\wsf} )_{\boldsymbol{L}^2(\Omega)} = \bia(\tilde{\ysf} -\bar{\ysf},\bar{\wsf}  - \tilde{\wsf}) = \bic(\bar{\wsf}-\tilde{\wsf},\tilde{\ysf}-\bar{\ysf}) = - \|\bar{\ysf}-\tilde{\ysf}\|_{\boldsymbol{L}^{2}(\Omega)}^{2} \leq 0,
\label{eq:leq0}
\end{equation}
upon noticing that, since $(\bar{\ysf},\bar \psf)$ solves the state equations of the optimality system \eqref{optimal_weak_system} and $(\tilde \ysf, \tilde \psf)$ solves \eqref{eq:y_tilde,p_tilde}, the fact that $\bar \qsf - \tilde \qsf \in Q$ implies that $\bib(\bar{\ysf}-\tilde{\ysf},\bar \qsf - \tilde \qsf)=0$.

Using the previous estimate in \eqref{eq:first_estimate_u} we obtain that
\begin{equation}\label{eq:usf-tildeusf_2}
\| \bar{\usf} - \tilde{\usf} \|^2_{\boldsymbol{L}^2(\Omega)} \leq  
\tfrac{2}{\vartheta^2}\|\tilde{\wsf}  - \hat{\wsf}\|_{\boldsymbol{L}^{2}(\Omega)}^{2} +
\tfrac{2}{\vartheta^2}\|\hat{\wsf} - \bar{\wsf}_{\T}\|_{\boldsymbol{L}^{2}(\Omega)}^{2}.
\end{equation}

The control of the second term on the right hand side of \eqref{eq:usf-tildeusf_2} follows from \eqref{C_Omega} and \textbf{Assumption 2}:
\begin{equation*}
\|\hat{\wsf} - \bar{\wsf}_{\T}\|_{\boldsymbol{L}^{2}(\Omega)}^{2} 
\leq \texttt{C}_{\Omega}^{2}\eta_{\wsf}^{2}.
\end{equation*}

We now turn our attention to bounding the term $\|\tilde{\wsf} - \hat{\wsf}\|_{\boldsymbol{L}^{2}(\Omega)}$. Applying similar arguments to the ones that lead to \eqref{eq:leq0} we obtain that
\begin{equation}
\begin{aligned}
\norm{\tilde{\wsf}-\hat{\wsf}}_{\mathbf{V},\Omega}^{2} 
& = 
\bic(\tilde{\wsf}-\hat{\wsf},\tilde{\wsf}-\hat{\wsf})
= 
(\tilde{\ysf}-\bar{\ysf}_{\T},\tilde{\wsf}-\hat{\wsf})_{\boldsymbol{L}^{2}(\Omega)}
 \\
&
 \leq \texttt{C}_{\Omega}
\|\tilde{\ysf}-\bar{\ysf}_{\T}\|_{\boldsymbol{L}^{2}(\Omega)}
\normv{\tilde{\wsf}-\hat{\wsf}}_{\mathbf{V},\Omega},
\label{eq:aux_estimate}
\end{aligned}
\end{equation}
where we have also used \eqref{C_Omega}. Consequently,
$
\|\tilde{\wsf}  - \hat{\wsf}\|_{\boldsymbol{L}^{2}(\Omega)}^{2} \leq 
\texttt{C}_{\Omega}^{4}
\|\tilde{\ysf}-\bar{\ysf}_{\T}\|_{\boldsymbol{L}^{2}(\Omega)}^{2},
$
upon using, again, \eqref{C_Omega}.
It thus suffices to bound $\|\tilde{\ysf}-\bar{\ysf}_{\T}\|_{\boldsymbol{L}^{2}(\Omega)}$. We proceed as follows:
\begin{equation*}
\|\tilde{\ysf}-\bar{\ysf}_{\T}\|_{\boldsymbol{L}^{2}(\Omega)}^{2}
\leq
2
\|\tilde{\ysf}-\hat{\ysf}\|_{\boldsymbol{L}^{2}(\Omega)}^{2}
+
2\|\hat{\ysf}-\bar{\ysf}_{\T}\|_{\boldsymbol{L}^{2}(\Omega)}^{2}.
\end{equation*}
To control the second term on the right hand side of the previous expression, we invoke \textbf{Assumption 1} and \eqref{C_Omega}. We thus conclude that
\begin{equation*}
\|\hat{\ysf}-\bar{\ysf}_{\T}\|_{\boldsymbol{L}^{2}(\Omega)}^{2}
\leq \texttt{C}_{\Omega}^{2}
\eta_{\ysf}^{2}.
\end{equation*}
To bound the first term, we employ that $(\hat \ysf, \hat \psf)$ and $(\tilde \ysf, \tilde \psf)$ solve \eqref{eq:y_hat-p_hat} and \eqref{eq:y_tilde,p_tilde}, respectively. This, on the basis of $\nabla \cdot \csf = 0$ and \eqref{C_Omega}, yields
\begin{equation}
\begin{aligned}
\norm{\tilde{\ysf}-\hat{\ysf}}_{\mathbf{V},\Omega}^{2}
& =
\bia(\tilde{\ysf}-\hat{\ysf},\tilde{\ysf}-\hat{\ysf}) 
=
(\tilde{\usf}-\bar{\usf}_{\T},\tilde{\ysf}-\hat{\ysf})_{\boldsymbol{L}^{2}(\Omega)} \\
& \leq 
\texttt{C}_{\Omega} \|\tilde{\usf}-\bar{\usf}_{\T}\|_{\boldsymbol{L}^{2}(\Omega)}\norm{\tilde{\ysf}-\hat{\ysf}}_{\mathbf{V},\Omega},
\end{aligned}
\label{eq:similar_arguments}
\end{equation}
which allows us to conclude, in view of \eqref{eq:indicator_control} and \eqref{C_Omega}, that 
\begin{equation*}
\|\tilde{\ysf}-\hat{\ysf}\|_{\boldsymbol{L}^{2}(\Omega)}^{2}
\leq 
\texttt{C}_{\Omega}^{4}
\eta_{\usf}^{2}.
\end{equation*}

On the basis of \eqref{u-uh-utilde} and \eqref{eq:usf-tildeusf_2}, we combine our previous findings 
and arrive at
\begin{align}\label{control_error_bound}
\|\boldsymbol{\esf}_{\usf}\|_{\boldsymbol{L}^{2}(\Omega)}^{2} 
& \leq
2\mu
 \texttt{C}_{\Omega}^{6}
\eta_{\ysf}^{2} 
+ 
\mu
\texttt{C}_{\Omega}^{2}
\eta_{\wsf}^{2} 
+
\left(2 + 2\mu
\texttt{C}_{\Omega}^{8}
\right)
\eta_{\usf}^{2},
\end{align}
where $\mu = 4 \vartheta^{-2}$.
~\\
\noindent Step 2. The goal of this step is to bound $\norm{\boldsymbol{\esf}_{\ysf}}_{\mathbf{\mathbf{V}},\Omega}$. To accomplish this task, we apply the triangle inequality and invoke \textbf{Assumption 1}. In fact,
\begin{equation}\label{triangle_y}
\normv{\boldsymbol{\esf}_{\ysf}}_{\mathbf{\mathbf{V}},\Omega}^{2} 
\leq 
2\normv{\bar{\ysf}-\hat{\ysf}}_{\mathbf{\mathbf{V}},\Omega}^{2} 
+
2\normv{\hat{\ysf}-\bar{\ysf}_{\T}}_{\mathbf{\mathbf{V}},\Omega}^{2} 
\leq 
2\normv{\bar{\ysf}-\hat{\ysf}}_{\mathbf{\mathbf{V}},\Omega}^{2} 
+
2\eta_{\ysf}^{2}
.
\end{equation}
To control the remaining term we employ similar ideas to the ones that lead to \eqref{eq:similar_arguments}. These arguments reveal that
\begin{equation}
\normv{\bar{\ysf}-\hat{\ysf}}_{\mathbf{\mathbf{V}},\Omega}^2
\leq \texttt{C}_{\Omega}^2
\|\bar{\usf}-\bar{\usf}_{\T}\|^2_{\boldsymbol{L}^{2}(\Omega)},
\label{bary-haty}
\end{equation}
which combined with \eqref{control_error_bound} and \eqref{triangle_y}, implies the error estimate
\begin{equation}
\begin{aligned}
\normv{\boldsymbol{\esf}_{\ysf}}_{\mathbf{\mathbf{V}},\Omega}^{2}
\leq
2\left(
2\mu
\texttt{C}_{\Omega}^{8}
+ 1
\right)\eta_{\ysf}^{2}
+ 
2\mu
\texttt{C}_{\Omega}^{4}
\eta_{\wsf}^{2}
+
2\texttt{C}_{\Omega}^{2}\left(2+2\mu
\texttt{C}_{\Omega}^{8}
\right)
\eta_{\usf}^{2}.
\label{state_error_bound}
\end{aligned}
\end{equation}
\noindent Step 3. We now bound the term $\normv{\boldsymbol{\esf}_{\wsf}}_{\mathbf{V},\Omega}$. To accomplish this task, we use, again, the triangle inequality and \textbf{Assumption 2} to obtain that
\begin{equation}\label{triangle_w}
\normv{\boldsymbol{\esf}_{\wsf}}_{\mathbf{V},\Omega}^{2} 
\leq 
2\normv{\bar{\wsf}-\hat{\wsf}}_{\mathbf{V},\Omega}^{2} 
+
2\eta_{\wsf}^{2}.
\end{equation}
To bound $\normv{\bar{\wsf}-\hat{\wsf}}_{\mathbf{V},\Omega}^{2}$ we invoke  the optimality system \eqref{optimal_weak_system} and \eqref{eq:w_hat-q_hat}. In fact, the arguments that allow us to obtain \eqref{eq:aux_estimate} immediately yield
\begin{align*}
\normv{\bar{\wsf}-\hat{\wsf}}_{\mathbf{V},\Omega}^{2} & = \bic(\bar{\wsf}-\hat{\wsf},\bar{\wsf}-\hat{\wsf}) 
= (\bar{\ysf}-\bar{\ysf}_{\T},\bar{\wsf}-\hat{\wsf})_{\boldsymbol{L}^{2}(\Omega)} 
\\
& \leq \|\bar{\ysf}-\bar{\ysf}_{\T}\|_{\boldsymbol{L}^{2}(\Omega)}\|\bar{\wsf}-\hat{\wsf}\|_{\boldsymbol{L}^{2}(\Omega)}
\end{align*}
upon using a Cauchy--Schwarz inequality. In view of \eqref{C_Omega}, we conclude that
\begin{equation}\label{bar(w)-hat(w)}
\normv{\bar{\wsf}-\hat{\wsf}}_{\mathbf{V},\Omega}^{2} \leq
\texttt{C}_{\Omega}^{4}\normv{\bar{\ysf}-\bar{\ysf}_{\T}}_{\mathbf{V},\Omega}^{2},
\end{equation}
which, combined with the estimates \eqref{state_error_bound} and \eqref{triangle_w}, yields 
\begin{equation}
\begin{aligned}
\normv{\boldsymbol{\esf}_{\wsf}}_{\mathbf{V},\Omega}^{2}
& \leq
4\texttt{C}_{\Omega}^{4}
\left(
2\mu
\texttt{C}_{\Omega}^{8}
+ 1
\right)
\eta_{\ysf}^{2} 
+ 2 \left(
2\mu
\texttt{C}_{\Omega}^{8}
+1\right) 
\eta_{\wsf}^{2}
+ 4\texttt{C}_{\Omega}^{6}
\left(2 + 2\mu
\texttt{C}_{\Omega}^{8}
\right)
\eta_{\usf}^{2}.
\end{aligned}
\label{adjoint_error_bound}
\end{equation}
\noindent Step 4. We now bound $\norm{\esf_{\psf}}_{Q,\Omega}$. We start with a simple application of the triangle inequality and \textbf{Assumption 1}:
\[
\normp{\esf_{\psf}}_{Q,\Omega}^{2}
\leq 2
\normp{\bar{\psf}-\hat{\psf}}_{Q,\Omega}^{2} + 2\normp{\hat{\psf}-\bar{\psf}_{\T}}_{Q,\Omega}^{2}
\leq 
2\normp{\bar{\psf}-\hat{\psf}}_{Q,\Omega}^{2} + 2\eta_{\psf}^{2};
\]
we recall that $( \hat \ysf, \hat \psf)$ solves \eqref{eq:y_hat-p_hat}. To control the first term on the right hand side of the previous expression, we utilize the inf-sup condition \eqref{inf-sup_2}:
\begin{equation}
\label{inf-sup3}
\norm{\bar{\psf}-\hat{\psf}}_{Q,\Omega}\leq \texttt{C}_{\textsf{is}}\sup_{\xisf\in \mathbf{V}\setminus\{\boldsymbol{0}\}}
\frac{\bib(\xisf,\bar{\psf}-\hat{\psf})}{\norm{\xisf}_{\mathbf{V},\Omega}}.
\end{equation}
Since $(\bar \ysf, \bar \psf)$ and $(\hat \ysf, \hat \psf)$ solve \eqref{optimal_weak_system} and \eqref{eq:y_hat-p_hat}, respectively, we conclude that
\begin{align*}
\bib(\xisf,\bar{\psf}-\hat{\psf})
& =
\bia(\bar{\ysf}-\hat{\ysf},\xisf)
-
(\bar{\usf}-\bar{\usf}_{\T},\xisf)_{\boldsymbol{L}^{2}(\Omega)}
 \\
& \leq 
\left(
\mathtt{C}_{\textsf{ct}}\normv{\bar{\ysf}-\hat{\ysf}}_{\mathbf{V},\Omega}
+
\texttt{C}_{\Omega} \|\bar{\usf}-\bar{\usf}_{\T}\|_{\boldsymbol{L}^{2}(\Omega)}
\right)
\normv{\xisf}_{\mathbf{V},\Omega},
\end{align*}
upon using \eqref{C_Omega} and \eqref{eq:a_continuous}. In view of \eqref{bary-haty} we thus arrive at
\begin{equation*}
\bib(\xisf,\bar{\psf}-\hat{\psf}) \leq \texttt{C}_{\Omega} ( 1 + \mathtt{C}_{\textsf{ct}} )
 \|\bar{\usf}-\bar{\usf}_{\T}\|_{\boldsymbol{L}^{2}(\Omega)} \normv{\xisf}_{\mathbf{V},\Omega}.
\end{equation*}
This and \eqref{inf-sup3} imply that
$
\normp{\bar{\psf}-\hat{\psf}}_{Q,\Omega}\leq \texttt{C}_{\textsf{is}}\texttt{C}_{\Omega} ( 1 + \mathtt{C}_{\textsf{ct}}) \|\bar{\usf}-\bar{\usf}_{\T}\|_{\boldsymbol{L}^{2}(\Omega)}.
$
Thus,
\begin{equation}
\normp{\esf_{\psf}}_{Q,\Omega}^{2} \leq 2\omega \texttt{C}_{\Omega}^2
\|\bar{\usf}-\bar{\usf}_{\T}\|_{\boldsymbol{L}^{2}(\Omega)}^2 + 2\eta_{\psf}^{2},
 \label{eq:p-pt}
\end{equation}
where $\omega =\texttt{C}_{\textsf{is}}^2 ( 1 + \mathtt{C}_{\textsf{ct}})^2 $.
We conclude the estimate for $\normp{\esf_{\psf}}_{Q,\Omega}^{2}$ by invoking 
\eqref{control_error_bound}:
\begin{equation}\label{pfinal_error_bound}
\begin{aligned}
\normp{\esf_{\psf}}_{Q,\Omega}^{2}  
\leq &
4\mu
\omega
 \texttt{C}_{\Omega}^{8}
\eta_{\ysf}^{2}
+ 
2\mu
\omega
\texttt{C}_{\Omega}^{4}
\eta_{\wsf}^{2}
+2\omega\texttt{C}_{\Omega}^{2}
\left(2 + 2\mu
\texttt{C}_{\Omega}^{8}
\right)
\eta_{\usf}^{2}
+
2\eta_{\psf}^{2}.
\end{aligned}
\end{equation}
\noindent Step 5. We bound $\norm{\esf_{\qsf}}_{Q,\Omega}$. Similar arguments to the ones employed in the previous step yield
\[
\normp{\esf_{\qsf}}_{Q,\Omega}^{2}
\leq 2
\normp{\bar{\qsf}-\hat{\qsf}}_{Q,\Omega}^{2} + 2\normp{\hat{\qsf}-\bar{\qsf}_{\T}}_{Q,\Omega}^{2}
\leq 
2\normp{\bar{\qsf}-\hat{\qsf}}_{Q,\Omega}^{2} + 2\eta_{\qsf}^{2}
\]
and
\begin{align*}
\normp{\bar{\qsf}-\hat{\qsf}}_{Q,\Omega}
& \leq \texttt{C}_{\textsf{is}}
\sup_{\zetasf\in\mathbf{V}\setminus\{\boldsymbol{0}\}}
\frac{
(\bar{\ysf}-\bar{\ysf}_{\T},\zetasf)_{\boldsymbol{L}^{2}(\Omega)}
-
\bic(\bar{\wsf}-\hat{\wsf},\zetasf)
}{\normv{\zetasf}_{\mathbf{V},\Omega}} \\
& \leq \texttt{C}_{\textsf{is}}
\left(
\texttt{C}_{\Omega}^{2}\normv{\bar{\ysf}-\bar{\ysf}_{\T}}_{\mathbf{V},\Omega}
+
\mathtt{C}_{\textsf{ct}} \normv{\bar{\wsf}-\hat{\wsf}}_{\mathbf{V},\Omega}
\right).
\end{align*} 
We finally use \eqref{bar(w)-hat(w)}, and conclude that
$
\normp{\bar{\qsf}-\hat{\qsf}}_{Q,\Omega} \leq \texttt{C}_{\textsf{is}}
\texttt{C}_{\Omega}^{2}\left(
1 +  \mathtt{C}_{\textsf{ct}} 
\right)
\normv{\bar{\ysf}-\bar{\ysf}_{\T}}_{\mathbf{V},\Omega},
$
and then that
\begin{equation}
\normp{\esf_{\qsf}}_{Q,\Omega}^{2} \leq 2 \omega
\texttt{C}_{\Omega}^{4}
\normv{\bar{\ysf}-\bar{\ysf}_{\T}}_{\mathbf{V},\Omega}^2
+
2\eta_{\qsf}^{2},
 \label{eq:q-qt}
\end{equation}
where, we recall that, $\omega =\texttt{C}_{\textsf{is}}^2 ( 1 + \mathtt{C}_{\textsf{ct}})^2$. Consequently,
\begin{equation}\label{qfinal_error_bound}
\begin{aligned}
\normp{\esf_{\qsf}}_{Q,\Omega}^{2}
& \leq
4 \omega\texttt{C}_{\Omega}^{4}
\left(
2\mu
\texttt{C}_{\Omega}^{8}
+ 1
\right)
\eta_{\ysf}^{2}
+ 
4\mu\omega
\texttt{C}_{\Omega}^{8}
\eta_{\wsf}^{2} 
+
4
\omega
\texttt{C}_{\Omega}^{6}\left(2+
2\mu
\texttt{C}_{\Omega}^{8}
\right)
\eta_{\usf}^{2}
+
2\eta_{\qsf}^{2}
.
\end{aligned}
\end{equation}
\noindent Step 6. Combining \eqref{control_error_bound}, \eqref{state_error_bound}, \eqref{adjoint_error_bound}, \eqref{pfinal_error_bound} and \eqref{qfinal_error_bound} allows us to arrive at \eqref{eq:reliability}.
\end{proof}


It is important in a posteriori error analysis to have an upper bound for the error that is in terms of local error indicators, so that it can be used to adaptively refine the mesh. Such a bound follows from Theorem \ref{th:global_reliability} under the following two assumptions.

\textbf{Assumption 3.} There exist quantities $\eta_{\ysf,K}$ and $\eta_{\psf,K}$ that are such that
\begin{equation}
\normv{\hat{\ysf}-\bar{\ysf}_{\T}}_{\mathbf{V},\Omega}^2 \leq \sum_{K\in\T}\eta_{\ysf,K}^2\mbox{ and }\normp{\hat{\psf}-\bar{\psf}_{\T}}_{Q,\Omega}^2 \leq \sum_{K\in\T}\eta_{\psf,K}^2.
\end{equation}

\textbf{Assumption 4.} There exist quantities $\eta_{\wsf,K}$ and $\eta_{\qsf,K}$ that are such that
\begin{equation}
\normv{\hat{\wsf}-\bar{\wsf}_{\T}}_{\mathbf{V},\Omega}^2 \leq \sum_{K\in\T}\eta_{\wsf,K}^2\mbox{ and }
\normp{\hat{\qsf}-\bar{\qsf}_{\T}}_{Q,\Omega}^2 \leq \sum_{K\in\T}\eta_{\qsf,K}^2.
\end{equation}

\begin{theorem}[global reliability]
\label{th:global_reliability2}
If \textbf{Assumptions 3} and \textbf{4} hold, then 
\begin{equation}
\label{eq:reliability2}
\norm{(\boldsymbol{\esf}_{\ysf},\esf_{\psf},\boldsymbol{\esf}_{\wsf},\esf_{\qsf},\boldsymbol{\esf}_{\usf})}_{\Omega}^2
\leq 
\sum_{K\in\T}\Upsilon_K^{2}
\end{equation}
where
\begin{equation}
\label{eq:indicator}
\Upsilon_K^{2}:=
\mathfrak{C}_{\ysf} 
\eta_{\ysf,K}^{2} +
2\varrho
\eta_{\psf,K}^{2} +
\mathfrak{C}_{\wsf} 
\eta_{\wsf,K}^{2} +
2\varrho
\eta_{\qsf,K}^{2} +
\mathfrak{C}_{\usf} 
\eta_{\usf,K}^{2},
\end{equation}
and $\mathfrak{C}_{\ysf}$, $\mathfrak{C}_{\wsf}$ and $\mathfrak{C}_{\usf}$ are defined by \eqref{eq:Cy} \eqref{eq:Cw}, and \eqref{eq:Cu}, respectively.
\end{theorem}
\begin{proof}
In view of \textbf{Assumptions 3} and \textbf{4}, the proof follows from a simple application of the result of Theorem \ref{th:global_reliability}.
\end{proof}

Theorem \ref{th:global_reliability2} can be used to obtain guaranteed upper bounds on the error if the value of a $\beta$ satisfying \eqref{inf-sup} is known and the quantities $\eta_{\ysf,K}$, $\eta_{\psf,K}$, $\eta_{\wsf,K}$ and $\eta_{\qsf,K}$ are computable. If this is not the case then Theorem \eqref{th:global_reliability2} can still be used to arrive at an a posteriori error estimator under the following assumption.

\textbf{Assumption 5.} There exist computable quantities $\tilde{\eta}_{\ysf,K}$, $\tilde{\eta}_{\psf,K}$, $\tilde{\eta}_{\wsf,K}$ and $\tilde{\eta}_{\qsf,K}$ which are such that $\eta_{\ysf,K}\lesssim\tilde{\eta}_{\ysf,K}$, $\eta_{\psf,K}\lesssim\tilde{\eta}_{\psf,K}$, $\eta_{\wsf,K}\lesssim\tilde{\eta}_{\wsf,K}$ and $\eta_{\qsf,K}\lesssim\tilde{\eta}_{\qsf,K}$ for all $K\in\T$.

\begin{corollary}[global reliability]
\label{th:global_reliability3}
If \textbf{Assumptions 3}, \textbf{4} and \textbf{5} hold, then
\begin{equation}
\label{eq:reliability3}
\norm{(\boldsymbol{\esf}_{\ysf},\esf_{\psf},\boldsymbol{\esf}_{\wsf},\esf_{\qsf},\boldsymbol{\esf}_{\usf})}_{\Omega}^2
\lesssim\tilde{\Upsilon}
:=\sum_{K\in\T}\tilde{\Upsilon}_K^{2}
\end{equation}
where
\begin{equation}\label{eq:indicator2}
\tilde{\Upsilon}_K^{2}:=
\tilde{\eta}_{\ysf,K}^{2} +
\tilde{\eta}_{\psf,K}^{2} + 
\tilde{\eta}_{\wsf,K}^{2} +
\tilde{\eta}_{\qsf,K}^{2} +
\tilde{\eta}_{\usf,K}^{2}.
\end{equation}
\end{corollary}
\begin{proof}
Upon invoking \textbf{Assumptions 3}, \textbf{4} and \textbf{5}, the estimate \eqref{eq:reliability3} is a consequence of Theorem \ref{th:global_reliability2}.
\end{proof}

\subsection{Efficiency analysis}
\label{efficiency}

In this section we prove the local efficiency of the a posteriori error indicators $\Upsilon_K$ and $\tilde{\Upsilon}_K$ defined by \eqref{eq:indicator} and \eqref{eq:indicator2}, respectively. In what follows we will assume that \textbf{Assumptions 3}, \textbf{4} and \textbf{5} are satisfied and that $\varrho\ne0$. In addition, we make two further assumptions. To state them, we first define, for nonnegative integers $l$, the discrete space
\begin{equation}
\mathbb{P}_l(\T)=\left\{\vsf \in\mathbf{L}^{2}(\Omega):\,\vsf_{|K}\in\mathbb{P}_l(K)^d\mbox{ for all }K\in\T\right\}.
\end{equation}

Our first additional assumption reads as follows:

\textbf{Assumption 6.} The spaces $\mathbf{V}(\T)$ and $Q(\T)$ and the set $\mathbf{U}_{ad}(\T)$ are such that
\begin{itemize}
\item $\mathbf{V}(\T)=\mathbf{V}\cap\mathbb{P}_{l_\mathbf{V}}(\T)$ for some positive integer $l_\mathbf{V}$,
\item $Q(\T)=Q\cap\mathbb{P}_{l_Q}(\T)$ for some nonnegative integer $l_Q$
or
$Q(\T)=Q\cap\mathbb{P}_{l_Q}(\T)\cap H^1(\Omega)$ for some positive integer $l_Q$,
\item $\mathbf{U}_{ad}(\T)=\mathbf{U}_{ad}\cap\mathbb{P}_{l_\mathbf{U}}(\T)$ for some nonnegative integer $l_\mathbf{U}$
or
$\mathbf{U}_{ad}(\T)=\mathbf{U}_{ad}\cap\mathbb{P}_{l_\mathbf{U}}(\T)\cap H^1(\Omega)$ for some positive integer $l_\mathbf{U}$.
\end{itemize}

For $K\in\T$, we define the following residuals and oscillation terms:
\begin{equation}
\label{eq:int_res_st}
\boldsymbol{\mathcal{R}}_{K}^{\textsf{st}}:=  
\Pi_{K,m}(\fsf)+\bar{\usf}_{\T|K}+\varepsilon\Delta\ysf_{\T|K}-
\Pi_{K,m}(\left(\csf\cdot\nabla\right)\bar{\ysf}_{\T|K})-\kappa\bar{\ysf}_{\T|K}-\nabla\bar{\psf}_{\T|K},
\end{equation}
\begin{equation}
\boldsymbol{\mathcal{R}}_{K}^{\textsf{ad}}:=  
\bar{\ysf}_{\T|K}-\Pi_{K,m}(\ysf_{\Omega})+\varepsilon\Delta\bar{\wsf}_{\T|K}+
\Pi_{K,m}(\left(\csf\cdot\nabla\right)\bar{\wsf}_{\T|K})-\kappa\bar{\wsf}_{\T|K}+\nabla\bar{\qsf}_{\T|K},
\end{equation}
\begin{equation}
\label{eq:osc_st}
\textbf{\textsf{osc}}_{K}^{\textsf{st}}:= 
\fsf-\Pi_{K,m}(\fsf)-(\left(\csf\cdot\nabla\right)\bar{\ysf}_{\T|K}-\Pi_{K,m}(\left(\csf\cdot\nabla\right)\bar{\ysf}_{\T|K})),
\end{equation}
and
\begin{equation}
\textbf{\textsf{osc}}_{K}^{\textsf{ad}}:= 
-(\ysf_{\Omega}-\Pi_{K,m}(\ysf_{\Omega}))+(\left(\csf\cdot\nabla\right)\bar{\wsf}_{\T|K}-\Pi_{K,m}(\left(\csf\cdot\nabla\right)\bar{\wsf}_{\T|K})),
\end{equation}
where $m=\max\left\{l_\mathbf{V},l_Q-1,l_\mathbf{U}\right\}$. We recall that the operator $\Pi_{K,m}$ is defined as in \eqref{projectionL2K}, and notice that, in view of the choice of $m$, we have the following invariance property: $\Pi_{K,m}(\boldsymbol{\mathcal{R}}_{K}^{\textsf{st}})= \boldsymbol{\mathcal{R}}_{K}^{\textsf{st}}$ and $\Pi_{K,m}(\boldsymbol{\mathcal{R}}_{K}^{\textsf{ad}})= \boldsymbol{\mathcal{R}}_{K}^{\textsf{ad}}$. For $\gamma\in\mathcal{F}_{I}$, we define

\begin{equation}\label{eq:jump_st}
\llbracket \boldsymbol{\mathcal{R}}_{\gamma}^{\textsf{st}} \rrbracket:= 
\sum_{K\in\Omega_\gamma}
\boldsymbol{\mathcal{R}}_{\gamma,K}^{\textsf{st}}
\quad\textrm{with}\quad
\boldsymbol{\mathcal{R}}_{\gamma,K}^{\textsf{st}}:=
-\varepsilon\left(\boldsymbol{n}_{\gamma}^{K}\cdot\nabla\right)\bar{\ysf}_{\T|K}+\bar{\psf}_{\T|K}\boldsymbol{n}_{\gamma}^{K},
\end{equation}
and
\begin{equation}
\llbracket \boldsymbol{\mathcal{R}}_{\gamma}^{\textsf{ad}} \rrbracket:=
\sum_{K\in\Omega_\gamma}\boldsymbol{\mathcal{R}}_{\gamma,K}^{\textsf{ad}}
\quad\textrm{with}\quad
\boldsymbol{\mathcal{R}}_{\gamma,K}^{\textsf{ad}}:=
-\varepsilon\left(\boldsymbol{n}_{\gamma}^{K}\cdot\nabla\right)\bar{\wsf}_{\T|K}-\bar{\qsf}_{\T|K}\boldsymbol{n}_{\gamma}^{K}.
\end{equation}

We now state our final assumption.

\textbf{Assumption 7.} For all $K\in\T$, the computable quantities $\tilde{\eta}_{\ysf,K}$, $\tilde{\eta}_{\psf,K}$ $\tilde{\eta}_{\wsf,K}$, and $\tilde{\eta}_{\qsf,K}$, introduced in \textbf{Assumption 5}, are such that
\begin{align}
\tilde{\Upsilon}_K^2 \lesssim &
\|\nabla\cdot\bar\ysf_\T\|_{L^{2}(K)}^{2}+\|\nabla\cdot\bar\wsf_\T\|_{L^{2}(K)}^{2}
+
\sum_{K'\in\hat{\T}_K}h_K^{2}\left(\|\boldsymbol{\mathcal{R}}_{K'}^{\textsf{st}}\|_{\boldsymbol{L}^{2}(K')}^{2}+\|\boldsymbol{\mathcal{R}}_{K'}^{\textsf{ad}}\|_{\boldsymbol{L}^{2}(K')}^{2}\right) \nonumber
\\
&
+\sum_{\gamma\in\hat{\mathcal{F}}_K}h_K\left(\|\llbracket \boldsymbol{\mathcal{R}}_{\gamma}^{\textsf{st}} \rrbracket\|_{\boldsymbol{L}^{2}(\gamma)}^{2}+\|\llbracket \boldsymbol{\mathcal{R}}_{\gamma}^{\textsf{ad}} \rrbracket\|_{\boldsymbol{L}^{2}(\gamma)}^{2}\right) \nonumber
\\
&
+\sum_{K'\in\hat{\T}_K}h_K^{2}\left(\|\textbf{\textsf{osc}}_{K'}^{\textsf{st}}\|_{\boldsymbol{L}^{2}(K')}^{2}+\|\textbf{\textsf{osc}}_{K'}^{\textsf{ad}}\|_{\boldsymbol{L}^{2}(K')}^{2}\right)+\eta_{\usf,K}^2
 \label{ass2lb}
\end{align}
where $\hat{\T}_K\subset\T$ and $\hat{\mathcal{F}}_K\subset\mathcal{F}_I$.

Under \textbf{Assumptions 3}, \textbf{4}, \textbf{5}, \textbf{6} and \textbf{7} we present an efficiency analysis. We start by noting that, since $\Upsilon_K\lesssim\tilde{\Upsilon}_K$, we only need to bound terms that appear on the right hand side of \eqref{ass2lb}. 

We first invoke integration by parts and \eqref{state_constraint} to conclude that
\begin{equation*}
\begin{split}
&\sum_{K\in\T}
(\boldsymbol{\mathcal{R}}_{K}^{\textsf{st}} ,\xisf)_{\boldsymbol{L}^{2}(K)}
+\sum_{\gamma\in\mathcal{F}_{I}}
(\llbracket \boldsymbol{\mathcal{R}}_{\gamma}^{\textsf{st}} \rrbracket,\xisf)_{\boldsymbol{L}^{2}(\gamma)}
\\
=&
\bia(\boldsymbol{\esf}_{\ysf},\xisf)+\bib(\xisf,\esf_{\psf})
-(\boldsymbol{\esf}_{\usf},\xisf)_{\boldsymbol{L}^{2}(\Omega)}-\sum_{K\in\T}(\textbf{\textsf{osc}}_{K}^{\textsf{st}},\xisf)_{\boldsymbol{L}^{2}(K)} \quad \forall \xisf \in\mathbf{V}.
\end{split}
\end{equation*}
We now apply standard bubble function arguments \cite{AObook,MR3059294} to this equation to obtain
\begin{equation}\label{stelres}
\|\boldsymbol{\mathcal{R}}_{K}^{\textsf{st}} \|_{\boldsymbol{L}^{2}(K)}^2
\lesssim
h_K^{-2}\left(\normv{\boldsymbol{\esf}_{\ysf}}_{\mathbf{V},K}^{2}
+\varrho\normp{\esf_{\psf}}_{Q,K}^{2}\right) + 
\|\boldsymbol{\esf}_{\usf}\|_{\boldsymbol{L}^{2}(K)}^2 +
\|\textbf{\textsf{osc}}_{K}^{\textsf{st}}\|_{\boldsymbol{L}^{2}(K)}^2
\end{equation}
for $K\in\T$, and that, for $\gamma\in\mathcal{F}_I$,
\begin{align}
\|\llbracket \boldsymbol{\mathcal{R}}_{\gamma}^{\textsf{st}} \rrbracket\|_{\boldsymbol{L}^{2}(\gamma)}^2
\lesssim &
\sum_{K'\in\Omega_{\gamma}}
\Big(
h_{K'}^{-1}\left(\normv{\boldsymbol{\esf}_{\ysf}}_{\mathbf{V},K'}^{2}
+\varrho\normp{\esf_{\psf}}_{Q,K'}^{2}\right)\nonumber
\\
&+ h_{K'}\left(\|\boldsymbol{\esf}_{\usf}\|_{\boldsymbol{L}^{2}(K')}^2 +
\|\textbf{\textsf{osc}}_{K'}^{\textsf{st}}\|_{\boldsymbol{L}^{2}(K')}^2\right) \Big).
\label{stefres}
\end{align}

On the other hand, using \eqref{adjoint_state} and, again, integration by parts we obtain that
\begin{equation*}
\begin{split}
&\sum_{K\in\T}
(\boldsymbol{\mathcal{R}}_{K}^{\textsf{ad}} ,\xisf)_{\boldsymbol{L}^{2}(K)}
+\sum_{\gamma\in\mathcal{F}_{I}}
(\llbracket \boldsymbol{\mathcal{R}}_{\gamma}^{\textsf{ad}} \rrbracket,\xisf)_{\boldsymbol{L}^{2}(\gamma)} 
\\
=& \bic(\boldsymbol{\esf}_{\wsf},\xisf)-\bib(\xisf,\esf_{\qsf})
-(\boldsymbol{\esf}_{\ysf},\xisf)_{\boldsymbol{L}^{2}(\Omega)}-\sum_{K\in\T}(\textbf{\textsf{osc}}_{K}^{\textsf{ad}},\xisf)_{\boldsymbol{L}^{2}(K)} \quad \forall \xisf \in\mathbf{V}.
\end{split}
\end{equation*}
Applying standard bubble function arguments, again, to this equation yields
\begin{equation}\label{adelres}
\|\boldsymbol{\mathcal{R}}_{K}^{\textsf{ad}} \|_{\boldsymbol{L}^{2}(K)}^2
\lesssim
h_K^{-2}\left(\normv{\boldsymbol{\esf}_{\wsf}}_{\mathbf{V},K}^{2}
+\varrho\normp{\esf_{\qsf}}_{Q,K}^{2}\right) + 
\|\boldsymbol{\esf}_{\ysf}\|_{\boldsymbol{L}^{2}(K)}^2 +
\|\textbf{\textsf{osc}}_{K}^{\textsf{ad}}\|_{\boldsymbol{L}^{2}(K)}^2
\end{equation}
for $K\in\T$, and, for $\gamma\in\mathcal{F}_I$,
\begin{align}
\|\llbracket \boldsymbol{\mathcal{R}}_{\gamma}^{\textsf{ad}} \rrbracket\|_{L^{2}(\gamma)}^2
\lesssim&
\sum_{K'\in\Omega_{\gamma}}
\Big(
h_{K'}^{-1}\left(\normv{\boldsymbol{\esf}_{\wsf}}_{\mathbf{V},K'}^{2}
+\varrho\normp{\esf_{\qsf}}_{Q,K'}^{2}\right) \nonumber
\\
& + h_{K'}\left(\|\boldsymbol{\esf}_{\ysf}\|_{\boldsymbol{L}^{2}(K')}^2 +
\|\textbf{\textsf{osc}}_{K'}^{\textsf{ad}}\|_{\boldsymbol{L}^{2}(K')}^2\right)
\Big). 
\label{adefres}
\end{align}

We now proceed to bound the terms $\|\nabla\cdot\bar\ysf_\T\|_{L^{2}(K)}^{2}$ and $\|\nabla\cdot\bar\wsf_\T\|_{L^{2}(K)}^{2}$ in \eqref{ass2lb}. To accomplish this task, we notice that $\nabla\cdot\xisf\in Q$ for all $\xisf\in\mathbf{V}$. Then, it follows from the second equation of \eqref{optimal_weak_system} that $\nabla\cdot \bar\ysf = 0$, and thus that
\begin{equation}\label{divyTbound}
\|\nabla\cdot\bar\ysf_\T\|_{L^{2}(K)}^{2}=\|\nabla\cdot\boldsymbol{\esf}_{\ysf}\|_{L^{2}(K)}^{2}\lesssim\normv{\boldsymbol{\esf}_{\ysf}}_{\mathbf{V},K}^{2}.
\end{equation}
Similarly, it follows from the fourth equation of \eqref{optimal_weak_system} that
\begin{equation}\label{divwTbound}
\|\nabla\cdot\bar\wsf_\T\|_{L^{2}(K)}^{2}=\|\nabla\cdot\boldsymbol{\esf}_{\wsf}\|_{L^{2}(K)}^{2}\lesssim\normv{\boldsymbol{\esf}_{\wsf}}_{\mathbf{V},K}^{2}.
\end{equation}

We conclude with an estimate for the term $\eta_{\usf,K}$ defined by \eqref{eq:indicator_control}:
\begin{equation*}
\eta_{\usf,K}
 \leq 
\left\| \boldsymbol{\esf}_{\usf}  \right\|_{\boldsymbol{L}^2(K)} + 
\left\| \Pi_{[\aasf,\bbsf]} (-\tfrac{1}{\vartheta}\bar{\wsf})-\Pi_{[\aasf,\bbsf]} (-\tfrac{1}{\vartheta}\bar{\wsf}_{\T})  \right\|_{\boldsymbol{L}^2(K)} 
 \leq 
\left\| \boldsymbol{\esf}_{\usf}  \right\|_{\boldsymbol{L}^2(K)} + 
\tfrac{1}{\vartheta}
\|\boldsymbol{\esf}_{\wsf}\|_{\boldsymbol{L}^{2}(K)}
\end{equation*}
upon invoking the triangle inequality, \eqref{projection_formula}, and \eqref{eq:projproperty}. Hence,
\begin{equation}\label{controleff}
\eta_{\usf,K}^2\lesssim
\left\| \boldsymbol{\esf}_{\usf}  \right\|_{\boldsymbol{L}^2(K)}^2 + 
\|\boldsymbol{\esf}_{\wsf}\|_{\boldsymbol{L}^{2}(K)}^2.
\end{equation}

The following theorem then follows upon combining \eqref{ass2lb}--\eqref{controleff}.

\begin{theorem}[local efficiency]
If $\varrho\ne0$ and \textbf{Assumptions 3}, \textbf{4}, \textbf{5}, \textbf{6} and \textbf{7} hold, then
\begin{equation*}
\begin{split}
\Upsilon_{K}^{2}\lesssim\tilde{\Upsilon}_{K}^{2}\lesssim&
\|\boldsymbol{\esf}_{\wsf}\|_{\boldsymbol{L}^{2}(K)}^2
+
\sum_{K'\in \tilde{\Omega}_K}
\bigg(
\norm{(\boldsymbol{\esf}_{\ysf},\esf_{\psf},\boldsymbol{\esf}_{\wsf},\esf_{\qsf},\boldsymbol{\esf}_{\usf})}_{K'}^2
\\
&+h_{K'}^{2}\left(
\|\boldsymbol{\esf}_{\usf}\|_{\boldsymbol{L}^{2}(K')}^2 +
\|\boldsymbol{\esf}_{\ysf}\|_{\boldsymbol{L}^{2}(K')}^2
+
\|\textbf{\textsf{osc}}_{K'}^{\textsf{st}}\|_{\boldsymbol{L}^{2}(K')}^2 +
\|\textbf{\textsf{osc}}_{K'}^{\textsf{ad}}\|_{\boldsymbol{L}^{2}(K')}^2\right)\bigg),
\end{split}
\end{equation*}
with
$
\displaystyle\tilde{\Omega}_K=\hat{\T}_K \cup\bigcup_{\gamma\in\hat{\mathcal{F}}_K}\Omega_\gamma.
$
\end{theorem}

The following corollary follows upon using \eqref{C_Omega} and the fact that $\Omega$ is bounded.

\begin{corollary}[global efficiency]
If $\varrho\ne0$ and \textbf{Assumptions 3}, \textbf{4}, \textbf{5}, \textbf{6} and \textbf{7} hold, then
\begin{equation*}
\sum_{K\in\T}\Upsilon_{K}^{2}\lesssim\tilde{\Upsilon}^{2}\lesssim
\norm{(\boldsymbol{\esf}_{\ysf},\esf_{\psf},\boldsymbol{\esf}_{\wsf},\esf_{\qsf},\boldsymbol{\esf}_{\usf})}_{\Omega}^2
+\sum_{K\in\T}h_K^2\left(
\|\textbf{\textsf{osc}}_{K}^{\textsf{st}}\|_{\boldsymbol{L}^{2}(K)}^2 +
\|\textbf{\textsf{osc}}_{K}^{\textsf{ad}}\|_{\boldsymbol{L}^{2}(K)}^2\right).
\end{equation*}
\end{corollary}

\section{A particular example}\label{particular}

Henceforth, we shall consider a particular case of the approximation scheme \eqref{fem_control}. 
We set $\mathbf{V}(\T)=\mathbf{V}\cap\mathbb{P}_{1}(\T)$, $Q(\T)=Q\cap\mathbb{P}_{0}(\T)$, $\mathbf{U}_{ad}(\T)=\mathbf{U}_{ad}\cap\mathbb{P}_{0}(\T)$,
\begin{equation}
\mathcal{S}(\bar{\ysf}_{\T},\bar{\psf}_{\T},\fsf+\bar{\usf}_{\T};\boldsymbol{\xisf})  =
\sum_{K\in\T}\mathcal{S}_K(\bar{\ysf}_{\T},\bar{\psf}_{\T},\fsf+\bar{\usf}_{\T};\boldsymbol{\xisf}),
\end{equation}
\begin{equation}
\mathcal{H}(\bar{\ysf}_{\T},\bar{\psf}_{\T},\fsf+\bar{\usf}_{\T};\phi)  =
\tau_{\gamma}\sum_{\gamma\in \mathcal{F}_{I}}h_{\gamma}\left([\bar{\psf}_{\T}],[\phi]\right)_{L^{2}(\gamma)},
\end{equation}
\begin{equation}
\mathcal{Q}(\bar{\wsf}_{\T},\bar{\qsf}_{\T},\bar{\ysf}_{\T}-\ysf_{\Omega};\zetasf)  =
\sum_{K\in\T}\mathcal{Q}_K(\bar{\wsf}_{\T},\bar{\qsf}_{\T},\bar{\ysf}_{\T}-\ysf_{\Omega};\zetasf),
\end{equation}
and
\begin{equation}
\mathcal{K}(\bar{\wsf}_{\T},\bar{\qsf}_{\T},\bar{\ysf}_{\T}-\ysf_{\Omega};\psi)  =
-\tau_{\gamma}\sum_{\gamma\in \mathcal{F}_{I}}h_{\gamma}\left([\bar{\qsf}_{\T}],[\psi]\right)_{L^{2}(\gamma)},
\end{equation}
where
\[
\mathcal{S}_K(\bar{\ysf}_{\T},\bar{\psf}_{\T},\fsf+\bar{\usf}_{\T};\boldsymbol{\xisf})=\tau_{K}(\left(\csf\cdot\nabla\right)\bar{\ysf}_{\T}+\kappa \bar{\ysf}_{\T}-(\fsf+\bar{\usf}_{\T}),\left(\csf\cdot\nabla\right)\boldsymbol{\xisf})_{L^{2}(K)},
\]
\[
\mathcal{Q}_K(\bar{\wsf}_{\T},\bar{\qsf}_{\T},\bar{\ysf}_{\T}-\ysf_{\Omega};\zetasf)=\tau_{K}(\left(\csf\cdot\nabla\right)\bar{\wsf}_{\T}-\kappa \bar{\wsf}_{\T}+\bar{\ysf}_{\T}-\ysf_{\Omega},\left(\csf\cdot\nabla\right)\zetasf)_{L^{2}(K)}
\]
and $[v]$ denotes the jumps in $v$. The stabilization parameters $\tau_\gamma$ and $\tau_{K}$ are such that $\tau_\gamma>0$ and $0<\tau_{K}\lesssim h_K^2$. Note that these choices correspond to solving the state equations using a particular case of the method given by \cite[equation (3.6)]{MR2454024} and are such that \textbf{Assumption 6} is satisfied.

We note that alternative methods for solving the state equations can be found in \cite{braack2007stabilized} but we restrict our attention to the method described above in order to simplify the presentation.

\subsection{Fully computable a posteriori error estimators}\label{sec:fully}

In this section we obtain a posteriori error estimators that satisfy the assumptions of Section \ref{A_posteriori} and are fully computable if the value of a $\beta$ satisfying \eqref{inf-sup} is known. We first define some quantities that the estimators will be defined in terms of.

For $\varsigma=\textsf{st}$ and $\varsigma=\textsf{ad}$, let the equilibrated fluxes $\boldsymbol{g}_{\gamma,K}^{\varsigma}\in\mathbb{P}_1(\gamma)^d$ be such that
\begin{equation}\label{fluxsum}
\boldsymbol{g}_{\gamma,K}^{\varsigma}+\boldsymbol{g}_{\gamma,K'}^{\varsigma}=0\mbox{, if }\gamma\in\mathcal{F}_{K}\cap\mathcal{F}_{K'}\mbox{, }K,K'\in\T\mbox{, }K\neq K',
\end{equation}
\begin{eqnarray*}
(\fsf+\bar{\usf}_{\T},\boldsymbol{\lambda})_{\boldsymbol{L}^{2}(K)}
-
\varepsilon(\nabla\bar{\ysf}_{\T},\nabla\boldsymbol{\lambda})_{\underset{\approx}{\boldsymbol{L}}^{2}(K)}
-
(\kappa\bar{\ysf}_{\T} + \left(\csf \cdot\nabla\right)\bar{\ysf}_{\T} ,\boldsymbol{\lambda})_{\boldsymbol{L}^{2}(K)}&&
\\
+
(\bar{\psf}_{\T},\nabla\cdot\boldsymbol{\lambda})_{L^{2}(K)}
-
\mathcal{S}_K(\bar{\ysf}_{\T},\bar{\psf}_{\T},\fsf+\bar{\usf}_{\T};\boldsymbol{\lambda})
+
\sum_{\gamma\in\mathcal{F}_K}(\boldsymbol{g}_{\gamma,K}^{\textsf{st}},\boldsymbol{\lambda})_{\boldsymbol{L}^{2}(\gamma)}
&&=0
\end{eqnarray*}
for all $\boldsymbol{\lambda}\in\mathbb{P}_1(K)^d$ and all $K\in\mathcal{P}$,
\begin{eqnarray*}
(\bar{\ysf}_{\T}-\ysf_{\Omega},\boldsymbol{\lambda})_{\boldsymbol{L}^{2}(K)}
-
\varepsilon(\nabla\bar{\wsf}_{\T},\nabla\boldsymbol{\lambda})_{\underset{\approx}{\boldsymbol{L}}^{2}(K)}
-
(\kappa\bar{\wsf}_{\T} - \left(\csf \cdot\nabla\right)\bar{\wsf}_{\T} ,\boldsymbol{\lambda})_{\boldsymbol{L}^{2}(K)}&&
\\
-
(\bar{\qsf}_{\T},\nabla\cdot\boldsymbol{\lambda})_{L^{2}(K)}
-
\mathcal{Q}_K(\bar{\wsf}_{\T},\bar{\qsf}_{\T},\bar{\ysf}_{\T}-\ysf_{\Omega};\boldsymbol{\lambda})
+
\sum_{\gamma\in\mathcal{F}_K}(\boldsymbol{g}_{\gamma,K}^{\textsf{ad}},\boldsymbol{\lambda})_{\boldsymbol{L}^{2}(\gamma)}
&&=0
\end{eqnarray*}
for all $\boldsymbol{\lambda}\in\mathbb{P}_1(K)^d$ and all $K\in\mathcal{P}$, and
\begin{equation}\label{fluxbound}
\sum_{\gamma\in\mathcal{F}_K}h_K\|\boldsymbol{g}_{\gamma,K}^{\varsigma}+\boldsymbol{\mathcal{R}}_{\gamma,K}^{\varsigma}\|_{L^{2}(\gamma)}^{2}
\lesssim
\sum_{K'\in\hat{\T}_K}h_K^{2}\|\boldsymbol{\mathcal{R}}_{K'}^{\varsigma}\|_{\boldsymbol{L}^{2}(K')}^{2}
+\sum_{\gamma\in\hat{\mathcal{F}}_K}h_K\|\llbracket \boldsymbol{\mathcal{R}}_{\gamma}^{\varsigma} \rrbracket\|_{\boldsymbol{L}^{2}(\gamma)}^{2}
\end{equation}
for all $K\in\mathcal{P}$, where
\begin{equation*}
\hat{\T}_K=\{K'\in\T:~ \mathcal{V}_{K}\cap\mathcal{V}_{K'}\neq\emptyset\} \quad \textrm{and} \quad
\displaystyle\hat{\mathcal{F}}_K=\bigcup_{\gamma\in\mathcal{F_K}}\{\gamma'\in\mathcal{F}_I:~ \mathcal{V}_{\gamma}\cap\mathcal{V}_{\gamma'}\neq\emptyset\}
\end{equation*}
with $\mathcal{V}_{K}$ denoting the set containing the vertices of element $K$ and $\mathcal{V}_{\gamma}$ denoting the set containing the vertices of the edge/face $\gamma$. For information that will help with the construction of such $\boldsymbol{g}_{\gamma,K}^{\varsigma}$ we refer the reader to \cite[Chapter 6]{AObook} and \cite{MR3556402,allendes2016adaptive}.

For $\varsigma=\textsf{st}$ and $\varsigma=\textsf{ad}$, we also define $\boldsymbol{\sigma}_K^{\varsigma}\in\mathbb{P}_2(K)^{d\times d}$ to be such that
\begin{equation*}
\left\{
\begin{array}{l}
-\mathbf{div}\,\boldsymbol{\sigma}_K^{\varsigma}  = \boldsymbol{\mathcal{R}}_{K}^{\varsigma}\textrm{ in }K, 
\\
\boldsymbol{\sigma}_K^{\varsigma}\boldsymbol{n}_{\gamma}^{K} = \boldsymbol{g}_{\gamma,K}^{\varsigma}+\boldsymbol{\mathcal{R}}_{\gamma,K}^{\varsigma}\textrm{ on }\gamma,~\forall~\gamma\in\mathcal{F}_{K},
\end{array}
\right.
\end{equation*}
and $\|\boldsymbol{\sigma}_K^{\varsigma}\|_{\boldsymbol{L}^{2}(K)}$ is minimized. We note that the $\boldsymbol{g}_{\gamma,K}^{\varsigma}$ are such that the data in the above problem are compatible in the sense that $\boldsymbol{\sigma}_K^{\varsigma}$ exists. Moreover, for all $K\in\T$,
\begin{equation}\label{sigmaidentity}
(\boldsymbol{\sigma}_K^{\varsigma},\nabla\xisf)_{\underset{\approx}{\boldsymbol{L}}^{2}(K)}
=
(\boldsymbol{\mathcal{R}}_{K}^{\varsigma} ,\xisf)_{\boldsymbol{L}^{2}(K)}
+\sum_{\gamma\in\mathcal{F}_{K}}
(\boldsymbol{g}_{\gamma,K}^{\varsigma}+\boldsymbol{\mathcal{R}}_{\gamma,K}^{\varsigma},\xisf)_{\boldsymbol{L}^{2}(\gamma)}\quad\forall~\xisf\in \mathbf{V}
\end{equation}
and
\begin{equation}\label{sigmabound}
\|\boldsymbol{\sigma}_K^{\varsigma}\|_{\boldsymbol{L}^{2}(K)}^2\lesssim
h_K^{2}\|\boldsymbol{\mathcal{R}}_{K'}^{\varsigma}\|_{\boldsymbol{L}^{2}(K)}^{2}
+\sum_{\gamma\in\mathcal{F}_K}h_K\|\boldsymbol{g}_{\gamma,K}^{\varsigma}+\boldsymbol{\mathcal{R}}_{\gamma,K}^{\varsigma}\|_{\boldsymbol{L}^{2}(\gamma)}^{2}.
\end{equation}
For information on the construction of such $\boldsymbol{\sigma}_K^{\varsigma}$ we refer the reader to \cite{ABR2017,MR3556402}.

Finally, for $\varsigma=\textsf{st}$ and $\varsigma=\textsf{ad}$, we define
\begin{equation}
\Psi_{\varsigma,K}=\frac{1}{\sqrt{\varepsilon}}\|\boldsymbol{\sigma}_K^{\varsigma}\|_{\boldsymbol{L}^{2}(K)}+
\texttt{C}_{K}
\|\boldsymbol{\mathsf{osc}}_{K}^{\varsigma}\|_{\boldsymbol{L}^{2}(K)}.
\end{equation}

We thus have the following result.

\begin{theorem}\label{th:fcyp}
\textbf{Assumption 3} holds with
\begin{equation}\label{etayK}
\eta_{\ysf,K}^2=
3\Psi_{\textup{\textsf{st}},K}^2
+
\textup{\texttt{C}}_{\textup{\textsf{is}}}^2\left(1+2\textup{\texttt{C}}_{\textup{\textsf{ct}}}^2\right)\|\nabla\cdot\bar\ysf_\T\|_{L^2(K)}^2
\end{equation}
and
\begin{equation}\label{etapK}
\eta_{\psf,K}^2=
2\textup{\texttt{C}}_{\textup{\textsf{is}}}^2\left(\left(1+
3\textup{\texttt{C}}_{\textup{\textsf{ct}}}^2\right)\Psi_{\textup{\textsf{st}},K}^2
+
\textup{\texttt{C}}_{\textup{\textsf{is}}}^2\textup{\texttt{C}}_{\textup{\textsf{ct}}}^2\left(1+2\textup{\texttt{C}}_{\textup{\textsf{ct}}}^2\right)\|\nabla\cdot\bar\ysf_\T\|_{L^2(K)}^2
\right).
\end{equation}
Moreover, \textbf{Assumption 1} holds with
\begin{equation}\label{etayp}
\eta_{\ysf}=\left(\sum_{K\in\mathcal{P}}\eta_{\ysf,K}^2\right)^{1/2}
\mbox{ and }
\eta_{\psf}=\left(\sum_{K\in\mathcal{P}}\eta_{\psf,K}^2\right)^{1/2}.
\end{equation}
\end{theorem}

\begin{proof}
Let $\boldsymbol{E}_{\ysf}\in\mathbf{V}$ be the solution to
\begin{equation}\label{eq:Ey}
\varepsilon(\nabla \boldsymbol{E}_{\ysf},\nabla \xisf)_{\underset{\approx}{\boldsymbol{L}}^{2}(\Omega)} + 
\kappa(\boldsymbol{E}_{\ysf},\xisf)_{\boldsymbol{L}^{2}(\Omega)}
=
\bia(\hat{\ysf}-\bar\ysf_\T,\xisf)-\bib(\xisf,\hat{\psf}-\bar\psf_\T)
\quad\forall~\xisf\in \mathbf{V}.
\end{equation}

Letting $\phi=\hat{\psf}-\psf_\T$ in \eqref{inf-sup_2} yields that
\[
\norm{\hat{\psf}-\bar\psf_\T}_{Q,\Omega} 
\leq \texttt{C}_{\textsf{is}}\sup_{\xisf\in \mathbf{V} \setminus\{\boldsymbol{0}\}}
\frac{\bib(\xisf,\hat{\psf}-\bar\psf_\T)}{\norm{\xisf}_{\mathbf{V},\Omega}}.
\]
To control the right--hand side of the previous estimate we use \eqref{eq:Ey} and obtain that
\begin{eqnarray*}
\bib(\xisf,\hat{\psf}-\bar\psf_\T)
&=&
\bia(\hat{\ysf}-\bar\ysf_\T,\xisf)
-\varepsilon(\nabla \boldsymbol{E}_{\ysf},\nabla \xisf)_{\underset{\approx}{\boldsymbol{L}}^{2}(\Omega)}
-\kappa(\boldsymbol{E}_{\ysf},\xisf)_{\boldsymbol{L}^{2}(\Omega)}
\\
&\le&
\texttt{C}_{\textsf{ct}}\norm{\hat{\ysf}-\bar\ysf_\T}_{\mathbf{V},\Omega}\norm{\xisf}_{\mathbf{V},\Omega}
+
\norm{\boldsymbol{E}_{\ysf}}_{\mathbf{V},\Omega}\norm{\xisf}_{\mathbf{V},\Omega},
\end{eqnarray*}
upon using \eqref{eq:a_continuous}. Hence,
\begin{equation}\label{eq:pbound}
\norm{\hat{\psf}-\bar\psf_\T}_{Q,\Omega}
\leq\texttt{C}_{\textsf{is}}\left(\norm{\boldsymbol{E}_{\ysf}}_{\mathbf{V},\Omega}+\texttt{C}_{\textsf{ct}}\norm{\hat{\ysf}-\ysf_\T}_{\mathbf{V},\Omega}\right).
\end{equation}

We now estimate $\norm{\hat{\ysf}-\bar{\ysf}_\T}_{\mathbf{V},\Omega}$. Since $\hat{\psf}-\bar\psf_\T\in Q$, by using the second equation of \eqref{eq:y_hat-p_hat} we have that
\begin{equation*}
\bib(\hat{\ysf}-\bar\ysf_\T,\hat{\psf}-\bar\psf_\T)=-\bib(\bar\ysf_\T,\hat{\psf}-\bar\psf_\T)\le\|\nabla\cdot\bar\ysf_\T\|_{L^2(\Omega)}\norm{\hat{\psf}-\bar\psf_\T}_{Q,\Omega}.
\end{equation*}
Thus, by using the previous estimate and letting $\xisf=\hat{\ysf}-\bar\ysf_\T$ in \eqref{eq:Ey}, we arrive at
\begin{eqnarray*}
\norm{\hat{\ysf}-\bar\ysf_\T}_{\mathbf{V},\Omega}^2
&=&
\varepsilon(\nabla \boldsymbol{E}_{\ysf},\nabla (\hat{\ysf}-\bar\ysf_\T))_{\underset{\approx}{\boldsymbol{L}}^{2}(\Omega)} + 
\kappa(\boldsymbol{E}_{\ysf},\hat{\ysf}-\bar\ysf_\T)_{\boldsymbol{L}^{2}(\Omega)}
+
\bib(\hat{\ysf}-\bar\ysf_\T,\hat{\psf}-\bar\psf_\T)
\\
&\le&
\norm{\boldsymbol{E}_{\ysf}}_{\mathbf{V},\Omega}\norm{\hat{\ysf}-\bar\ysf_\T}_{\mathbf{V},\Omega}
+
\|\nabla\cdot\bar\ysf_\T\|_{L^2(\Omega)}\norm{\hat{\psf}-\bar\psf_\T}_{Q,\Omega}.
\end{eqnarray*}
This, in view of \eqref{eq:pbound}, then yields that
\begin{eqnarray*}
\norm{\hat{\ysf}-\bar\ysf_\T}_{\mathbf{V},\Omega}^2
&\le&
\texttt{C}_{\textsf{is}}\|\nabla\cdot\bar\ysf_\T\|_{L^2(\Omega)}\norm{\boldsymbol{E}_{\ysf}}_{\mathbf{V},\Omega}
\\
&&+
\left(\norm{\boldsymbol{E}_{\ysf}}_{\mathbf{V},\Omega}
+
\texttt{C}_{\textsf{is}}\texttt{C}_{\textsf{ct}}\|\nabla\cdot\bar\ysf_\T\|_{L^2(\Omega)}\right)
\norm{\hat{\ysf}-\bar\ysf_\T}_{\mathbf{V},\Omega}
\\
&\le&
\tfrac{\texttt{C}_{\textsf{is}}^2}{2}\|\nabla\cdot\bar\ysf_\T\|_{L^2(\Omega)}^2+\tfrac{1}{2}\norm{\boldsymbol{E}_{\ysf}}_{\mathbf{V},\Omega}^2
\\
&&+
\tfrac{1}{2}\left(
\norm{\boldsymbol{E}_{\ysf}}_{\mathbf{V},\Omega}
+
\texttt{C}_{\textsf{is}}\texttt{C}_{\textsf{ct}}\|\nabla\cdot\bar\ysf_\T\|_{L^2(\Omega)}\right)^2
+
\tfrac{1}{2}\norm{\hat{\ysf}-\bar\ysf_\T}_{\mathbf{V},\Omega}^2
\end{eqnarray*}
from which it follows that
\[
\norm{\hat{\ysf}-\bar\ysf_\T}_{\mathbf{V},\Omega}^2
\le
\texttt{C}_{\textsf{is}}^2\|\nabla\cdot\bar\ysf_\T\|_{L^2(\Omega)}^2+\norm{\boldsymbol{E}_{\ysf}}_{\mathbf{V},\Omega}^2
+
\left(
\norm{\boldsymbol{E}_{\ysf}}_{\mathbf{V},\Omega}
+
\texttt{C}_{\textsf{is}}\texttt{C}_{\textsf{ct}}\|\nabla\cdot\bar\ysf_\T\|_{L^2(\Omega)}\right)^2.
\]
Hence, upon observing that
\[
\left(
\norm{\boldsymbol{E}_{\ysf}}_{\mathbf{V},\Omega}
+
\texttt{C}_{\textsf{is}}\texttt{C}_{\textsf{ct}}\|\nabla\cdot\bar\ysf_\T\|_{L^2(\Omega)}\right)^2\le
2\norm{\boldsymbol{E}_{\ysf}}_{\mathbf{V},\Omega}^2
+
2\texttt{C}_{\textsf{is}}^2\texttt{C}_{\textsf{ct}}^2\|\nabla\cdot\bar\ysf_\T\|_{L^2(\Omega)}^2,
\]
we can arrive at
\begin{equation}\label{ymytbound}
\norm{\hat{\ysf}-\bar\ysf_\T}_{\mathbf{V},\Omega}^2
\le
3\norm{\boldsymbol{E}_{\ysf}}_{\mathbf{V},\Omega}^2
+
\textup{\texttt{C}}_{\textsf{is}}^2\left(1+2\textup{\texttt{C}}_{\textsf{ct}}^2\right)\|\nabla\cdot\bar\ysf_\T\|_{L^2(\Omega)}^2.
\end{equation}

Furthermore, \eqref{eq:pbound} allows us to conclude that
\[
\norm{\hat{\psf}-\bar\psf_\T}_{Q,\Omega}^2
\leq 2\texttt{C}_{\textsf{is}}^2\left(\norm{\boldsymbol{E}_{\ysf}}_{\mathbf{V},\Omega}^2+\texttt{C}_{\textsf{ct}}^2\norm{\hat{\ysf}-\bar\ysf_\T}_{\mathbf{V},\Omega}^2\right).
\]
Applying \eqref{ymytbound} then yields that
\begin{equation}\label{pmptbound}
\norm{\hat{\psf}-\bar\psf_\T}_{Q,\Omega}^2
\leq 2\textup{\texttt{C}}_{\textsf{is}}^2\left(\left(1+
3\textup{\texttt{C}}_{\textsf{ct}}^2\right)\norm{\boldsymbol{E}_{\ysf}}_{\mathbf{V},\Omega}^2
+
\textup{\texttt{C}}_{\textsf{is}}^2\textup{\texttt{C}}_{\textsf{ct}}^2\left(1+2\textup{\texttt{C}}_{\textsf{ct}}^2\right)\|\nabla\cdot\bar\ysf_\T\|_{L^2(\Omega)}^2
\right).
\end{equation}

Now, letting $\xisf=\boldsymbol{E}_{\ysf}$ in \eqref{eq:Ey} yields that
\begin{eqnarray*}
\norm{\boldsymbol{E}_{\ysf}}_{\mathbf{V},\Omega}^2
&=&
\bia(\hat{\ysf}-\bar\ysf_\T,\boldsymbol{E}_{\ysf})-\bib(\boldsymbol{E}_{\ysf},\hat{\psf}-\bar\psf_\T)
\\
&=&\sum_{K\in\T}\left(
(\boldsymbol{\mathcal{R}}_{K}^{\textsf{st}} ,\boldsymbol{E}_{\ysf})_{\boldsymbol{L}^{2}(K)}
+\sum_{\gamma\in\mathcal{F}_{K}}
(\boldsymbol{g}_{\gamma,K}^{\textsf{st}}+\boldsymbol{\mathcal{R}}_{\gamma,K}^{\textsf{st}},\boldsymbol{E}_{\ysf})_{\boldsymbol{L}^{2}(\gamma)}+(\textbf{\textsf{osc}}_{K}^{\textsf{st}},\boldsymbol{E}_{\ysf})_{\boldsymbol{L}^{2}(K)}
\right)
\end{eqnarray*}
by \eqref{eq:y_hat-p_hat}, integration by parts, \eqref{eq:int_res_st}, \eqref{eq:osc_st}, \eqref{eq:jump_st} and \eqref{fluxsum}. Applying \eqref{sigmaidentity} and \eqref{projectionL2K} then yields that
\begin{eqnarray*}
\norm{\boldsymbol{E}_{\ysf}}_{\mathbf{V},\Omega}^2
&=&\sum_{K\in\T}\left(
(\boldsymbol{\sigma}_K^{\textsf{st}},\nabla\boldsymbol{E}_{\ysf})_{\underset{\approx}{\boldsymbol{L}}^{2}(K)}
+
(\textbf{\textsf{osc}}_{K}^{\textsf{st}},\boldsymbol{E}_{\ysf}-\Pi_{K,0}(\boldsymbol{E}_{\ysf}))_{\boldsymbol{L}^{2}(K)}
\right)
\\
&\le&\left(\sum_{K\in\T}\Psi_{\textsf{st},K}^2\right)^{1/2}
\norm{\boldsymbol{E}_{\ysf}}_{\mathbf{V},\Omega}
\end{eqnarray*}
by the Cauchy--Schwarz inequality and \eqref{Poincare}. Consequently,
\begin{equation}\label{Eybound}
\norm{\boldsymbol{E}_{\ysf}}_{\mathbf{V},\Omega}^2
\le
\sum_{K\in\T}\Psi_{\textsf{st},K}^2.
\end{equation}
The theorem then follows upon combining \eqref{ymytbound}, \eqref{pmptbound} and \eqref{Eybound}.
\end{proof}

We note that the above theorem is an improvement and adaptation to the case considered in this section of the results from \cite{MR3556402}. The below theorem can be proved similarly to how the above theorem was proved.

\begin{theorem}\label{th:fcwq}
\textbf{Assumption 4} holds with
\begin{equation}\label{etawK}
\eta_{\wsf,K}^2=
3\Psi_{\textup{\textsf{ad}},K}^2
+
\textup{\texttt{C}}_{\textup{\textsf{is}}}^2\left(1+2\textup{\texttt{C}}_{\textup{\textsf{ct}}}^2\right)\|\nabla\cdot\bar\wsf_\T\|_{L^2(K)}^2
\end{equation}
and
\begin{equation}\label{etaqK}
\eta_{\qsf,K}^2=
2\textup{\texttt{C}}_{\textup{\textsf{is}}}^2\left(\left(1+
3\textup{\texttt{C}}_{\textup{\textsf{ct}}}^2\right)\Psi_{\textup{\textsf{ad}},K}^2
+
\textup{\texttt{C}}_{\textup{\textsf{is}}}^2\textup{\texttt{C}}_{\textup{\textsf{ct}}}^2\left(1+2\textup{\texttt{C}}_{\textup{\textsf{ct}}}^2\right)\|\nabla\cdot\bar\wsf_\T\|_{L^2(K)}^2
\right).
\end{equation}
Moreover, \textbf{Assumption 2} holds with
\begin{equation}\label{etawq}
\eta_{\wsf}=\left(\sum_{K\in\mathcal{P}}\eta_{\wsf,K}^2\right)^{1/2}
\mbox{ and }
\eta_{\qsf}=\left(\sum_{K\in\mathcal{P}}\eta_{\qsf,K}^2\right)^{1/2}.
\end{equation}
\end{theorem}

We note that, if the value of a $\beta$ satisfying \eqref{inf-sup} is known, then \textbf{Assumption 5} holds with $\tilde{\eta}_{\ysf,K}=\eta_{\ysf,K}$, $\tilde{\eta}_{\psf,K}=\eta_{\psf,K}$, $\tilde{\eta}_{\wsf,K}=\eta_{\wsf,K}$ and $\tilde{\eta}_{\qsf,K}=\eta_{\qsf,K}$. Furthermore, by \eqref{fluxbound} and \eqref{sigmabound} we have that
\begin{align}
\eta_{\ysf,K}^2+\eta_{\psf,K}^2\lesssim &
\|\nabla\cdot\bar\ysf_\T\|_{L^{2}(K)}^{2}
+
\sum_{\gamma\in\hat{\mathcal{F}}_K}h_K\|\llbracket \boldsymbol{\mathcal{R}}_{\gamma}^{\textsf{st}} \rrbracket\|_{\boldsymbol{L}^{2}(\gamma)}^{2}\nonumber
\\
&
+
\sum_{K'\in\hat{\T}_K}h_K^{2}\left(\|\boldsymbol{\mathcal{R}}_{K'}^{\textsf{st}}\|_{\boldsymbol{L}^{2}(K')}^{2}+\|\textbf{\textsf{osc}}_{K'}^{\textsf{st}}\|_{\boldsymbol{L}^{2}(K')}^{2}\right)\label{ypeff}
\end{align}
and
\begin{align}
\eta_{\wsf,K}^2+\eta_{\qsf,K}^2\lesssim &
\|\nabla\cdot\bar\wsf_\T\|_{L^{2}(K)}^{2}
+
\sum_{\gamma\in\hat{\mathcal{F}}_K}h_K\|\llbracket \boldsymbol{\mathcal{R}}_{\gamma}^{\textsf{ad}} \rrbracket\|_{\boldsymbol{L}^{2}(\gamma)}^{2}\nonumber
\\
&
+
\sum_{K'\in\hat{\T}_K}h_K^{2}\left(\|\boldsymbol{\mathcal{R}}_{K'}^{\textsf{ad}}\|_{\boldsymbol{L}^{2}(K')}^{2}+\|\textbf{\textsf{osc}}_{K'}^{\textsf{ad}}\|_{\boldsymbol{L}^{2}(K')}^{2}\right)\label{wqeff}
\end{align}
from which it follows that \textbf{Assumption 7} is also satisfied. We note that it also follows that
\begin{align}
\normv{\hat{\ysf}-\bar{\ysf}_{\T}}_{\mathbf{V},\Omega}^2+\normp{\hat{\psf}-\bar{\psf}_{\T}}_{Q,\Omega}^2
\lesssim &
\sum_{K\in\T}\bigg{(}\|\nabla\cdot\bar\ysf_\T\|_{L^{2}(K)}^{2}
+\sum_{\gamma\in\mathcal{F}_K}h_K\|\llbracket \boldsymbol{\mathcal{R}}_{\gamma}^{\textsf{st}} \rrbracket\|_{\boldsymbol{L}^{2}(\gamma)}^{2}\nonumber
\\
&+
h_K^{2}\left(\|\boldsymbol{\mathcal{R}}_{K}^{\textsf{st}}\|_{\boldsymbol{L}^{2}(K)}^{2}+\|\textbf{\textsf{osc}}_{K}^{\textsf{st}}\|_{\boldsymbol{L}^{2}(K)}^{2}\right)\bigg{)}\label{ypglobaleff}
\end{align}
and
\begin{align}
\normv{\hat{\wsf}-\bar{\wsf}_{\T}}_{\mathbf{V},\Omega}^2+\normp{\hat{\qsf}-\bar{\qsf}_{\T}}_{Q,\Omega}^2
\lesssim &
\sum_{K\in\T}\bigg{(}\|\nabla\cdot\bar\wsf_\T\|_{L^{2}(K)}^{2}
+\sum_{\gamma\in\mathcal{F}_K}h_K\|\llbracket \boldsymbol{\mathcal{R}}_{\gamma}^{\textsf{ad}} \rrbracket\|_{\boldsymbol{L}^{2}(\gamma)}^{2}\nonumber
\\
&+
h_K^{2}\left(\|\boldsymbol{\mathcal{R}}_{K}^{\textsf{ad}}\|_{\boldsymbol{L}^{2}(K)}^{2}+\|\textbf{\textsf{osc}}_{K}^{\textsf{ad}}\|_{\boldsymbol{L}^{2}(K)}^{2}\right)\bigg{)}\label{wqglobaleff}
\end{align}

\subsection{Residual--based a posteriori error estimators}\label{sec:alternative}

From \eqref{ypglobaleff} and \eqref{wqglobaleff} the following result follows.

\begin{theorem}\label{th:rbypwq}
Let
\begin{align}
\tilde{\eta}_{\ysf,K}^2=\tilde{\eta}_{\psf,K}^2=&
\|\nabla\cdot\bar\ysf_\T\|_{L^{2}(K)}^{2}
+\sum_{\gamma\in\mathcal{F}_K}h_K\|\llbracket \boldsymbol{\mathcal{R}}_{\gamma}^{\textup{\textsf{st}}} \rrbracket\|_{\boldsymbol{L}^{2}(\gamma)}^{2}\nonumber
\\
&+
h_K^{2}\left(\|\boldsymbol{\mathcal{R}}_{K}^{\textup{\textsf{st}}}\|_{\boldsymbol{L}^{2}(K)}^{2}+\|\textbf{\textsf{osc}}_{K}^{\textup{\textsf{st}}}\|_{\boldsymbol{L}^{2}(K)}^{2}\right),
\end{align}
\begin{align}
\tilde{\eta}_{\wsf,K}^2=\tilde{\eta}_{\qsf,K}^2=&
\|\nabla\cdot\bar\wsf_\T\|_{L^{2}(K)}^{2}
+\sum_{\gamma\in\mathcal{F}_K}h_K\|\llbracket \boldsymbol{\mathcal{R}}_{\gamma}^{\textup{\textsf{ad}}} \rrbracket\|_{\boldsymbol{L}^{2}(\gamma)}^{2}\nonumber
\\
&+
h_K^{2}\left(\|\boldsymbol{\mathcal{R}}_{K}^{\textup{\textsf{ad}}}\|_{\boldsymbol{L}^{2}(K)}^{2}+\|\textbf{\textsf{osc}}_{K}^{\textup{\textsf{ad}}}\|_{\boldsymbol{L}^{2}(K)}^{2}\right),
\end{align}
$\eta_{\ysf,K}=C\tilde{\eta}_{\ysf,K}$, $\eta_{\psf,K}=C\tilde{\eta}_{\psf,K}$, $\eta_{\wsf,K}=C\tilde{\eta}_{\wsf,K}$, $\eta_{\qsf,K}=C\tilde{\eta}_{\qsf,K}$,
\begin{equation*}
\eta_{\ysf}=\eta_{\psf}=\sum_{K\in\T}\eta_{\ysf,K}^2=\sum_{K\in\T}\eta_{\psf,K}^2,
\quad
\eta_{\wsf}=\eta_{\qsf}=\sum_{K\in\T}\eta_{\wsf,K}^2=\sum_{K\in\T}\eta_{\qsf,K}^2,
\end{equation*}
and $\hat{\T}_K=\hat{\mathcal{F}}_K=\{K\}$, where $C$ is a positive constant that is independent of the size of the elements in the mesh. Then \textbf{Assumptions 1}, \textbf{2}, \textbf{3}, \textbf{4}, \textbf{5} and \textbf{7} hold.
\end{theorem}

\section{Numerical examples}\label{sec:numex}

We performed numerical examples using the approximation method described in section \ref{particular} with $\tau_{K}=h_K^2$ for all $K\in\T$ and $\tau_{\gamma}=1$ for all $\gamma\in\mathcal{F}_{I}$. We considered $\vartheta=1$ and $\varrho=1$. The number of degrees of freedom $\textrm{Ndof}=2dN_v+(d+2)N_e$, where $N_v$ is the number of vertices in the mesh and $N_e$ is the number of elements in the mesh.

\subsection{Two dimensional examples}

We perform two dimensional examples on polygonal domains for which the value of a $\beta$ satisfying \eqref{inf-sup} is known. After obtaining the approximate solution, the a posteriori error estimator $\Upsilon$ from Theorem \ref{th:global_reliability} was computed with the aid of Theorems \ref{th:fcyp} and \ref{th:fcwq}. We note that the estimator $\Upsilon$ provides a guaranteed upper bound on $\norm{(\boldsymbol{\esf}_{\ysf},\esf_{\psf},\boldsymbol{\esf}_{\wsf},\esf_{\qsf},\boldsymbol{\esf}_{\usf})}_{\Omega}$. The local error indicators $\Upsilon_K$ from Theorem \ref{th:global_reliability2} were also computed, again with the aid of Theorems \ref{th:fcyp} and \ref{th:fcwq}. Each mesh $\T$ was adaptively refined by marking for refinement the elements $K\in\T$ that were such that $\Upsilon_K^2\ge N_e^{-1} \sum_{K'\in\T}\Upsilon_{K'}^2$. In this way a sequence of adaptively refined meshes was generated from the initial meshes shown in Figure \ref{Fig:initial2D}.

\begin{figure}[!htbp]
\begin{center}
\scalebox{0.2}{\includegraphics{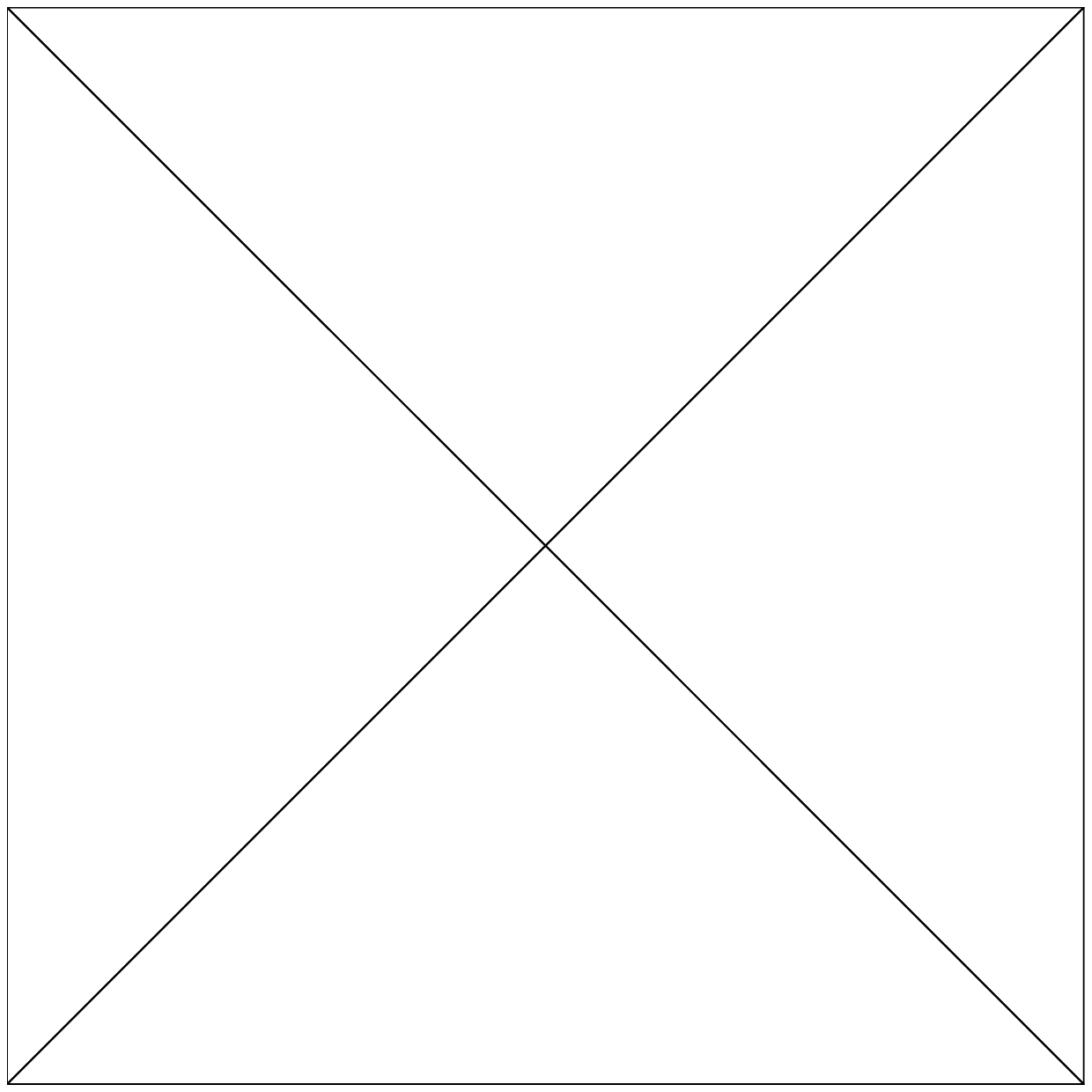}}
\scalebox{0.2}{\includegraphics{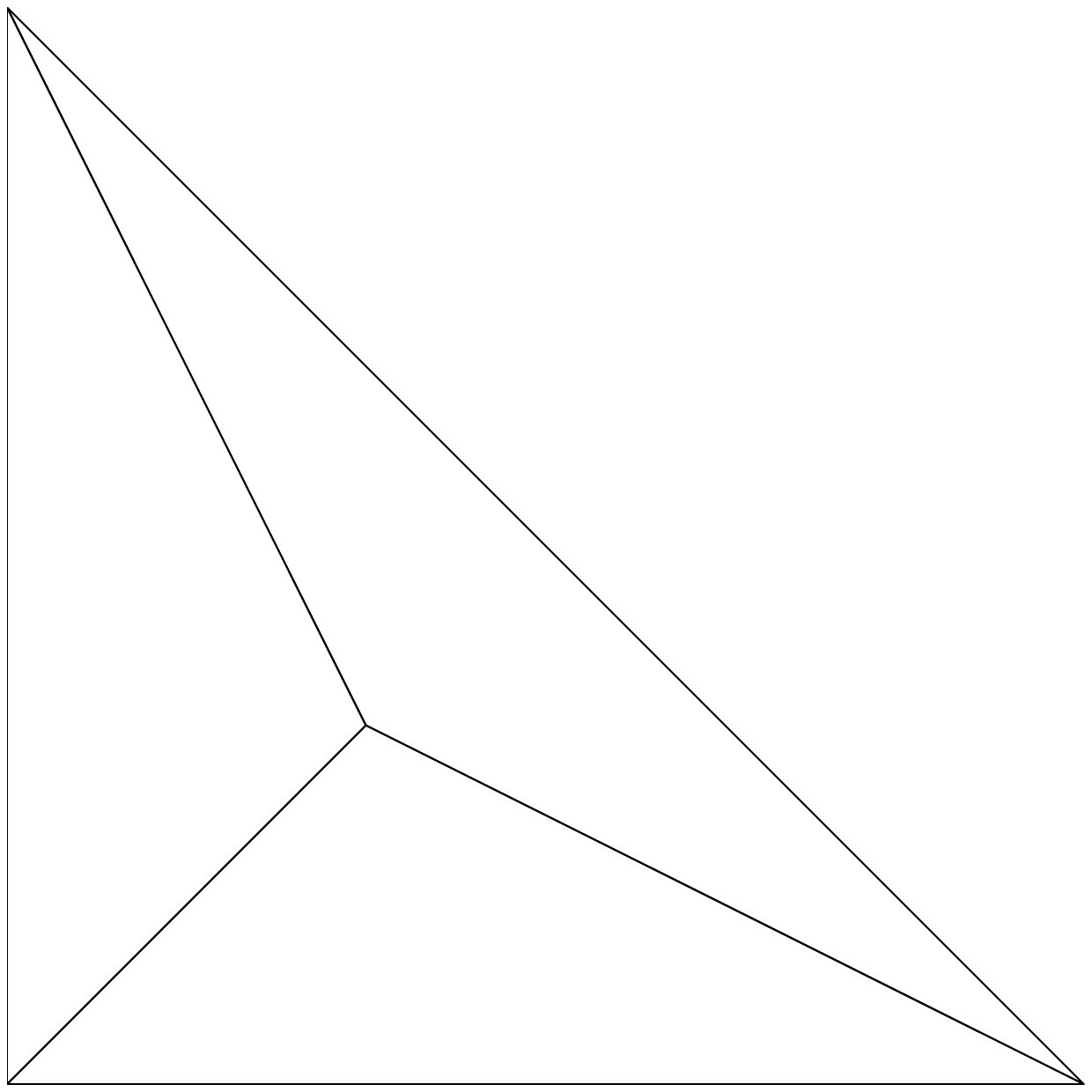}}
\scalebox{0.2}{\includegraphics{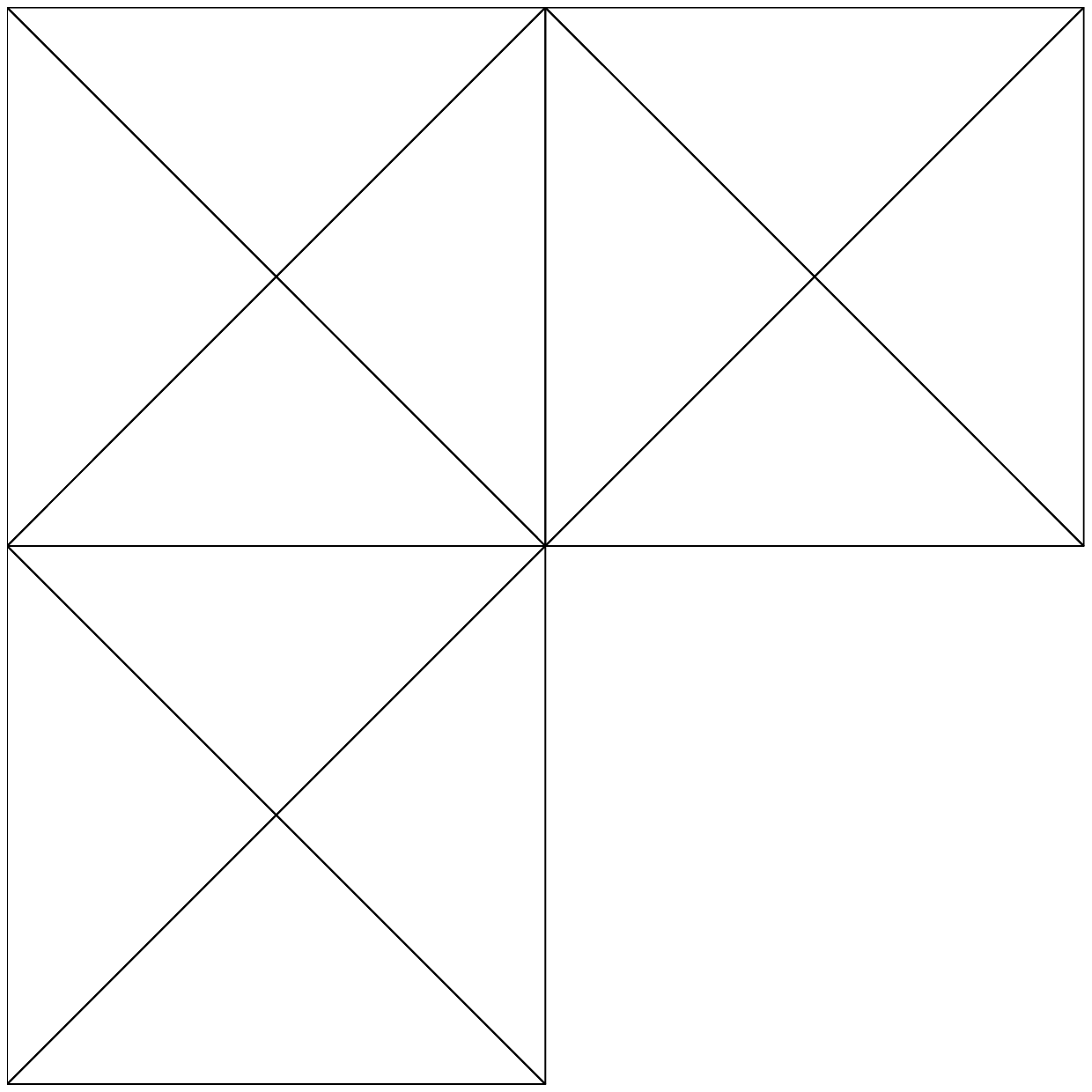}}
\scalebox{0.2}{\includegraphics{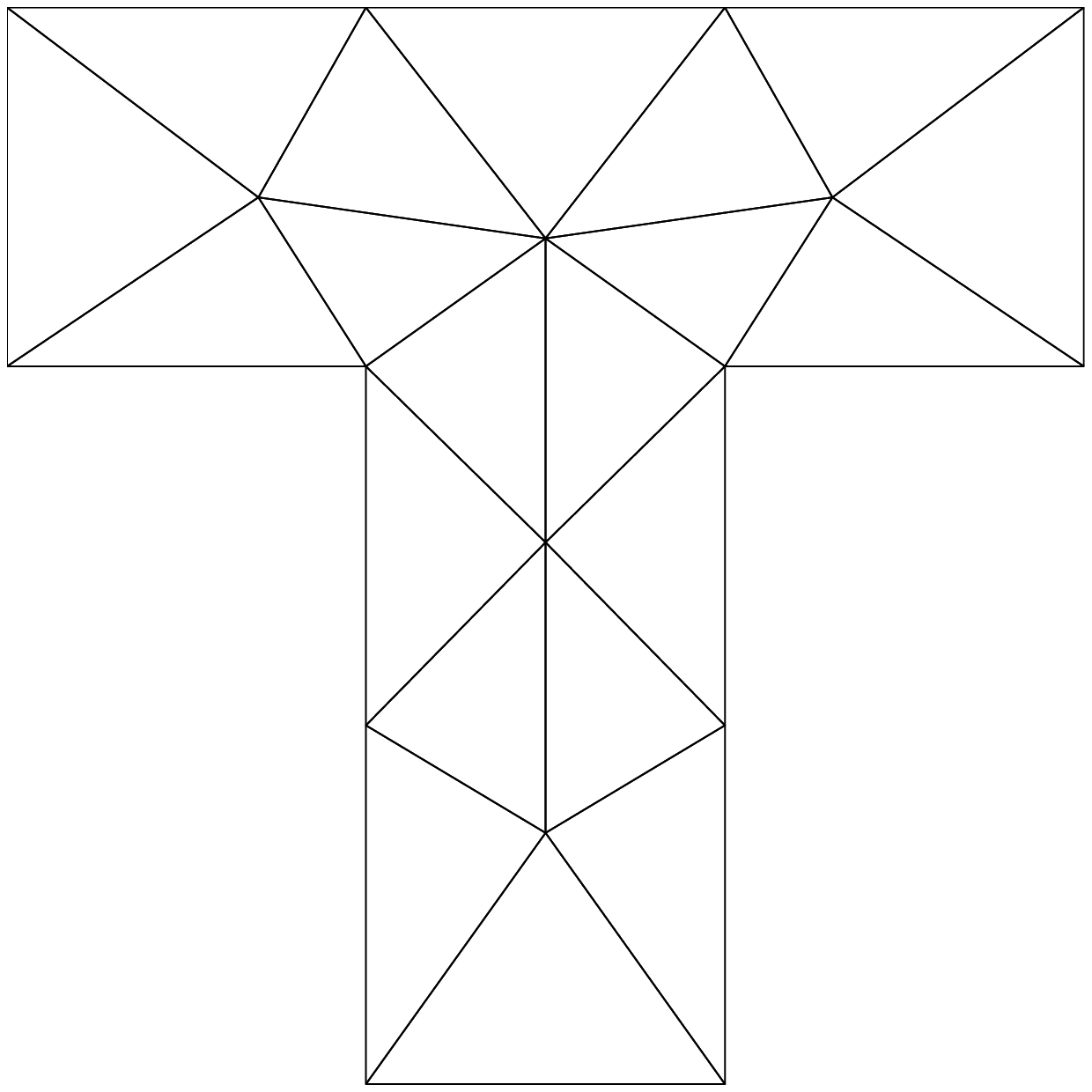}}
\end{center}
\caption{The initial meshes used for Examples \ref{bubble2D}, \ref{layer2D}, \ref{L2D} and \ref{T2D}.}
\label{Fig:initial2D}
\end{figure}

\begin{example}\label{bubble2D}
We consider the square domain $\Omega=(0,1)^2$. From \cite{MR1713178} we have that \eqref{inf-sup} holds with $\beta=\sin(\pi/8)$. We took $\varepsilon=1$, $\csf(x_1,x_2)=(x_2,-x_1)$, $\kappa=1$, $\aasf=(-0.5,-0.5)$ and $\bbsf=(0.5,0.5)$. The data $\fsf$ and $\ysf_{\Omega}$ were chosen to be such that
\[
\bar{\ysf}(x_1,x_2)=\mathbf{curl}\left((x_1(1-x_1)x_2(1-x_2))^2\right), \quad \bar{\psf}(x_1,x_2)=\cos(2\pi x_1)\cos(2\pi x_2),
\]
\[
\bar{\wsf}(x_1,x_2)=\mathbf{curl}\left((\sin(2\pi x_1)\sin(2\pi x_2))^2\right), \quad \bar{\qsf}(x_1,x_2)=\sin(2\pi x_1)\sin(2\pi x_2).
\]
The results are shown in Figure \ref{Fig:bubble2D}. We observe that the error $\norm{(\boldsymbol{\esf}_{\ysf},\esf_{\psf},\boldsymbol{\esf}_{\wsf},\esf_{\qsf},\boldsymbol{\esf}_{\usf})}_{\Omega}$ and the estimator $\Upsilon$ are decreasing at the optimal rate.
\end{example}

\begin{figure}[!htbp]
\begin{center}
\psfrag{Ndof}{\huge Ndof}
\psfrag{Total error eyepeweqeu}{\LARGE $\norm{(\boldsymbol{\esf}_{\ysf},\esf_{\psf},\boldsymbol{\esf}_{\wsf},\esf_{\qsf},\boldsymbol{\esf}_{\usf})}_{\Omega}$}
\psfrag{Total estimator}{\LARGE $\Upsilon$}
\psfrag{Optimal rate}{\LARGE Ndof$^{-1/2}$}
\scalebox{0.4}{\includegraphics{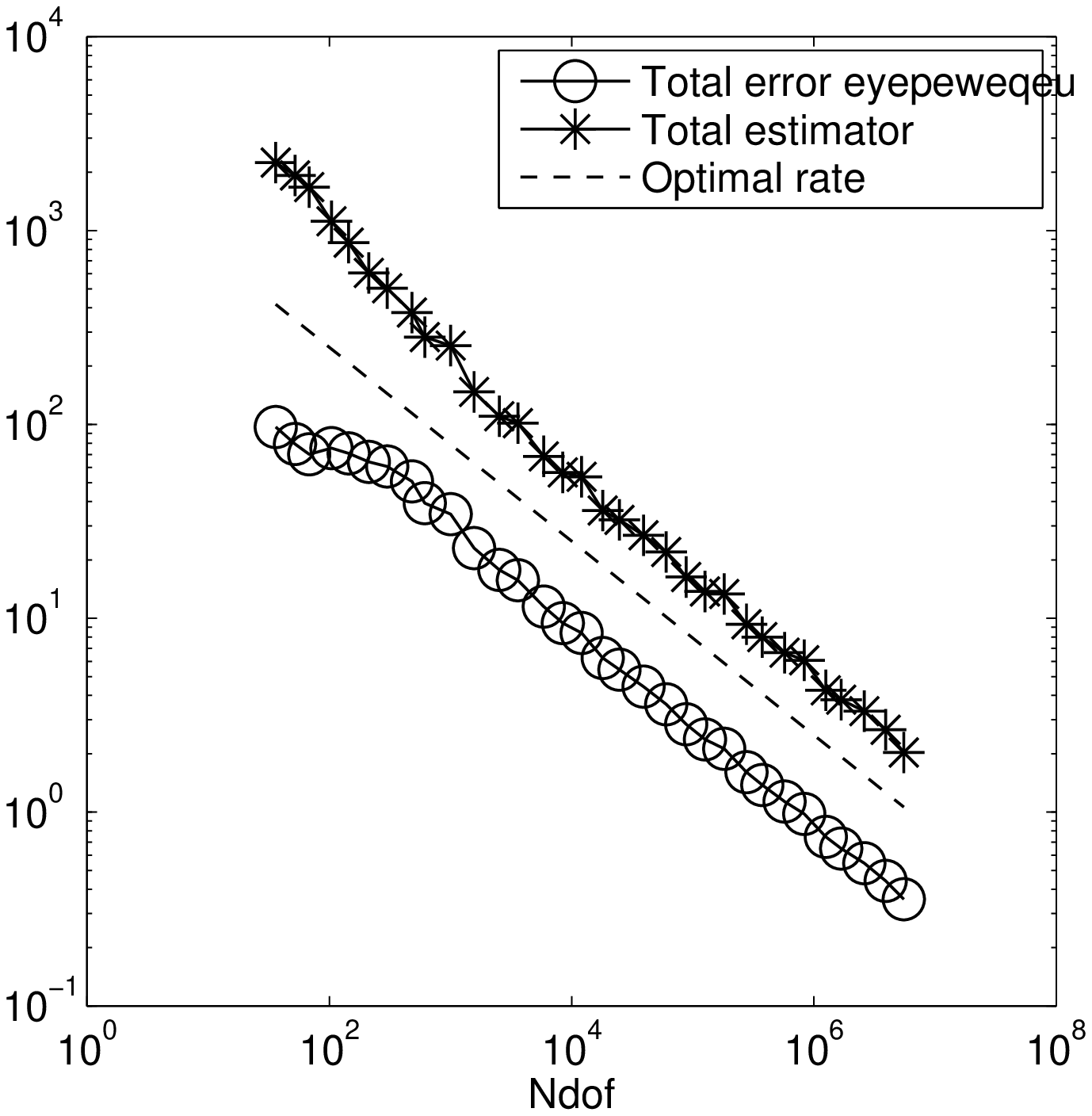}}
\scalebox{0.4}{\includegraphics{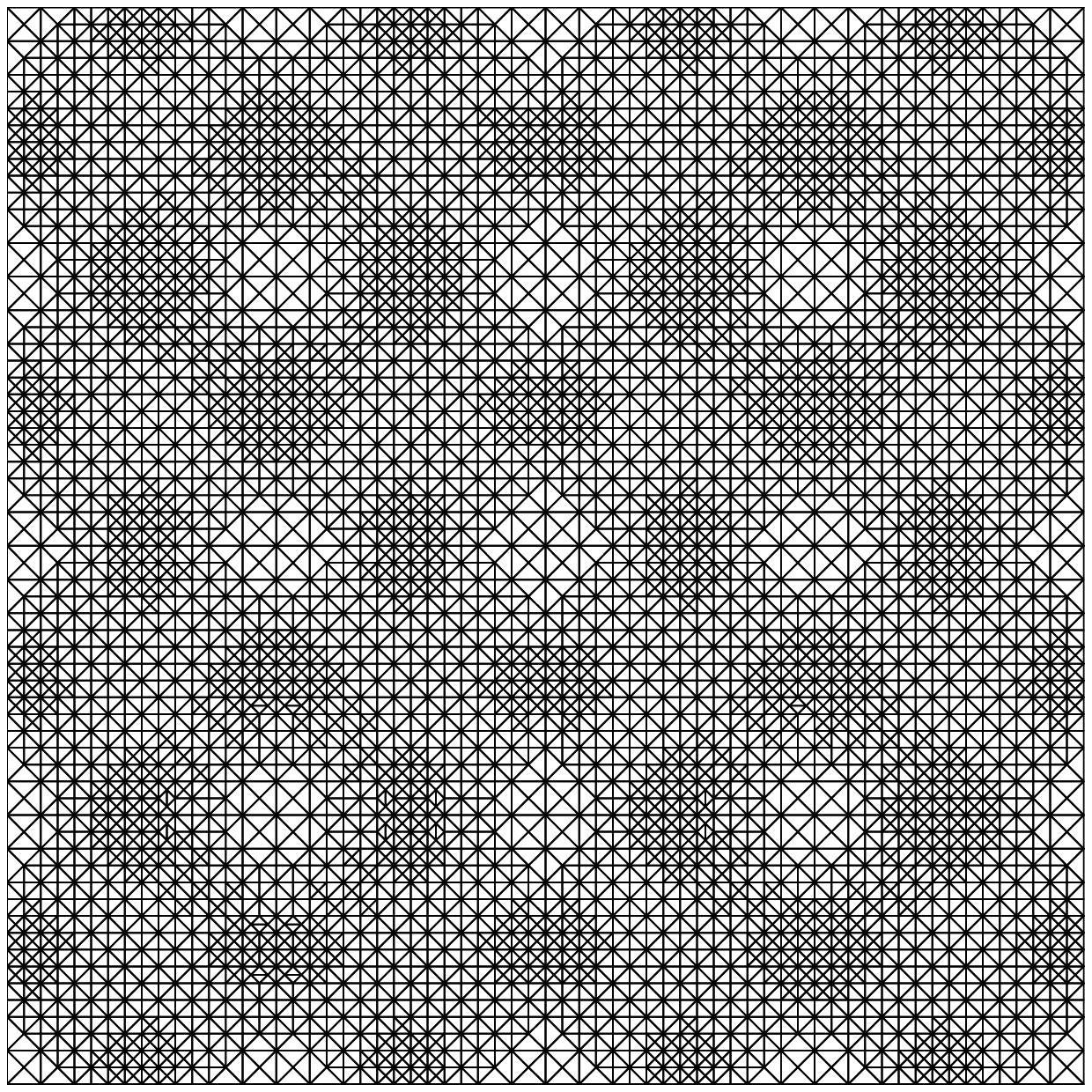}}
\end{center}
\caption{Example \ref{bubble2D}: The error $\norm{(\boldsymbol{\esf}_{\ysf},\esf_{\psf},\boldsymbol{\esf}_{\wsf},\esf_{\qsf},\boldsymbol{\esf}_{\usf})}_{\Omega}$ and estimator $\Upsilon$ (left) and the 19th adaptively refined mesh (right).}
\label{Fig:bubble2D}
\end{figure}

\begin{example}\label{layer2D}
We consider the triangular domain $\Omega=\{(x_1,x_2):x_1>0,x_2>0,x_1+x_2<1\}$. From \cite{MR1713178} we have that \eqref{inf-sup} holds with $\beta=\sin(\pi/16)$. We took $\varepsilon=0.01$, $\csf=(0,0)$, $\kappa=1$, $\aasf=(0,0)$ and $\bbsf=(0.1,0.1)$. The data $\fsf$ and $\ysf_{\Omega}$ were chosen to be such that
\[
\bar{\ysf}(x_1,x_2)=\mathbf{curl}\left(x_1x_2^2(1-x_1-x_2)^2
\left(1-x_1-\frac{\exp(-100x_1)-\exp(-100)}{1-\exp(-100)}\right)\right),
\]
\[
\bar{\psf}(x_1,x_2)=\cos(2\pi x_2)/1024,
\]
\[
\bar{\wsf}(x_1,x_2)=\mathbf{curl}\left(x_1^2x_2(1-x_1-x_2)^2
\left(1-x_2-\frac{\exp(-100x_2)-\exp(-100)}{1-\exp(-100)}\right)\right),
\]
and
\[
\bar{\qsf}(x_1,x_2)=\cos(2\pi x_1)/1024.
\]
The results are shown in Figure \ref{Fig:layer2D}. We observe that, once the mesh has been sufficiently refined, the error $\norm{(\boldsymbol{\esf}_{\ysf},\esf_{\psf},\boldsymbol{\esf}_{\wsf},\esf_{\qsf},\boldsymbol{\esf}_{\usf})}_{\Omega}$ and the estimator $\Upsilon$ decrease at the optimal rate. We also observe that more refinement has been performed in the regions where the solution has boundary layers.
\end{example}

\begin{figure}[!htbp]
\begin{center}
\psfrag{Ndof}{\huge Ndof}
\psfrag{Total error eyepeweqeu}{\LARGE $\norm{(\boldsymbol{\esf}_{\ysf},\esf_{\psf},\boldsymbol{\esf}_{\wsf},\esf_{\qsf},\boldsymbol{\esf}_{\usf})}_{\Omega}$}
\psfrag{Total estimator}{\LARGE $\Upsilon$}
\psfrag{Optimal rate}{\LARGE Ndof$^{-1/2}$}
\scalebox{0.4}{\includegraphics{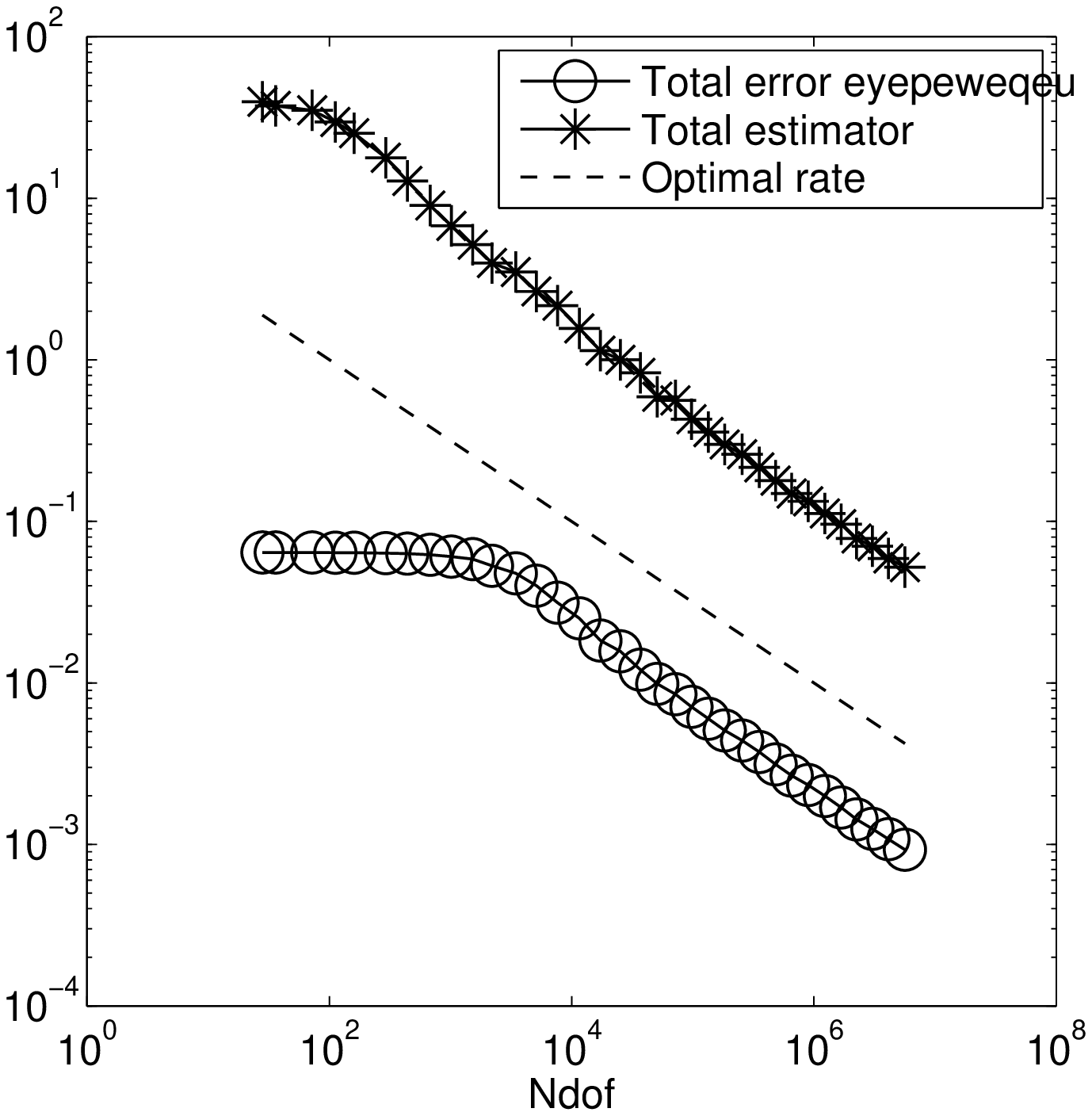}}
\scalebox{0.4}{\includegraphics{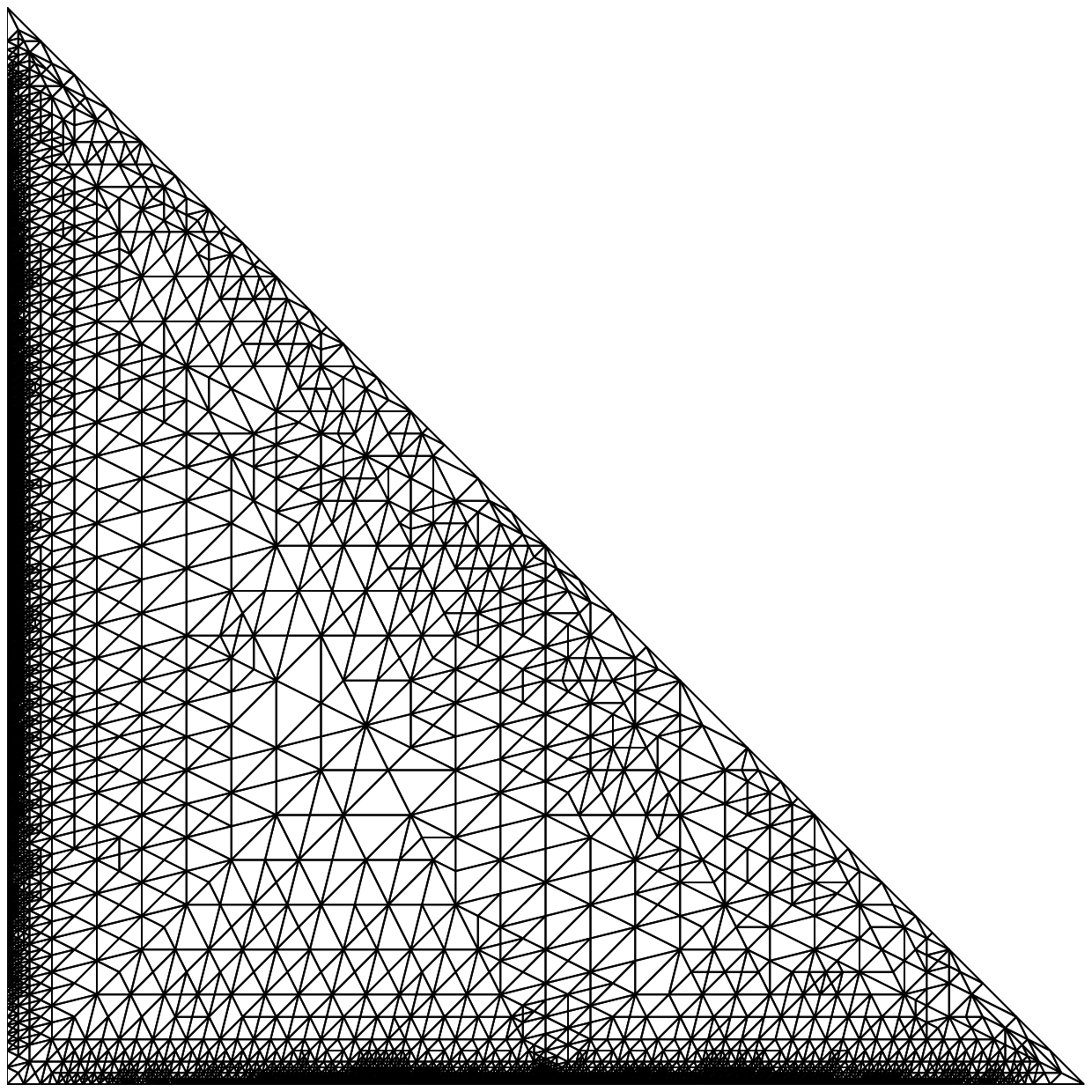}}
\end{center}
\caption{Example \ref{layer2D}: The error $\norm{(\boldsymbol{\esf}_{\ysf},\esf_{\psf},\boldsymbol{\esf}_{\wsf},\esf_{\qsf},\boldsymbol{\esf}_{\usf})}_{\Omega}$ and estimator $\Upsilon$ (left) and the 19th adaptively refined mesh (right).}
\label{Fig:layer2D}
\end{figure}

\begin{example}\label{L2D}
We consider the L-shaped domain $\Omega=(-1,1)^2\setminus([0,1)\times(-1,0])$. From \cite{MR1713178} we have that \eqref{inf-sup} holds with $\beta=0.1601$. We took $\varepsilon=1$, $\csf=(0,0)$, $\kappa=0$, $\aasf=(0,0)$, $\bbsf=(1,1)$, $\fsf=(1,1)$ and $\ysf_{\Omega}(x_1,x_2)=(x_2,-x_1)$. The results are shown in Figure \ref{Fig:L2D}. We observe that the estimator $\Upsilon$ decreases at the optimal rate and that more refinement is being performed in regions close to the reentrant corner. The true solution to this problem is unknown and hence we cannot compute $\norm{(\boldsymbol{\esf}_{\ysf},\esf_{\psf},\boldsymbol{\esf}_{\wsf},\esf_{\qsf},\boldsymbol{\esf}_{\usf})}_{\Omega}$. However, from Theorem \ref{th:global_reliability} we know that $\norm{(\boldsymbol{\esf}_{\ysf},\esf_{\psf},\boldsymbol{\esf}_{\wsf},\esf_{\qsf},\boldsymbol{\esf}_{\usf})}_{\Omega}\le\Upsilon$.
\end{example}

\begin{figure}[!htbp]
\begin{center}
\psfrag{Ndof}{\huge Ndof}
\psfrag{Estimator}{\LARGE $\Upsilon$}
\psfrag{Optimal rate}{\LARGE Ndof$^{-1/2}$}
\scalebox{0.4}{\includegraphics{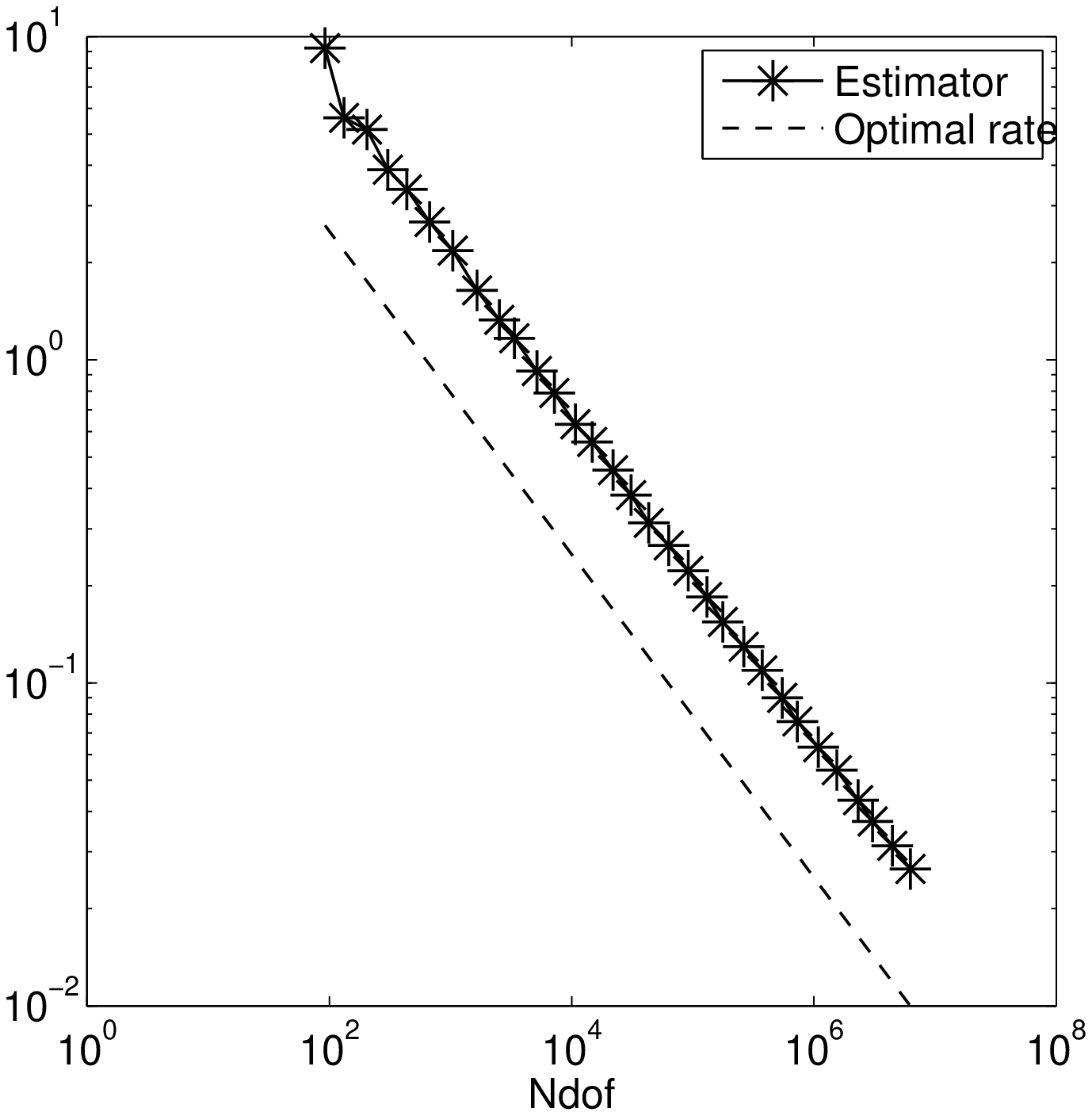}}
\scalebox{0.4}{\includegraphics{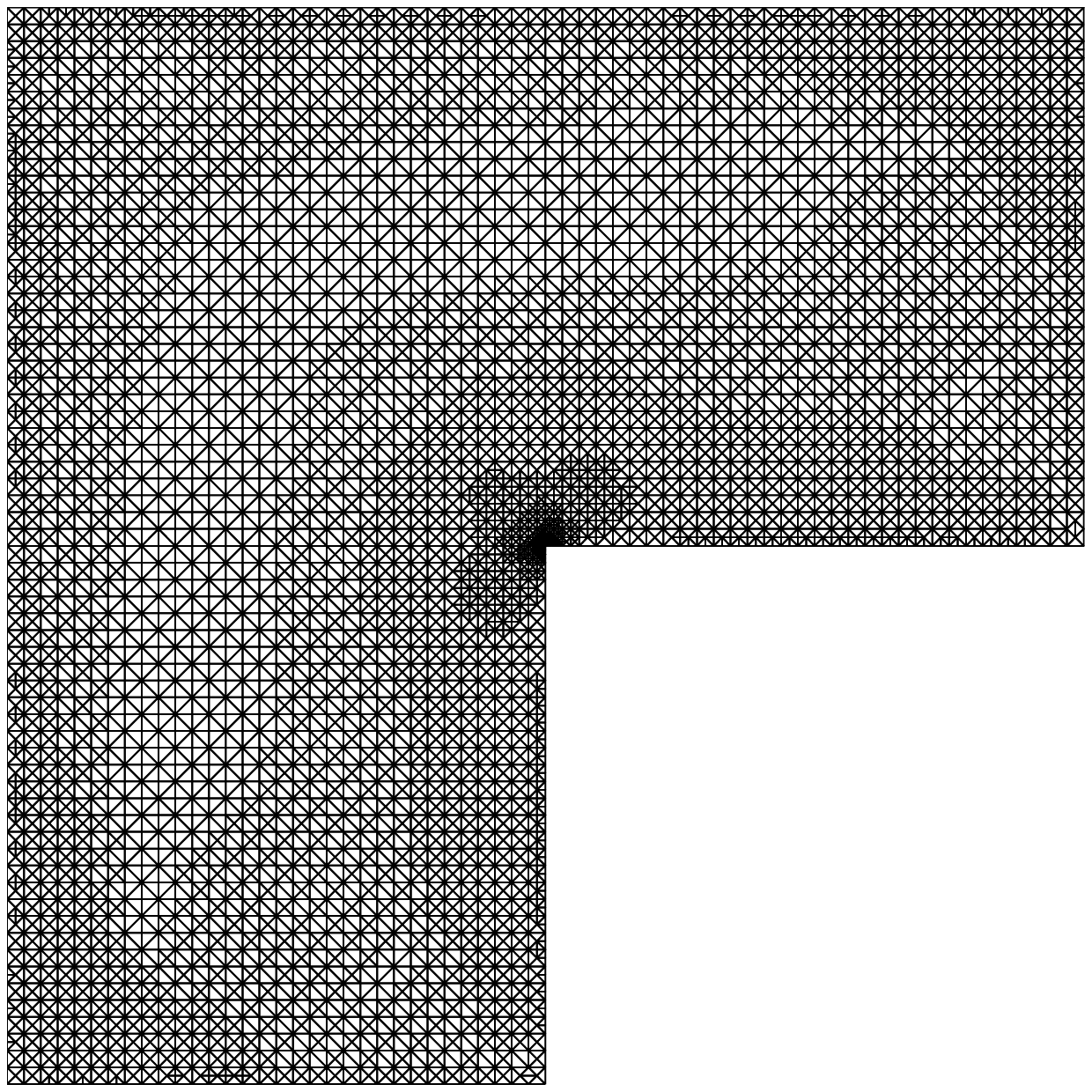}}
\end{center}
\caption{Example \ref{L2D}: The estimator $\Upsilon$ (left) and the 17th adaptively refined mesh (right).}
\label{Fig:L2D}
\end{figure}

\begin{example}\label{T2D}
We considered the same problem as in the previous example with the exception that we took the domain to be the T-shaped domain $\Omega=((-1.5,1.5)\times(0,1))\cup((-0.5,0.5)\times(-2,0])$ on which we have that \eqref{inf-sup} holds with $\beta=0.1076$ from \cite{MR1713178}. The results are shown in Figure \ref{Fig:T2D}. Similar observations to those made about the previous example can be made.
\end{example}

\begin{figure}[!htbp]
\begin{center}
\psfrag{Ndof}{\huge Ndof}
\psfrag{Estimator}{\LARGE $\Upsilon$}
\psfrag{Optimal rate}{\LARGE Ndof$^{-1/2}$}
\scalebox{0.4}{\includegraphics{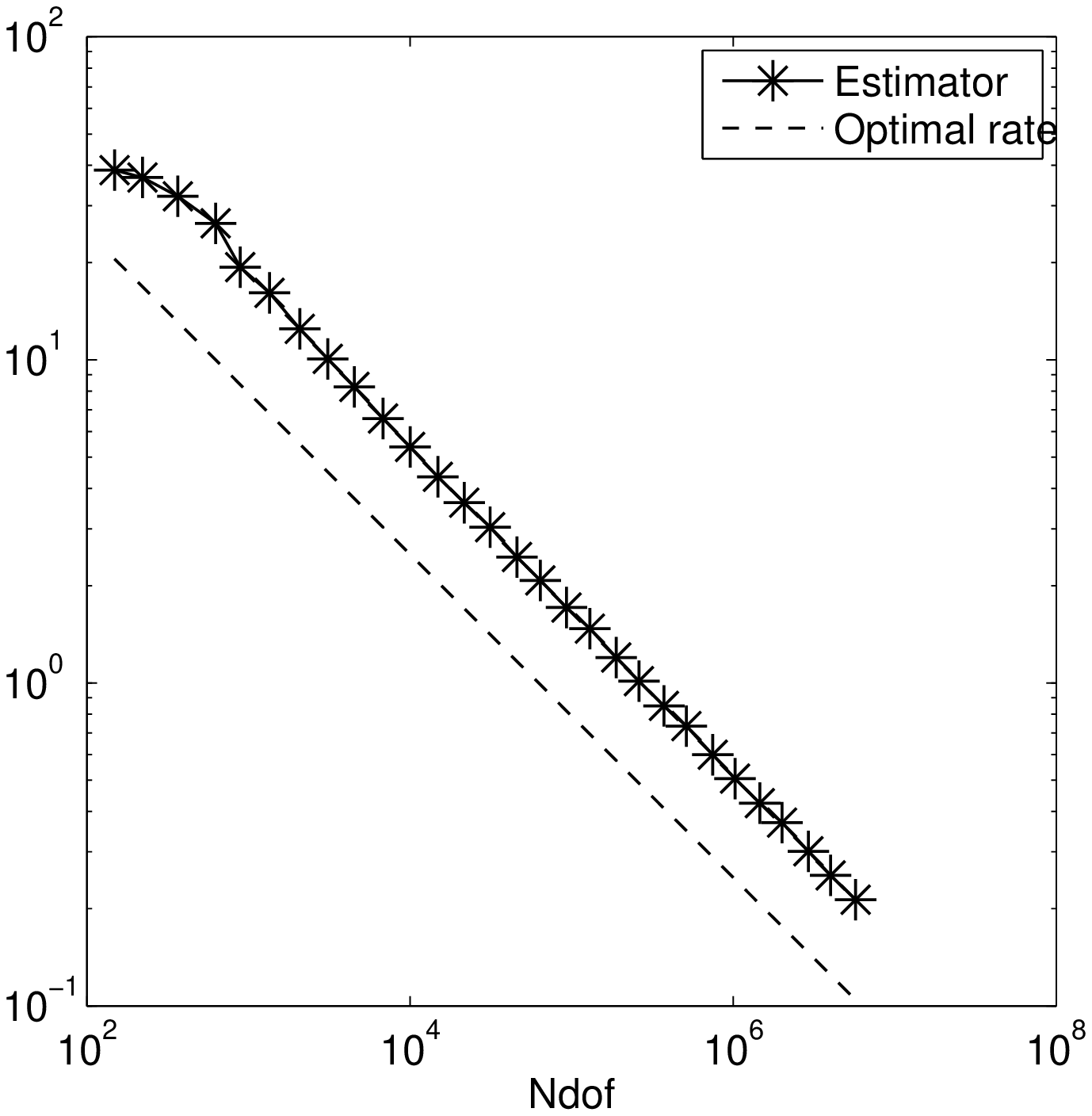}}
\scalebox{0.4}{\includegraphics{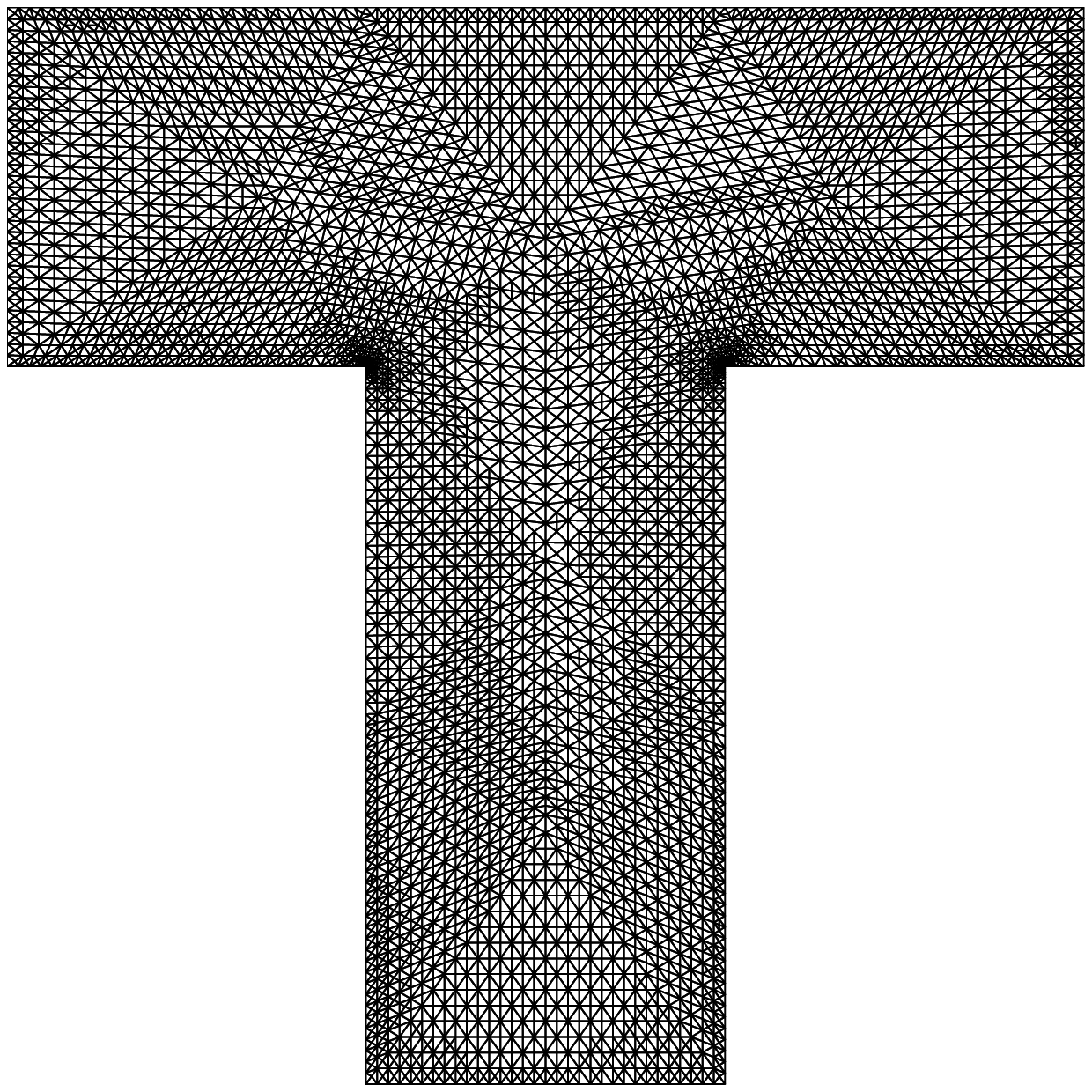}}
\end{center}
\caption{Example \ref{T2D}: The estimator $\Upsilon$ (left) and the 15th adaptively refined mesh (right).}
\label{Fig:T2D}
\end{figure}

\subsection{Three dimensional examples}

Unfortunately, we are not aware of any polyhedral domains for which the value of a $\beta$ satisfying \eqref{inf-sup} is known. Hence, when the domain is three dimensional, the estimator $\Upsilon$ from Theorem \ref{th:global_reliability} and the local error indicators $\Upsilon_K$ from Theorem \ref{th:global_reliability2} are not computable. Consequently, after obtaining the approximate solution, the a posteriori error estimator $\tilde{\Upsilon}$ and the local error indicators $\tilde{\Upsilon}_K$ from Theorem \ref{th:global_reliability3} were computed, with the aid of Theorem \ref{th:rbypwq}. Each mesh $\T$ was adaptively refined by marking for refinement the elements $K\in\T$ that were such that $\tilde{\Upsilon}_K^2\ge N_e^{-1} \sum_{K'\in\T}\tilde{\Upsilon}_{K'}^2$. In this way a sequence of adaptively refined meshes was generated from the initial meshes shown in Figure \ref{Fig:initial3D}. We note that we have not proved that the estimator $\tilde{\Upsilon}$ provides a guaranteed upper bound on $\norm{(\boldsymbol{\esf}_{\ysf},\esf_{\psf},\boldsymbol{\esf}_{\wsf},\esf_{\qsf},\boldsymbol{\esf}_{\usf})}_{\Omega}$. However, from Theorem \ref{th:global_reliability3} we know that $\norm{(\boldsymbol{\esf}_{\ysf},\esf_{\psf},\boldsymbol{\esf}_{\wsf},\esf_{\qsf},\boldsymbol{\esf}_{\usf})}_{\Omega}\lesssim\tilde{\Upsilon}$.

\begin{figure}[!htbp]
\begin{center}
\scalebox{0.2}{\includegraphics{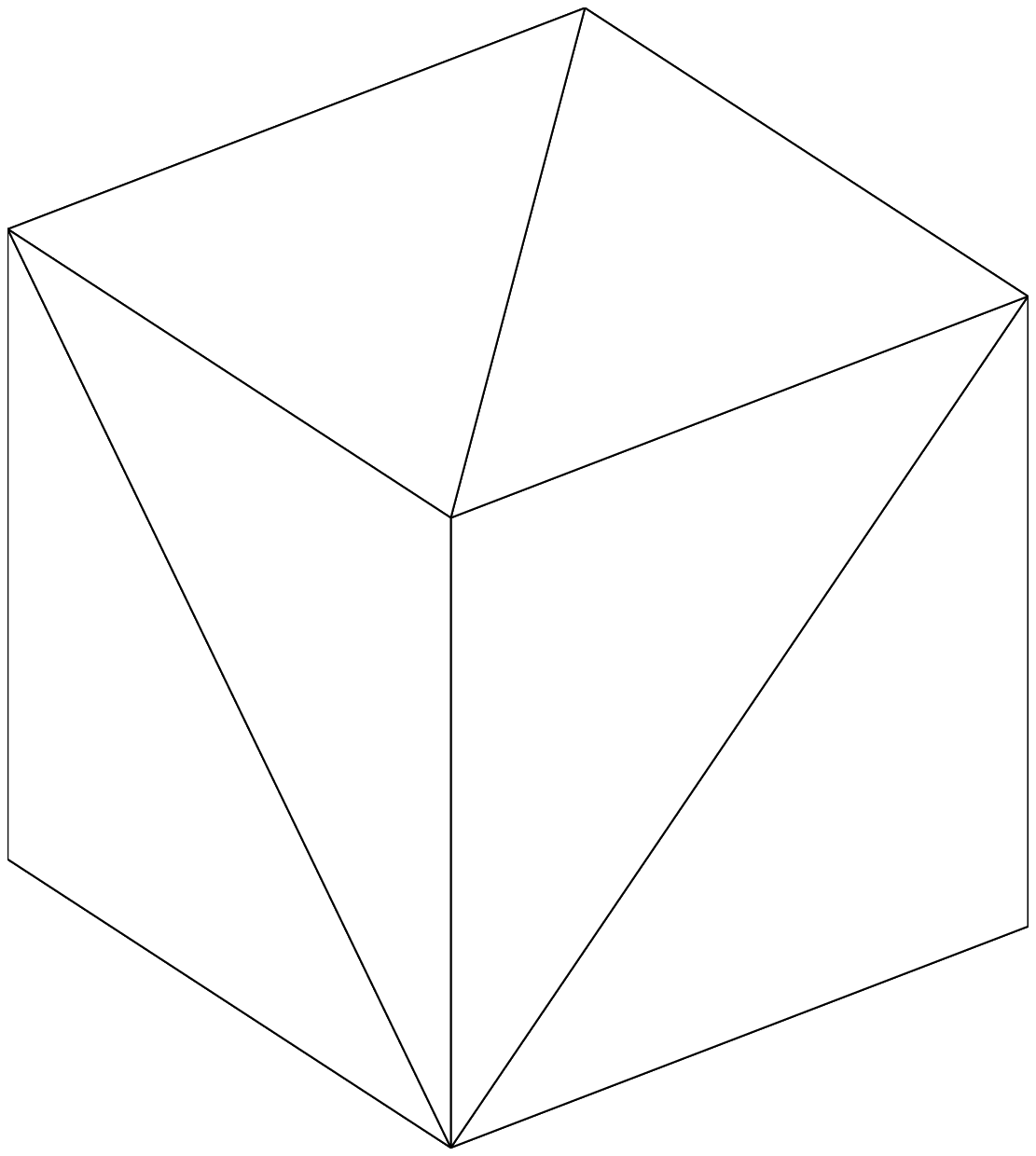}}
\scalebox{0.2}{\includegraphics{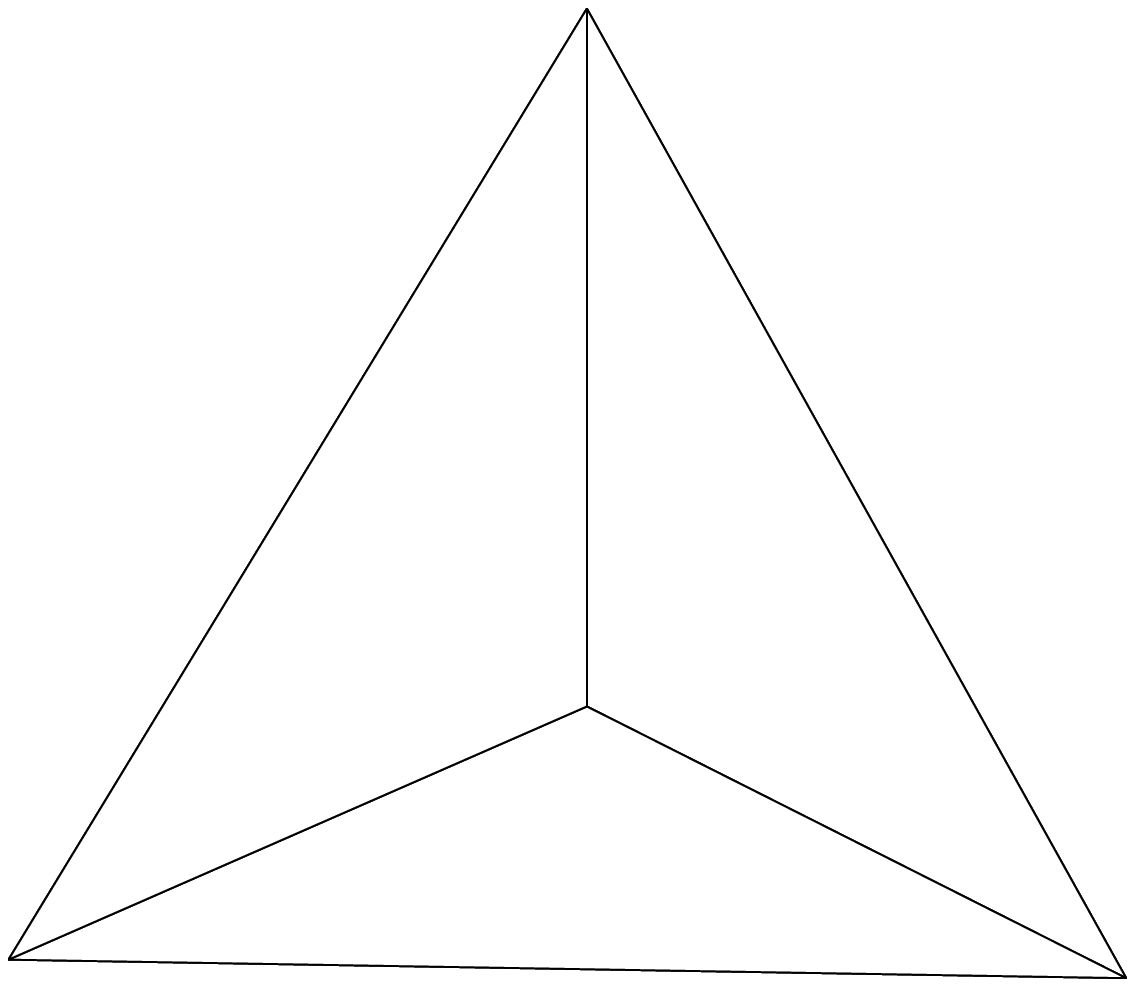}}
\end{center}
\caption{Exterior views of the initial meshes used for Examples \ref{bubble3D} and \ref{layer3D}.}
\label{Fig:initial3D}
\end{figure}

\begin{example}\label{bubble3D}
We consider the cuboidal domain $\Omega=(0,1)^3$. We took $\varepsilon=1$, $\csf(x_1,x_2,x_3)=(x_2-x_3, x_3-x_1, x_1-x_2)$, $\kappa=1$, $\aasf=(-0.5,-0.5,-0.5)$ and $\bbsf=(0.5,0.5,0.5)$. The data $\fsf$ and $\ysf_{\Omega}$ were chosen to be such that
\[
\bar{\ysf}(x_1,x_2,x_3)=\mathbf{curl}\left((x_1(1-x_1)x_2(1-x_2)x_3(1-x_3))^2\right), \, \,  \bar{\psf}(x_1,x_2,x_3)=\cos(2\pi x_3),
\]
\[
\bar{\wsf}(x_1,x_2,x_3)=\mathbf{curl}\left((\sin(2\pi x_1)\sin(2\pi x_2)\sin(2\pi x_3))^2\right), \, \,  \bar{\qsf}(x_1,x_2,x_3)=\sin(2\pi x_3).
\]
The results are shown in Figure \ref{Fig:bubble3D}. We observe that the error $\norm{(\boldsymbol{\esf}_{\ysf},\esf_{\psf},\boldsymbol{\esf}_{\wsf},\esf_{\qsf},\boldsymbol{\esf}_{\usf})}_{\Omega}$ and the estimator $\tilde{\Upsilon}$ are decreasing at the optimal rate.
\end{example}

\begin{figure}[!htbp]
\begin{center}
\psfrag{Ndof}{\huge Ndof}
\psfrag{Errors: eyepeweqeu}{\LARGE $\norm{(\boldsymbol{\esf}_{\ysf},\esf_{\psf},\boldsymbol{\esf}_{\wsf},\esf_{\qsf},\boldsymbol{\esf}_{\usf})}_{\Omega}$}
\psfrag{Total estimator}{\LARGE $\tilde{\Upsilon}$}
\psfrag{Optimal rate}{\LARGE Ndof$^{-1/3}$}
\scalebox{0.4}{\includegraphics{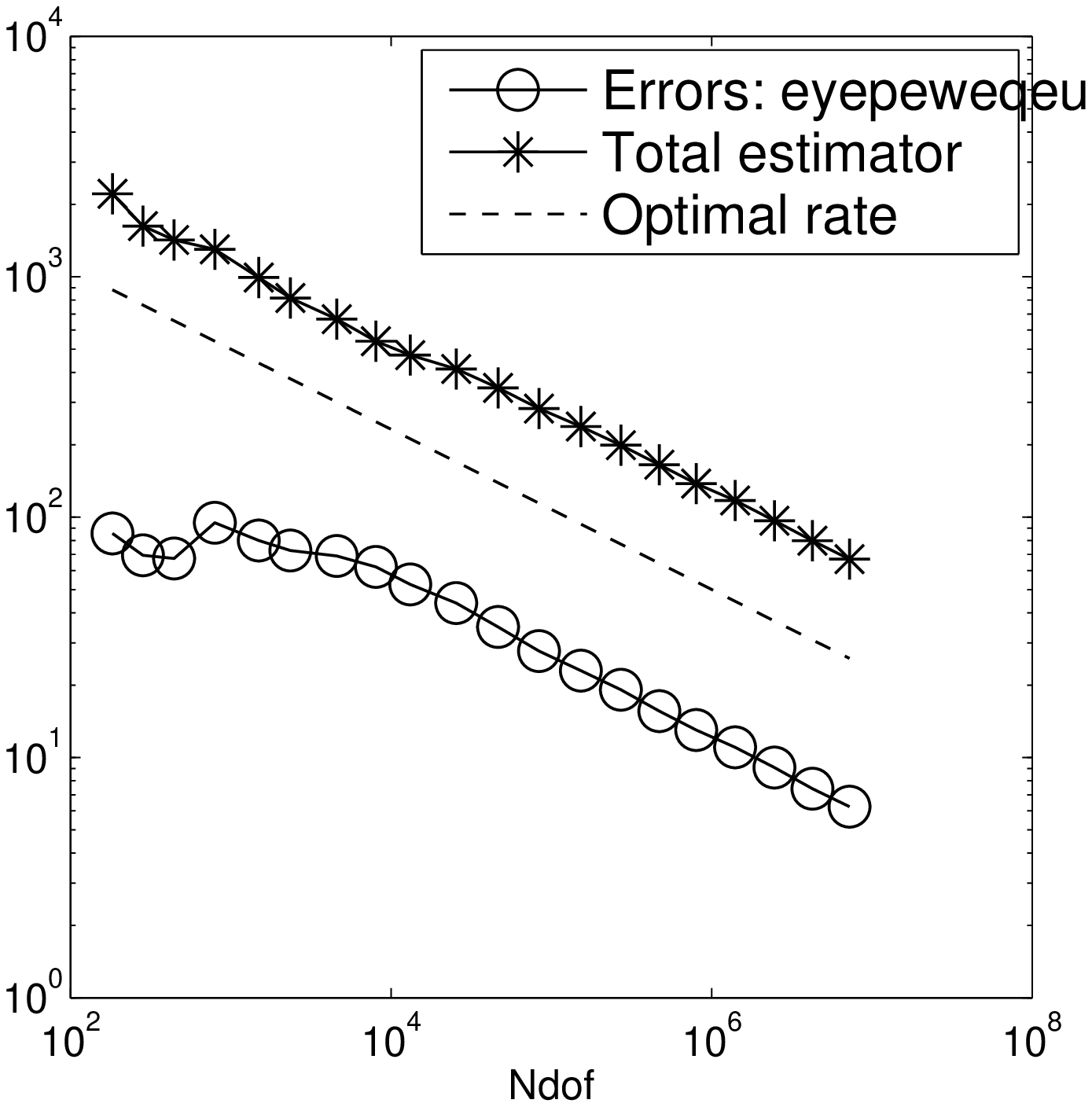}}
\scalebox{0.4}{\includegraphics{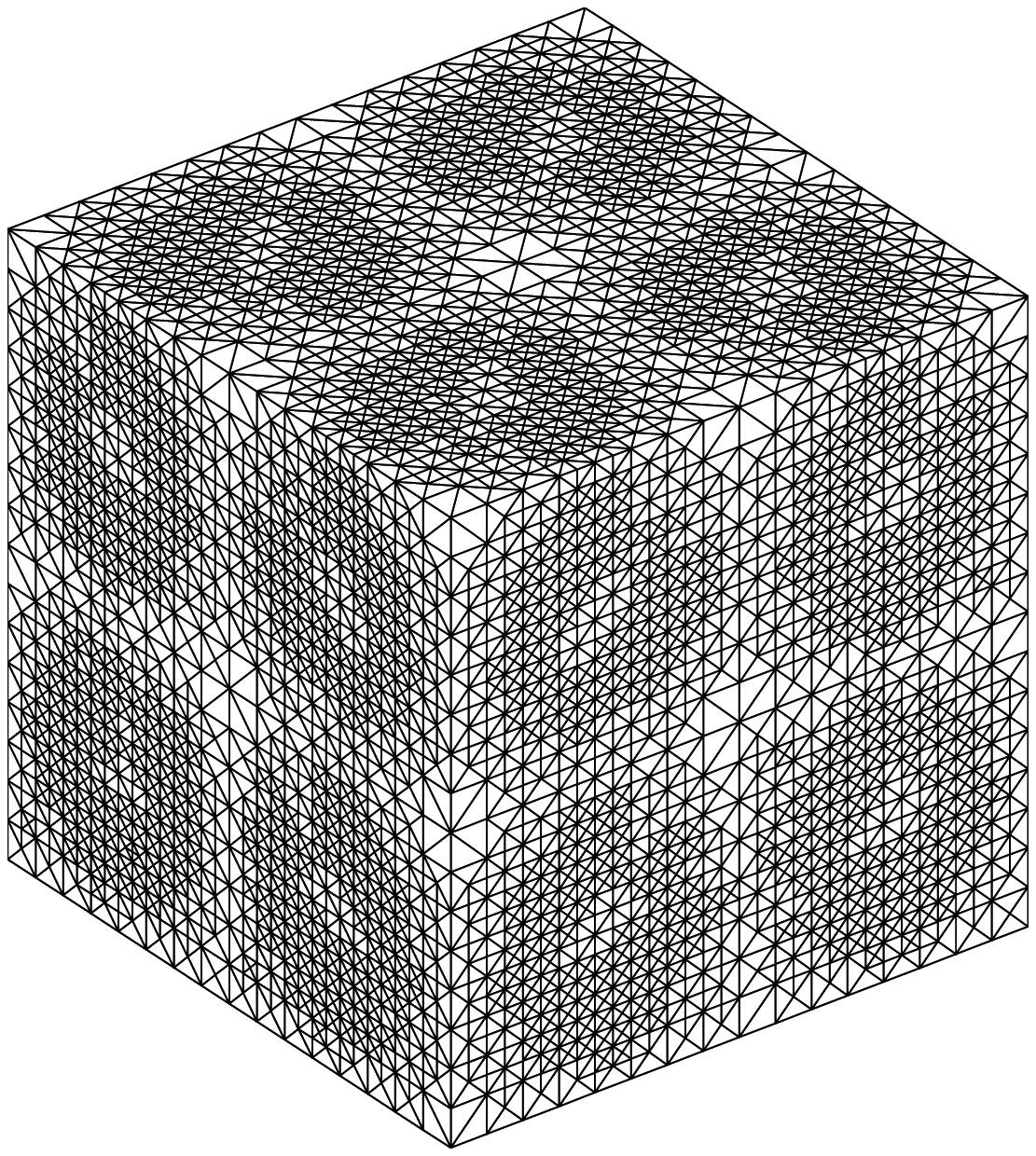}}
\end{center}
\caption{Example \ref{bubble3D}: The error $\norm{(\boldsymbol{\esf}_{\ysf},\esf_{\psf},\boldsymbol{\esf}_{\wsf},\esf_{\qsf},\boldsymbol{\esf}_{\usf})}_{\Omega}$ and estimator $\tilde{\Upsilon}$ (left) and an exterior view of the 19th adaptively refined mesh (right).}
\label{Fig:bubble3D}
\end{figure}

\begin{example}\label{layer3D}
We consider the tetrahedral domain $\Omega=\{(x_1,x_2,x_3):x_1>0,x_2>0,x_3>0,x_1+x_2+x_3<1\}$. We took $\varepsilon=0.01$, $\csf=(1,1,1)$, $\kappa=0$, $\aasf=(0,0,0)$ and $\bbsf=(0.1,0.1,0.1)$. The data $\fsf$ and $\ysf_{\Omega}$ were chosen to be such that
\[
\bar{\ysf}(x_1,x_2,x_3)=\mathbf{curl}\left(x_1x_2^2\chi \left(1-x_1-\frac{\exp(-100x_1)-\exp(-100)}{1-\exp(-100)}\right)\right),
\]
\[
\bar{\psf}(x_1,x_2,x_3)=\left(\cos(2\pi  z)-3/(2\pi^2)\right)/1024,
\]
\[
\bar{\wsf}(x_1,x_2,x_3)=\mathbf{curl}\left(x_1^2x_2\chi \left(1-x_2-\frac{\exp(-100x_2)-\exp(-100)}{1-\exp(-100)}\right)\right),
\]
and
\[
\bar{\qsf}(x_1,x_2,x_3)=\left(\sin(2\pi z)-3/(2\pi)\right)/1024,
\]
where $\chi=x_3^2(1-x_1-x_2-x_3)^2$. The results are shown in Figure \ref{Fig:layer3D}. We observe that, once the mesh has been sufficiently refined, the error $\norm{(\boldsymbol{\esf}_{\ysf},\esf_{\psf},\boldsymbol{\esf}_{\wsf},\esf_{\qsf},\boldsymbol{\esf}_{\usf})}_{\Omega}$ and the estimator $\tilde{\Upsilon}$ decrease at the optimal rate. We also observe that more refinement has been performed in the regions where the solution has boundary layers.
\end{example}

\begin{figure}[!htbp]
\begin{center}
\psfrag{Ndof}{\huge Ndof}
\psfrag{Errors: eyepeweqeu}{\LARGE $\norm{(\boldsymbol{\esf}_{\ysf},\esf_{\psf},\boldsymbol{\esf}_{\wsf},\esf_{\qsf},\boldsymbol{\esf}_{\usf})}_{\Omega}$}
\psfrag{Total estimator}{\LARGE $\tilde{\Upsilon}$}
\psfrag{Optimal rate}{\LARGE Ndof$^{-1/3}$}
\scalebox{0.4}{\includegraphics{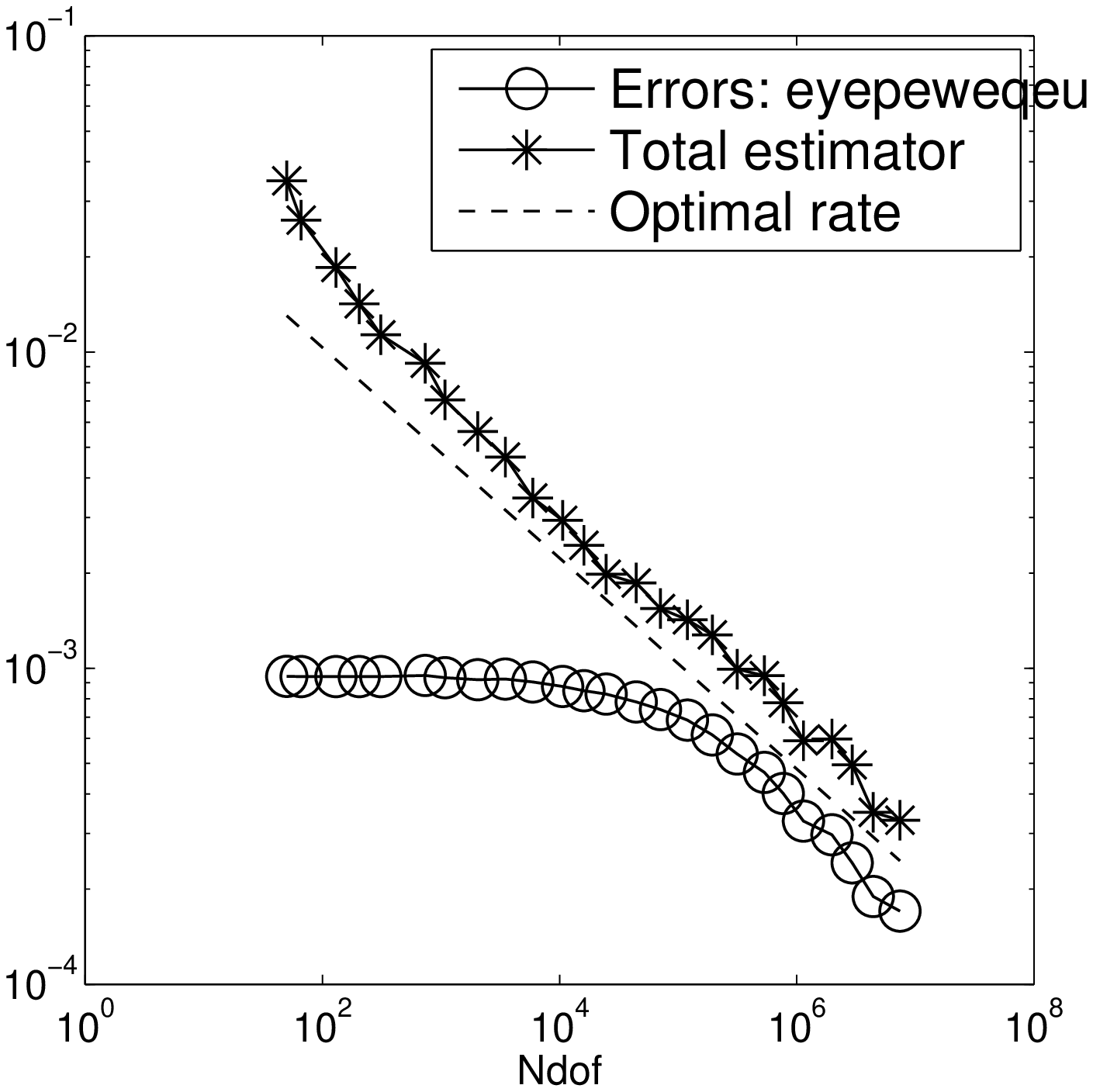}}
\scalebox{0.4}{\includegraphics{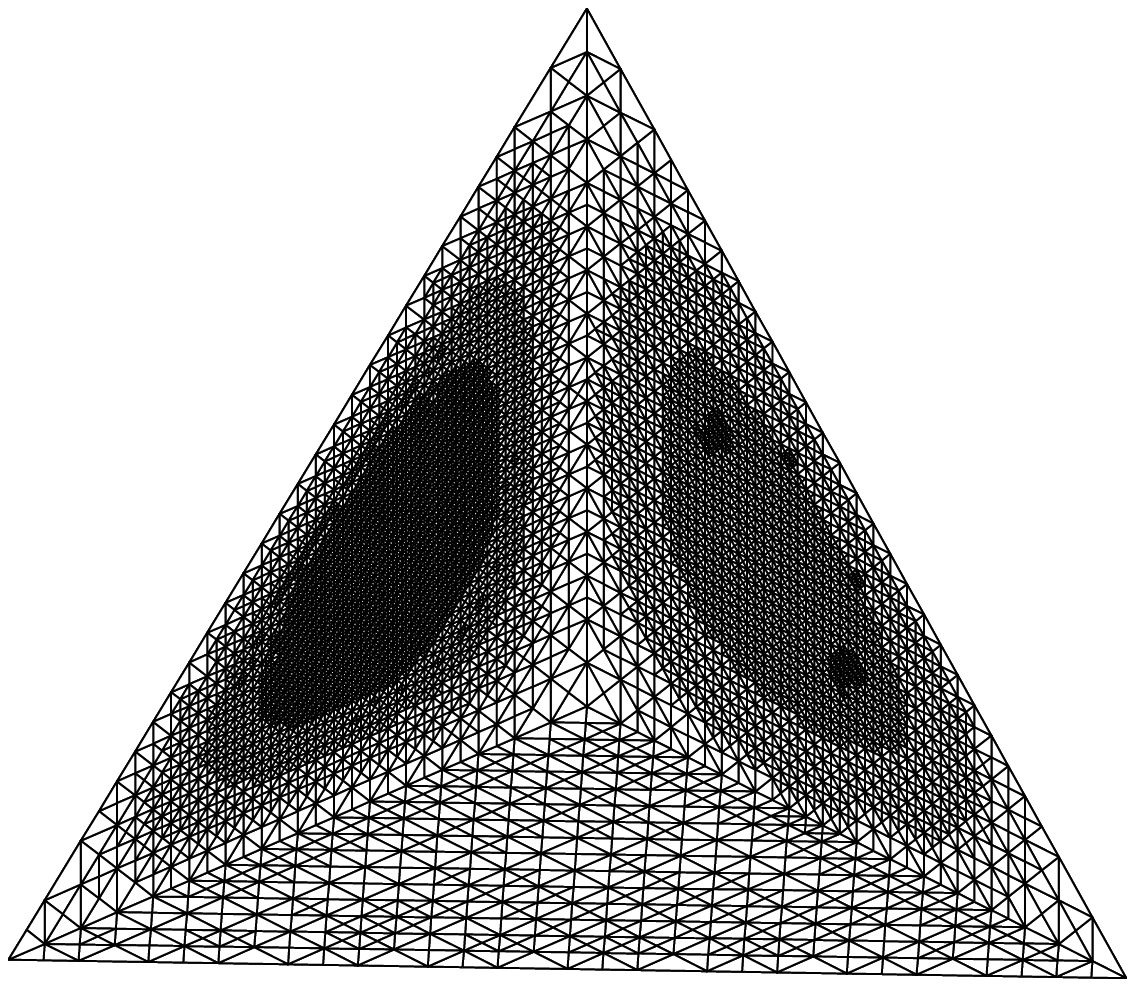}}
\end{center}
\caption{Example \ref{layer3D}: The error $\norm{(\boldsymbol{\esf}_{\ysf},\esf_{\psf},\boldsymbol{\esf}_{\wsf},\esf_{\qsf},\boldsymbol{\esf}_{\usf})}_{\Omega}$ and estimator $\tilde{\Upsilon}$ (left) and an exterior view of the 19th adaptively refined mesh (right).}
\label{Fig:layer3D}
\end{figure}

\bibliographystyle{siamplain}
\bibliography{bibliography}

\end{document}